\documentclass[12pt,a4paper]{amsart}

\usepackage{euscript,amsfonts,amssymb,amsmath,amscd,amsthm,enumerate,hyperref}

\usepackage{comment}

\usepackage{enumitem}
\usepackage{graphicx}

\usepackage{color}
\usepackage{mathtools}

\usepackage[margin=1.1in]{geometry}
\newtheorem{theorem}{Theorem} 
\newtheorem*{theorem*}{Theorem}
\newtheorem{lemma}[theorem]{Lemma}
\newtheorem{definition}[theorem]{Definition}
\newtheorem{proposition}[theorem]{Proposition}
\newtheorem{conjecture}[theorem]{Conjecture}
\newtheorem{corollary}[theorem]{Corollary}

\theoremstyle{remark}
\newtheorem{remark}[theorem]{Remark}

\numberwithin{equation}{section} \numberwithin{theorem}{section}

\newcommand{\la}{\lambda}

\newcommand{\GT}{\mathbb{GT}}

\newcommand{\N}{\mathbb N}

\newcommand{\R}{\mathbb R}

\newcommand{\D}{\mathcal D}

\newcommand{\pa}{\partial}

\newcommand{\ii}{{\mathbf i}}

\newcommand{\E}{\mathbb E}

\renewcommand{\c}{\mathbf c}

\renewcommand{\d}{\mathbf d}

\newcommand{\eps}{\varepsilon}

\newcommand{\Cov}{\mathrm{Cov}}

\renewcommand{\c}{\mathbf c}

\renewcommand{\d}{\mathbf d}

\newcommand{\p}{\mathfrak p}
\newcommand{\cov}{\mathfrak{cov}}

\newcommand{\Sym}{\mathbf{Sym}}

\newcommand{\x}{\mathbf x}

\newcommand{\K}{\mathcal K}

\newcommand{\pn}{p_{\text{north}}}
\newcommand{\ps}{p_{\text{south}}}
\newcommand{\pe}{p_{\text{east}}}
\newcommand{\pw}{p_{\text{west}}}

\title[Fourier transform on $U(N)$ as $N\to\infty$.]
{Fourier transform on high--dimensional unitary groups with applications to random
tilings}

\author{Alexey Bufetov}

\address[Alexey Bufetov]{Department of Mathematics, Massachusetts Institute of Technology, Cambridge, MA, USA. E-mail: alexey.bufetov@gmail.com}

\author{Vadim Gorin}

\address[Vadim Gorin]{Department of Mathematics, Massachusetts Institute of Technology, Cambridge, MA, USA, and
 Institute for Information Transmission Problems of Russian Academy of Sciences, Moscow, Russia. E-mail: vadicgor@gmail.com}

\begin{document}

\begin{abstract}
 A combination of direct and inverse Fourier transforms on the unitary group $U(N)$
 identifies normalized characters with probability
 measures on $N$--tuples of integers. We develop the $N\to\infty$ version of this correspondence
 by matching the asymptotics of partial derivatives at the identity of logarithm of characters
 with Law of Large Numbers and Central Limit Theorem for global behavior of corresponding
 random $N$--tuples.

 As one application we study fluctuations of the height function of random domino
 and lozenge tilings of a rich class of domains. In this direction we prove the
 Kenyon--Okounkov's conjecture (which predicts asymptotic Gaussianity and exact form of the
 covariance) for a family of non-simply
 connected polygons.

 Another application is a central limit theorem for the $U(N)$ quantum random walk with
 random initial data.
\end{abstract}

\maketitle

\section{Introduction}

\subsection{General theorem}

Due to E.~Cartan and H.~Weyl (see  \cite{Weyl-book}), the characters of irreducible representations
of the unitary group $U(N)$ are parameterized by signatures (i.e.\ highest weights), which are
$N$--tuples of integers
$$
\lambda=(\lambda_1\ge\lambda_2\ge\dots\ge\lambda_N),
$$
and the value of the character on a unitary matrix with eigenvalues $x_1,\dots,x_N$
is given by the Schur symmetric (Laurent) polynomial
$$
 s_\lambda(x_1,\dots,x_N)=\frac{\det\left[x_i^{\lambda_j+N-j}\right]_{i,j=1}^N}{\prod\limits_{1\le i<j\le N} (x_i-x_j)}.
$$
When $N=1$, $U(1)$ is the unit circle, and the characters are monomials $x^k$, $k\in\mathbb Z$. In
this case the Fourier transform identifies functions on the unit circle on one side with functions
on $\mathbb Z$ on the other side. Its probabilistic version starts from a probability measure
$\mathbb P(k)$ on $\mathbb Z$ and maps it to the generating function $\sum_{k\in\mathbb Z} \mathbb
P(k) z^k$. If we set $z=\exp(\ii t)$, then we arrive at the conventional notion of the
characteristic function. The \emph{uniqueness theorem} for characteristic functions claims that
they are in one-to-one correspondence with probability measures, i.e.\ if characteristic functions
for two measures coincide, then so do the measures themselves.

\smallskip

For general $N$ we start from a probability distribution $\mathbb P$ on
$\lambda=(\lambda_1\ge\dots\lambda_N)$ and identify it with a generating function
\begin{equation}
\label{eq_SGF_intro}
 S_{\mathbb P}(x_1,\dots,x_N) =\sum_{\lambda} \mathbb P(\lambda)
 \frac{s_\lambda(x_1,\dots,x_N)}{s_\lambda(1^N)}, \qquad
 1^N=(\underbrace{1,\dots,1}_N).
\end{equation}
which was called \emph{Schur generating function} (SGF) of $\mathbb P$ in \cite{BG,BG_CLT}. It is
again straightforward to show that the correspondence is bijective: different measures correspond
to different functions. This can be viewed as a version of the Fourier transform on $U(N)$.

\medskip

The main topic of our article is to understand whether this correspondence between measures on
signatures and (central, positive--definite) functions on $U(N)$ survives in $N\to\infty$ limit. In
other words, we would like to develop the \emph{$N=\infty$ Fourier analysis on $U(N)$}.

At this point we need to choose topologies by specifying what we mean by $N\to\infty$ version of
signatures and by $N\to\infty$ version of the functions on $U(N)$. The answer inevitably depends on
this choice.

We are driven by applications in $2d$ statistical mechanics and asymptotic
representation theory, and for that we concentrate on the \emph{empirical
distribution} of $\{\lambda_i\}_{i=1}^N$. In more detail, we deal with pairings of
the empirical distribution with polynomial test functions or, equivalently, with
\emph{moments} of random signatures, i.e.\ random variables
\begin{equation}
\label{eq_intro_moments}
p_k^N= \sum_{i=1}^N \left(\frac{\lambda_i+N-i}{N}\right)^k, \quad k=1,2,\dots.
\end{equation}
We would like to know the Law of Large Numbers and Central Limit Theorem for the moments: the
deterministic limit $\lim_{N\to\infty} \frac{1}{N} p_k^N$, and asymptotic Gaussianity and
covariance for the centered versions $\lim_{N\to\infty} (p_k^N -\E p_k^N)$.

On the side of functions on $U(N)$, we are interested in the \emph{germ} of a
function near the identity element of the group. From the technical point of view,
it turns out to be convenient to pass to the logarithm of SGF and to deal with
partial derivatives at $x_1=x_2=\dots=1$ of
\begin{equation}
\ln( S_{\mathbb P} (x_1,\dots,x_N) ) ,\quad N\to\infty.
\end{equation}
One interpretation of our previous work \cite{BG}, \cite{BG_CLT} is that germ \emph{contains
complete information} about the LLN and CLT. In other words, we constructed a map from germs to LLN
and CLT for the moments \eqref{eq_intro_moments}.


The key new results of this article \emph{invert} the theorems of \cite{BG},
\cite{BG_CLT}: we show that if the Law of Large Numbers and Central Limit Theorem
for monomials \eqref{eq_intro_moments} hold, then partial derivatives of
$\ln(\mathcal S_\mathbb P)$ at $(1^N)$ converge. The limiting germ is linked to the
deterministic limits in LLN and covariances in CLT by explicit polynomial formulas.
Theorem \ref{Theorem_main} states our main result.

Therefore, we claim that the correspondence between LLN, CLT for moments and asymptotic germ is an
$N\to\infty$ version of the Fourier transform on $U(N)$. We now have a \emph{bijective} map between
asymptotic information on probability measures on signatures and asymptotic information on
characters of $U(N)$ encoding them.
\smallskip

We also go further and provide two multidimensional generalizations, in which a single $N$--tuple
$\lambda$ is replaced by a collection of those. Conceptually the results remain the same (LLN, CLT
hold if and only if partial derivatives of the logarithm of SGF converge), and we refer to Theorems
\ref{Theorem_main_multi}, \ref{Theorem_CLT_multi_ext} for the details.

\bigskip

Let us explain the consequences of our theorems. Beforehand, in order to apply the
results of \cite{BG}, \cite{BG_CLT} to get LLN and CLT one needed precise asymptotic
information about Schur generating functions (e.g.\ in \cite{BG}, \cite{BG_CLT} we
relied on \cite{GP} for some applications). Therefore, the asymptotic theorems were
restricted to the class of systems, for which SGF was computable in somewhat
explicit form. This was a serious obstacle for the wide applicability of the
theorems, since in many examples of interest almost no information on SGF is
available.

With the new approach, we can start from a system where LLN and CLT are known by any
other method  and then perform operations on it, which act ``nicely'' on Schur
generating functions, and conclude that the the resulting system will again have LLN
and CLT. We further demonstrate how this works in two applications.

For the first one, we observe that setting a part of variables in a Schur generating
function to $1$ is closely related to combinatorics of uniformly random lozenge
tilings. Thus, when we combine our results with the CLT along a section obtained in
\cite{BGG} via dicrete loop equations, we obtain a 2d Central Limit Theorem for
random tilings leading to the Gaussian Free Field. This is explained in the next
section and detailed in Section \ref{Section_tilings}. A combination of
multiplication by simple factors and setting variables gives rise to domino tilings
and more general dimers on rail--yard graphs, cf.\ \cite{BG_CLT}, \cite{BK},
\cite{BL}.

For the second application, we note that multiplication of Schur generating
functions have a representation--theoretic interpretation as time evolution of
\emph{quantum random walk}, cf.\ \cite{Bi3,Bi4}. Thus, we can produce the LLN and
CLT for the latter with random initial conditions, as detailed in Appendix
\ref{Section_quantum_CLT}. Given our theorems, the nature of this example becomes
close to the textbook proof of the conventional CLT for sums of independent random
variables: there one also manipulates with products of characteristic functions in
order to observe the universal Gaussian limit.

\medskip

In our applications we go well beyond the cases, which were considered exactly
solvable or integrable before. For instance, for random tilings we are able to
investigate non-simply connected planar domains, which were never accessible before
due to lack of any explicit formulas. We do not find such formulas as well, instead
we show that they are not necessary  and the Central Limit Theorem for global
fluctuations can be obtained without them by operating with Schur generating
functions.

\subsection{Application to random tilings}

We now present one important application of our general approach. It deals with uniformly random
lozenge and domino tilings of planar domains. We would like to take a planar domain drawn on
triangular (or square) grid and tile it with lozenges (or dominos), which are unions of adjacent
triangles (or squares) of the grid, see Figure \ref{Figure_tilings} for a basic illustration and
Figures \ref{Figure_hex_hole}, \ref{Figure_domino_domain}, \ref{Figure_others} for more elaborate
examples.

\begin{figure}[t]
\begin{center}
 {\scalebox{0.6}{\includegraphics{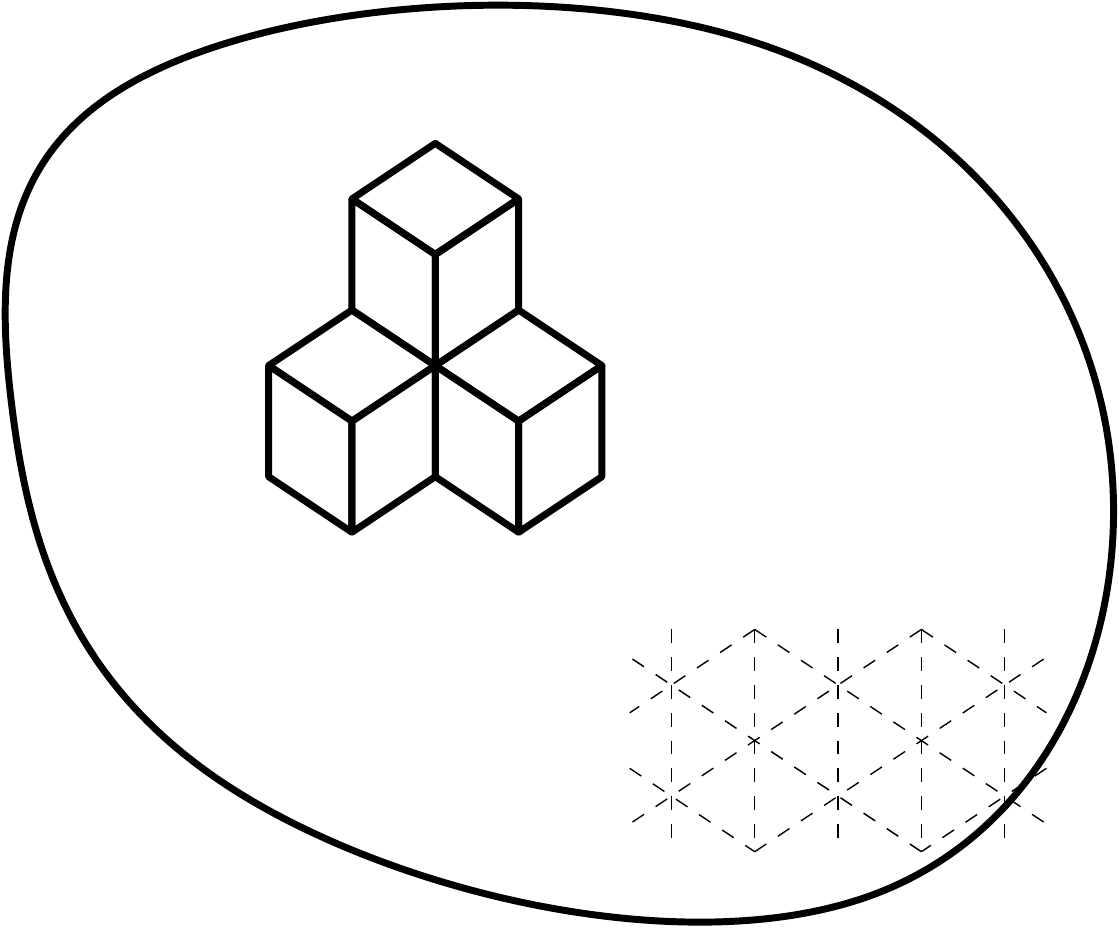}}} \qquad
 {\scalebox{0.7}{\includegraphics{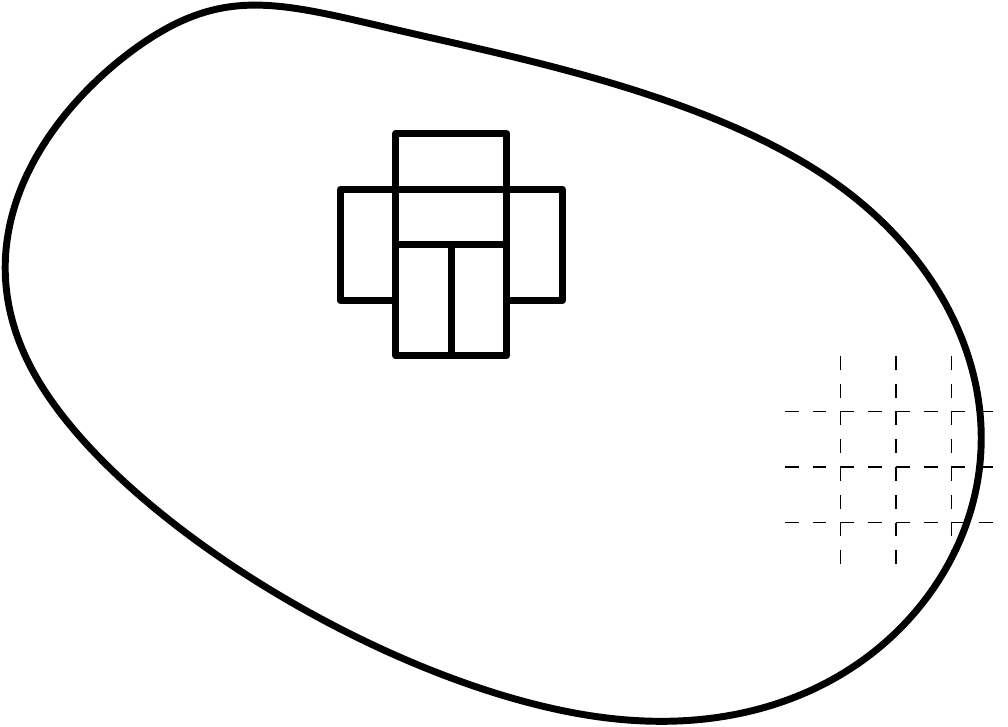}}}
 \caption{Lozenge and domino tilings on triangular and square grids.
 \label{Figure_tilings}}
\end{center}
\end{figure}

Each domain has finitely many tilings, and therefore, we can choose one uniformly at
random and study its asymptotic properties as the mesh size of the grid goes to
zero. Lozenge and domino tilings are two particular cases of the \emph{dimer model}
and we refer to \cite{Kenyon_notes} for a review.

\bigskip

A form of the Law of Large Numbers for tilings of very general domains was developed
in \cite{CKP} (see also \cite{KOS}) through a variational principle. Following
\cite{Thurston}, the article identifies a tiling with its \emph{height function},
which is the most transparent for lozenge tilings, that can be viewed as a
projection in $(1,1,1)$ direction of a stepped surface, which, in turn, can be
treated as a graph of a function. For domino tilings the construction of the height
function is slightly more complicated, cf.\ \cite{Kenyon_notes}. For both types of
tilings \cite{CKP} proves\footnote{Formally, \cite{CKP} deals only with
simply--connected domains, while we study here some regions with holes. However, the
proofs of \cite{CKP} probably extend to the multiple-connected domains.} that random
height functions concentrate as the mesh size goes to $0$ near a deterministic
\emph{limit shape}, which is found as a maximizer of an integral functional of the
gradient of the height function. One intriguing feature of the limit shapes is that
they have facets, or \emph{frozen regions}, where only one type of lozenges (or
dominos) survives in the limit, see the right panel of Figure \ref{Figure_hex_hole}
for an example.

 \cite{KO_Burgers} found that if one encodes the gradient of the height function as a complex
 number (``complex slope''), then the Euler--Lagrange equation for the limit shape turns into the
 complex Burgers equation (see Section \ref{Section_tilings} for some details), and its solution is
 given in terms of complex--analytic functions. For polygonal domains, the limit shape is in
 fact given in terms of algebraic functions, which are found in an algorithmic way.

 In the same paper Kenyon and Okounkov formulated a conjecture: the fluctuations of the height function
 around the limit shape should be described by the 2d \emph{Gaussian Free Field} in the conformal
 structure given by the complex slope.

 In its largest generality this conjecture remains open at this time. However, several approaches
 arose to solve some of its particular cases. \cite{Kenyon_domino}, \cite{Kenyon_height} developed
 a method for analyzing a particular class of domains with \emph{no} frozen regions through
 determinantal structure of correlation functions and by exploiting the convergence of discrete
 harmonic functions to their continuous counterparts. New ideas were added to this approach more recently
 in \cite{Li}, in \cite{BLR1}, \cite{BLR2}, and in \cite{Russ}, yet any domains with frozen regions (in
 particular, the most aesthetically pleasant case of polygonal domains) remain out of reach.

 Another set of methods, which work for specific polygons (somewhat complementing the aforementioned
 results) was developed in the framework of Integrable Probability by finding and using explicit formulas:
 double contour integral representations of the correlation kernel \cite{BorFer},
 \cite{Petrov_GFF}, \cite{CJY}; orthogonal polynomials \cite{Duits}; differential and difference
 operators acting on explicit inhomogeneous partition functions \cite{BG_CLT}, \cite{BK},
 \cite{Ahn}.

 \bigskip

 Our general theorems provide a completely new approach to these problems. Here is the first new result
 along the lines of this approach.

 \begin{theorem} \label{Theorem_GFF_intro}
   The Kenyon--Okounkov conjecture (on the convergence to the Gaussian Free Field in the conformal structure
   of the complex slope) holds for lozenge tilings of a hexagon with a hole and for
   domino tilings of an Aztec rectangle with arbitrary many holes along a single axis.
 \end{theorem}
 We refer to Figures \ref{Figure_hex_hole}, \ref{Figure_domino_domain} for the domains that we
 analyze and to Theorems \ref{Theorem_holey_hex}, \ref{Theorem_holey_Aztec} for detailed formulations and proofs.

 We did not want to overload this article with a complete list of examples that can be analyzed,
 rather sticking to the simplest cases in Theorem \ref{Theorem_GFF_intro}. Section
 \ref{Section_further_dimers} outlines other boundary conditions as well as other dimer models,
 which we believe to be analyzable by our approach; a detailed treatment of these cases will appear
 elsewhere.

 The domains in Theorem \ref{Theorem_GFF_intro} are \emph{not} simply connected, and they do not
 posses the exact solvability as previously analyzed polygonal domains. There are no
 known formulas for the correlation functions (neither as contour integrals, nor through orthogonal
 polynomials) and no information is available about the partition functions. Even the answer ---
 covariance of the Gaussian Free Field --- is no longer explicit: for the one hole case it involves
 elliptic integrals, and for larger number of holes one needs periods of higher genus Riemann surfaces.

 As far as we know, Theorem \ref{Theorem_GFF_intro} is the first instance of the proof of Kenyon-Okounkov's conjecture
 for multiply connected planar domains.\footnote{The
other case with non-trivial topology is dimers on the torus
 (there will be no frozen regions in this case, and the complex structure is the standard one, as the limit shape is flat) studied in \cite{Dubedat}. See also \cite{BLR3} for the new progress.}

\subsection{Methodology and organization}

Let us outline the key idea underlying the proof of the correspondence between LLN, CLT and germs
of Schur generating functions.

We start similarly to \cite{BG}, \cite{BG_CLT} (see also \cite{BorBuf}, \cite{BBO}) by applying
differential operators diagonalized by Schur functions to the Schur generating function of the
measure. This leads to polynomial formulas relating the expectations of the products of the moments
\eqref{eq_intro_moments} with partial derivatives at the unity of $\ln(S_{\mathbb P})$ for
finite values of $N$.

Our first surprising observation is that the \emph{highest order partial
derivatives} part of these relations is \emph{linear} and of \emph{square} format,
i.e.\ with the number of variables on the partial derivatives side equal to the
number of variables on the moments side. Thus, one can treat this part as a system
of linear equations on partial derivatives and hope to solve it.

Further investigation reveals that the matrix of this system splits into the sum of
two: the first one is essentially \emph{triangular}, while the second one decays as
$N\to\infty$. We do not know any a priori reasons to expect this triangularity.
Inverting the triangular part we get expressions for partial derivatives in terms of
centered moments (in fact, we rather work with cumulants), which behave in a
controlled way as $N\to\infty$.

\medskip

Section \ref{Section one_dim} provides a full proof of the correspondence by
developing this idea. Section \ref{Section_multi} makes a next step by extending the
correspondence to the multi--dimensional setting of several random signatures. In
Section \ref{Section_tilings} we first outline the strategy for applying our
approach to the study of random tilings, then give precise formulations for both
lozenge and domino tilings and provide detailed proofs. Finally, Appendix
\ref{Section_quantum_CLT} explains how our approach can be used to get the Central
Limit Theorem for quantum random walk on $U(N)$.

\subsection{Related articles} The interplay between functions (or measures) on a group and
functions (or measures) on labels of its irreducible representations in high--dimensional
\emph{asymptotic} problems of probability and representation theory was previously discussed in
several groups of articles.

The oldest example originates in the study of finite factor representations and
characters for the infinite symmetric group $S(\infty)$ and the
infinite--dimensional unitary group $U(\infty)$. Vershik and Kerov \cite{VK_S},
\cite{VK_U} introduced a way to think about these objects as limits for random Young
diagrams, with complete details and extensions first appearing in \cite{KOO},
\cite{OO}. The topology in these articles is very different from ours, rather than
looking at all coordinates of the signature simultaneously as in
\eqref{eq_intro_moments}, they deal with the first few rows and columns of the
corresponding Young diagrams. This is dictated by the infinite groups $S(\infty)$
and $U(\infty)$ themselves, see \cite[Section 1.4]{Olsh} for the Fourier point of
view in this setting. It would be interesting, if we could identify our germs of SGF
with functions on some natural infinite--dimensional groups or algebras, similarly
to harmonic analysis of \cite{KOV}, \cite{BO_U_infty}; so far we have no progress in
this direction.

\cite{Biane1}, \cite{Biane2} is closer to our studies, as in these papers Biane
explained how limit shape theorems for random Young diagrams are equivalent to the
approximate factorization property for the characters of symmetric groups. The
connections between values of characters and fluctuations were further explored in
\cite{Kerov}, \cite{IO} and in \cite{Sniady}. More recent papers \cite{DF},
\cite{DS} aimed at generalizing these results to the deformation related to Jack
polynomials. Some of the asymptotic structures related to unitary groups admit
degenerations to symmetric groups (see e.g.\ \cite{BO}), and it is to be understood,
whether our results linking LLN, CLT for the moments to germs of SGF can be
degenerated into meaningful statements for the $S(n)$ as $n\to\infty$.

Another direction is Wiengarten calculus for the integrals over unitary groups and
its application to asymptotic problems e.g.\ in \cite{Collins}, \cite{CNS},
\cite{MN}.  In particular, it was used in relation to the degeneration of the
character theory for the unitary groups (known as the semi-classical limit) to
invariant random matrix ensembles.

In a slightly different direction, the mixing time for random walks on large groups
was often studied using character theory of this group, see e.g.\ \cite{Diaconis}.
For instance, the card shuffling problems can be analyzed with the use of characters
of symmetric groups. In such work, probability measures usually live on a group,
while the corresponding functions are defined on the labels of irreducible
representations. In our approach these two roles are exchanged.

\subsection{Acknowledgements.} We would like to thank G.~Borot, R.~Kenyon, and G.~Olshanski for fruitful
discussions. V.G.~is partially supported by the NSF grants DMS-1407562, DMS-1664619,
by the Sloan Research Fellowship. Both authors were partially supported by The
Foundation Sciences Math\'{e}matiques de Paris.

\section{One--dimensional theorem}

\label{Section one_dim}

\subsection{Formulation}

Let $\GT_N$ denote the set of all $N$--tuples $\lambda_1\ge \lambda_2\ge
\dots\ge\lambda_N$ of integers. For a probability measure ${\tau}$ on $\GT_N$ let
$S_{\tau}$ denote its Schur generating function, i.e.\
\begin{equation}
\label{eq_SGF_def}
 S_{\tau}(x_1,\dots,x_N)=\sum_{\lambda\in \GT_N} {\tau}(\lambda)
 \frac{s_\lambda(x_1,\dots,x_N)}{s_\lambda(1^N)},\qquad 1^N=(\underbrace{1,\dots,1}_N),
\end{equation}
where $s_\lambda$ is Schur symmetric (Laurent) polynomial.

\begin{definition} We call a probability measure $\tau$ on $\GT_N$ \emph{smooth}, if
there exists $r>1$ such that
\begin{equation}
\label{eq_smoothness}
 \sum_{\lambda\in \GT_N} \tau(\lambda) \cdot r^{|\lambda_1|+|\lambda_2|+\dots+|\lambda_N|}
 <\infty.
\end{equation}
\end{definition}
\begin{lemma}
\label{Lemma_smooth_SGF}
 If $\tau$ is a smooth measure on $\GT_N$, then its Schur--generating function $S_\tau$ is
 uniformly convergent in a neighborhood of $1^N$.
\end{lemma}
\begin{proof}
 The combinatorial formula for Schur functions (see  e.g.\ \cite[Section 5, Chapter
 I]{Mac}) implies that for any $t>0$,
 \begin{equation}
 \label{eq_Schur_bound}
  \sup_{\begin{smallmatrix}(x_1,\dots,x_N)\in\mathbb C^N,\\ |x_1|=\dots=|x_N|=t\end{smallmatrix}}
  |s_\lambda(x_1,\dots,x_N)|=s_{\lambda}(t^N)=t^{\lambda_1+\lambda_2+\dots+\lambda_N}
  s_\lambda(1^N).
 \end{equation}
 Combining \eqref{eq_Schur_bound} with \eqref{eq_smoothness} we conclude that
 \eqref{eq_SGF_def} is uniformly convergent in the annulus $t^{-1}<|x_i|<t$,
 $i=1,\dots,N$ for any $1<t<r$.
\end{proof}

Now for $N=1,2,\dots$, let $\rho_N$ be a smooth probability measure on $\GT_N$. Let
$\lambda$ be $\rho_N$--distributed element of $\GT_N$. Define
$\{p_k^N\}_{k=1,2,\dots}$ to be a countable collection of random variables via
\begin{equation}
\label{eq_random_power_sum}
 p_k^N= \sum_{k=1}^N \left(\frac{\lambda_i+N-i}{N}\right)^k.
\end{equation}
 Let us recall the definition of a joint cumulant. For $r=1,2,\dots$, let
$\Theta_{r}$ be the collection of all unordered partitions of $\{1, \cdots, r\}$,
i.e.\
\begin{equation}
\label{eq_Set_partition} \Theta_{r} = \left\{ \left\{ U_1, \cdots, U_t \right\}
\left| t\in \mathbb{Z}_+, \bigcup_{i=1}^t U_i = \{1, \cdots, r\}, U_i \bigcap U_j =
\emptyset, U_i \neq \emptyset, \forall 1\leq i \leq j \leq t \right. \right\} .
\end{equation}
\begin{definition}
 Take a collection of random variables $(\xi_k)_{k=1,2,\dots}$. For each $r=1,2,\dots,$,
 and $r$--tuple of positive integers $(k_1,\dots, k_r)$, the $r$th order cumulant
 of $\xi_1,\xi_2,\dots$ is defined through
\begin{equation}  \label{eq_cum_def}
\kappa_{k_1,\dots,k_r}(\xi_k,\, k=1,2,\dots):=\sum_{\left\{ U_1, \cdots, U_t
\right\} \in \Theta_r} (-1)^{t-1} (t-1)! \prod_{i=1}^t \E\left[ \prod_{j \in U_i}
\xi_{k_j} \right].
\end{equation}
\end{definition}
Note that the joint cumulants uniquely define joint moments of $\xi_1,\xi_2,\dots$
and vice versa; in particular, the first order cumulant is the mean, while the
second order cumulant is the covariance. $\xi_1,\xi_2,\dots$ is a Gaussian vector if
and only if all the $r$th cumulants for each $r>2$ vanish, see e.g.\ \cite[Section
3.1, 3.2]{PT}.

\begin{definition} \label{Def_CLT}
 We say that smooth measures $\rho_N$ satisfy a CLT as $N\to\infty$, if there exists
 two countable collection of numbers $\mathfrak p(k)$, $\mathfrak {cov}(k,m)$,
 $k,m=1,2,\dots$ such that
\begin{enumerate}[label=(\Alph*)]
\item \label{ass_1} For each $k=1,2,\dots$,
$$
 \lim\limits_{N\to\infty} \frac{1}{N} \E [p_k^N] = {\mathfrak p}(k),
$$
\item \label{ass_2} For each $k,m=1,2,\dots$,
$$
 \lim\limits_{N\to\infty}\Bigl( \E [p_k^N p_m^N] -\E[p_k^N]\E [p_m^N] \Bigr)= \mathfrak {cov}(k,m),
$$
\item \label{ass_3} For each $r>2$ and any collection of positive integers
$k_1,\dots, k_r$
$$\lim_{N\to\infty} \kappa_{k_1,\dots,k_r}\bigl(  p_k^N,\, k=1,2,\dots\bigr)=0.
$$
\end{enumerate}

\end{definition}

\begin{definition}
\label{Def_CLT_appropriate}
 We say that smooth measures $\rho_N$ are CLT--appropriate as $N\to\infty$, if there
 exist two countable collections of numbers $\c_k$, $\d_{k,m}$, $k,m=1,2,\dots$,
 such that
\begin{enumerate}
\item \label{exp_1} For each $k=1,2,\dots$:
$$\lim\limits_{N\to\infty} \frac{1}{N} \left(\frac{\partial}{\partial x_i}\right)^k \ln( S_{\rho_N})
\Bigl|_{x_1=\dots=x_N=1}= \c_k,$$
\item For each $k,m=1,2,\dots$, with $k\ne m$:\label{exp_2}
$$\lim\limits_{N\to\infty} \left(\frac{\partial}{\partial x_i}\right)^k \left(\frac{\partial}{\partial x_j}\right)^m  \ln( S_{\rho_N})
\Bigl|_{x_1=\dots=x_N=1}= \d_{k,m},
$$
\item \label{exp_3} For each set of indices $\{i_1,\dots,i_s\}$ with at least
three distinct elements:
$$\lim\limits_{N\to\infty} \frac{\partial}{\partial x_{i_1}}
\dots \frac{\partial}{\partial x_{i_s}}  \ln( S_{\rho_N}) \Bigl|_{x_1=\dots=x_N=1}=
0.$$
\end{enumerate}
 \end{definition}

\begin{theorem} \label{Theorem_main}
Smooth measures $\rho_N$ satisfy CLT as $N\to\infty$ if and only if
they are CLT--appropriate.

The numbers $\p(k)$, $\cov(k,m)$, $k,m=1,\dots,$ are polynomials in $\c_k$,
$\d_{k,m}$, $k,m=1,\dots$, and can be computed through the following formulas:

\begin{equation}
\label{eq_LLN_1_formula} \mathfrak p(k) = [z^{-1}] \frac{1}{(k+1)(1+z)} \left( (1+z)
\left(\frac1z  +  \sum_{a=1}^\infty \frac{\c_a z^{a-1} }{(a-1)!}\right)
\right)^{k+1},
\end{equation}
\begin{multline}
\label{eq_CLT_1_formula} \mathfrak {cov}(k,m)= [z^{-1} w^{-1}]
 \left( (1+z)
\left(\frac1z  +  \sum_{a=1}^\infty \frac{\c_a z^{a-1} }{(a-1)!}\right)\right)^{k}
\, \left( (1+w)
\left(\frac1w  +  \sum_{a=1}^\infty \frac{\c_a w^{a-1} }{(a-1)!}\right) \right)^{m} \\
\times \left(  \left(\sum_{a=0}^{\infty} \frac{z^a}{w^{1+a}}\right)^2
+\sum_{a,b=1}^{\infty} \frac{\d_{a,b}}{(a-1)! (b-1)!} z^{a-1} w^{b-1} \right),
\end{multline}
where $[z^{-1}]$ and $[z^{-1} w^{-1}]$ stay for the coefficients of $z^{-1}$ and
$z^{-1}w^{-1}$, respectively, in the Laurent power series given afterwards.
\end{theorem}
\begin{remark}
 If the series in \eqref{eq_LLN_1_formula}, \eqref{eq_CLT_1_formula} are convergent
 for $z$ and $w$ in a neighborhood of $0$, then the formulas can be converted into
 contour integrals (as residues at $0$). In particular, in \eqref{eq_CLT_1_formula} we then
 need to assume $|z|<|w|$ leading to
$$
\left(\sum_{a=0}^{\infty} \frac{z^a}{w^{1+a}}\right)^2=\frac{1}{(z-w)^2}.
$$
This is the form used in \cite{BG_CLT}.
\end{remark}

\begin{lemma} \label{Lemma_as_invert}
 The formulas \eqref{eq_LLN_1_formula}, \eqref{eq_CLT_1_formula} can be inverted as
 follows:
\begin{equation}
\label{eq_LLN_inversion_formula} \sum_{a=1}^\infty \frac{\c_a z^{a-1} }{(a-1)!}=
\frac{1}{1+z}\left(\sum_{k=0}^{\infty} \frac{\mathfrak
p(k)}{u^{k+1}}\right)^{(-1)}\Biggr|_{u=\ln(1+z)} -\frac{1}{z}
\end{equation}
where $(\cdot)^{(-1)}$ is the functional inversion of the power series,
$\bigr|_{u=\ln(1+z)}$ means that we change the variables \emph{after} inversion, and
we use the convention $\mathfrak p(0)=1$.
\begin{multline}
\label{eq_CLT_inversion_formula} \sum_{a,b=1}^{\infty} \frac{\d_{a,b}}{(a-1)!(b-1)!}
z^{a-1} w^{b-1}= \frac{f'(z)}{f(z)} \frac{f'(w)}{f(w)} \\ \times \left(
\sum_{k,m=0}^{\infty} \frac{\mathfrak {cov}(k,m) }{f(z)^k f(w)^m} - [u^{-1}
v^{-1}]\left( \frac{1}{1-f(u)/f(z)} \cdot \frac{1}{1-f(v)/f(w)}
\left(\sum_{a=0}^{\infty} \frac{u^a}{v^{1+a}}\right)^2\right) \right),
\end{multline}
where
$$
 f(z)= (1+z)
\left(\frac1z  +  \sum_{a=1}^\infty \frac{\c_a z^{a-1} }{(a-1)!}\right)
$$
and $\frac{1}{1-f(u)/f(z)}$ is expanded in series as
$1+\frac{f(u)}{f(z)}+\left(\frac{f(u)}{f(z)}\right)^2+\dots$.
\end{lemma}
\begin{remark}
 The statement of the lemma is purely algebraic and does not claim anything about
 positivity. However, in our context the numbers $\cov(k,m)$ are not arbitrary, as any covariance
 matrix needs to be positive--definite. Similarly, the numbers $\p(k)$ satisfy
 constraints as moments of a positive probability measure.

 From the other side, the numbers $\c_k$, $\d_{k,m}$ are similarly not arbitrary, as
 they arise as limits of partial derivatives of symmetric functions with
 \emph{positive} decompositions into Schur polynomials (such functions can be
 identified with
 characters of the unitary group $U(N)$).

 It would be interesting to find out whether the formulas \eqref{eq_LLN_1_formula},
 \eqref{eq_CLT_1_formula}, \eqref{eq_LLN_inversion_formula}, \eqref{eq_CLT_inversion_formula} define a bijection between all possible families $\p(k)$,
 $\cov(k,m)$ (satisfying all the aforementioned inequalities) from one side and all
 possible families $\c_k$, $\d_{k,m}$ (again satisfying all the necessary
 inequalities) on the other side. We will not address this question here.
\end{remark}
\begin{proof}[Proof of Lemma \ref{Lemma_as_invert}]

The relations \eqref{eq_LLN_1_formula} and \eqref{eq_CLT_1_formula} express
$\mathfrak p(k)$ and $\mathfrak{cov}(k,m)$ as polynomials in $\c_a$, $\d_{a,b}$.
Moreover, these expressions have triangular structure. In more details, $\mathfrak
p(k)$ is expressed through $\c_a$ with $a=1,2,\dots,k$, and the dependence on $\c_k$
is linear (and non-degenerate).

Given all $\c_a$'s, $\mathfrak{cov}(k,m)$ is expressed through $\d_{a,b}$ with $a\le
k$, $b\le m$ in a linear way, and the dependence on $\d_{k,m}$ is non-degenerate.

It readily follows that the formulas\eqref{eq_LLN_1_formula} and
\eqref{eq_CLT_1_formula} can be inverted and yield a polynomial expressions for
$\c_a$, $\d_{a,b}$ through $\mathfrak p(k)$, $\mathfrak{cov}(k,m)$ with a similar
triangular structure.

Due to polynomiality, it then suffices to check \eqref{eq_LLN_inversion_formula},
\eqref{eq_CLT_inversion_formula} for such $\c_a$, $\d_{a,b}$ that the series
$\sum_{a=1}^\infty \frac{\c_a z^{a-1} }{(a-1)!}$ and $\sum_{a,b=1}^{\infty}
\frac{\d_{a,b}}{(a-1)! (b-1)!} z^{a-1} w^{b-1}$ are convergent for all
$z,w\in\mathbb C$. For \eqref{eq_LLN_inversion_formula} this check is contained in
the proof of \cite[Lemma 6.1]{BG} --- it reduces to the Lagrange inversion formula.

For \eqref{eq_CLT_inversion_formula} we  rewrite \eqref{eq_CLT_1_formula} as a
contour integral
\begin{multline}
\label{eq_x34} \frac{1}{(2\pi\ii)^2} \oint\oint_{\begin{smallmatrix} |z|=\eps\\
|w|=2\eps\end{smallmatrix}}
 f(z)^{k}
\, \ f(w)^{m}\left(  \left(\sum_{a=0}^{\infty} \frac{z^a}{w^{1+a}}\right)^2
+\sum_{a,b=1}^{\infty} \frac{\d_{a,b}}{(a-1)! (b-1)!} z^{a-1} w^{b-1} \right) dz dw,
\end{multline}
with integration around $0$ in both $z$ and $w$. We introduce two small complex
variables $u$ and $v$, multiply \eqref{eq_x34} by $u^k$, $v^m$ and sum over all
$k,m$ to get (we agree that $\mathfrak{cov}(0,0)=0$)
\begin{multline}
\label{eq_x35} \sum_{k,m=0}^\infty \mathfrak{cov}(k,m) u^k v^m=
 \frac{1}{(2\pi\ii)^2} \oint\oint_{\begin{smallmatrix} |z|=\eps\\
|w|=2\eps\end{smallmatrix}}
 \frac{1}{1-u f(z)} \cdot \frac{1}{1-v f(w)}  \frac{dzdw}{\left(z-w\right)^2} \\ +
  \frac{1}{(2\pi\ii)^2} \oint\oint_{\begin{smallmatrix} |z|=\eps\\
|w|=2\eps\end{smallmatrix}}
 \frac{1}{1-u f(z)} \cdot \frac{1}{1-v f(w)}
  \left(\sum_{a,b=1}^{\infty}
\frac{\d_{a,b}}{(a-1)! (b-1)!} z^{a-1} w^{b-1} \right) dz dw,
\end{multline}
The first term in the right--hand side does not depend on $\d_{a,b}$ and we treat it
as being known. For the second one note that when $\eps$ is small, and $u$ is even
smaller, the equation $f(z)=\frac{1}{u}$ has a unique solution $f^{(-1)}(1/u)$
inside the contour $|z|=\eps$. We can therefore compute the integral as a residue,
transforming \eqref{eq_x35} into
\begin{multline}
\label{eq_x36}
 \frac{1}{(2\pi\ii)^2} \oint\oint_{\begin{smallmatrix} |z|=\eps\\
|w|=2\eps\end{smallmatrix}}
 \frac{1}{1-u f(z)} \cdot \frac{1}{1-v f(w)}  \frac{dzdw}{\left(z-w\right)^2} \\ +
 \frac{1}{-u f'(f^{(-1)}(1/u))} \cdot \frac{1}{-u f'(f^{(-1)}(1/v))}
\cdot \left(  \sum_{a,b=1}^{\infty} \frac{\d_{a,b}}{(a-1)! (b-1)!}
(f^{(-1)}(1/u))^{a-1} (f^{(-1)}(1/v))^{b-1} \right) ,
\end{multline}
Therefore, setting $1/u=f(x)$, $1/v=f(y)$, and denoting
$$
 G(u,v)= \sum_{k,m=0}^\infty \mathfrak{cov}(k,m) u^k v^m-
 \frac{1}{(2\pi\ii)^2} \oint\oint_{\begin{smallmatrix} |z|=\eps\\
|w|=2\eps\end{smallmatrix}}
 \frac{1}{1-u f(z)} \cdot \frac{1}{1-v f(w)}  \frac{dzdw}{\left(z-w\right)^2},
$$
 we get
$$
 G\left(\frac{1}{f(x)},\frac{1}{f(y)}\right)=
  \frac{f(x)}{f'(x)} \cdot \frac{f(y)}{f'(y)}
\cdot \left(  \sum_{a,b=1}^{\infty} \frac{\d_{a,b}}{(a-1)! (b-1)!} x^{a-1} y^{b-1}
\right),
$$
which coincides with \eqref{eq_CLT_inversion_formula}.
\end{proof}

In \cite{BG}, \cite[Theorem 2.4, Theorem 2.8]{BG_CLT} we proved that if the measures
$\rho_N$ are CLT appropriate, then they satisfy CLT and the formulas
\eqref{eq_LLN_1_formula}, \eqref{eq_CLT_1_formula} hold. Here we prove the opposite
direction of Theorem \ref{Theorem_main}, i.e.\ that if measures satisfy CLT, then they are CLT--appropriate. The proof is split into several steps and culminates in Section \ref{Section_proof_of_inversion}.

\subsection{Expansion of differential operators}
\label{Section_diff_ops_degrees} Throughout this and the next section we explain how
the random power sums $p_k^N$ of \eqref{eq_random_power_sum} can be analyzed through
differential operators. Below we always silently assume that $\rho_N$ are smooth
measures.

\smallskip

We let $\x$ be the set of variables $(x_1,\dots,x_N)$ and define
$$
 V(\x)=\prod_{1\le i<j \le N} (x_i-x_j).
$$
Introduce the following differential expression in $N$ variables:
\begin{equation}
\label{eq_D_k_def}
 \mathcal D_k = V(\x)^{-1} \left(\sum_{i=1}^N \left({x_i} \frac{\partial}{\partial
 x_i} \right)^k \right) V(\x).
\end{equation}
We will generally apply this operator to $S_{\rho_N}$ to compute the moments of
$\rho_N$. We use the notation
\begin{equation}
\label{eq_uncentered_moments}
 \mathcal P_k:=\mathcal D_{k}[S_{\rho_N}] \bigr|_{x_1=\dots=x_N=1}
\end{equation}
 We also need a version for cumulants. For an integral vector
$k_1,\dots, k_r \ge 1$ let us define
\begin{equation}
\label{eq_centered_moments}
 \K_{k_1,\dots,k_r}:= \sum_{\left\{ U_1,
\cdots, U_t \right\} \in \Theta_r} (-1)^{t-1} (t-1)! \prod_{i=1}^t \left(
\left[\prod_{j \in U_i} \mathcal D_{k_j}\right][S_{\rho_N}]
\right)\Biggr|_{x_1=\dots=x_N=1}.
\end{equation}
where $[S_{\rho_N}]$ means an application of a differential operator to the function
$S_{\rho_N}$. In principle, all operators in \eqref{eq_centered_moments} commute and
therefore the order in which we apply them is irrelevant. However, for the sake of
being precise we specify this order by declaring that we first apply the factor
corresponding to the largest $j$, then the factor with second to largest $j$ acts on
the result, etc.

The following lemma is a starting point for all our proofs.
\begin{lemma} We have \label{Lemma_eigenrelation}
$$
\mathcal P_k= \E \bigl[N^k  p^N_k\bigr],\quad \quad \K_{k_1,\dots,k_r}=
\kappa_{k_1,\dots,k_r}\bigl( N^k p^N_k, \, k=1,2,\dots\bigr).
$$
\end{lemma}
\begin{proof}
 This is a corollary of the eigenrelarion $\mathcal D_k (s_\mu(\x))=\sum_{i=1}^N
 (\mu_i+N-i)^k s_\mu(\x)$.
\end{proof}

In what follows, we analyze the combinatorics of $\mathcal P_k$ and
$\K_{k_1,\dots,k_r}$. The main difficulty lies in the fact that we need to plug in
$x_1=\dots=x_N=1$ into $V(\x)$ and $V(\x)^{-1}$ in \eqref{eq_D_k_def} and therefore
to resolve the arising singularity.

Nevertheless, we explain below that $\mathcal P_k$ and $\K_{k_1,\dots,k_r}$ can be
written as a linear combination of partial derivatives of $\ln(S_{\rho_N})$ at
$x_1=\dots=x_N=1$. Since $S_{\rho_N}$ is symmetric, there is no need to consider all
possible partial derivatives. In more details, recall that a partition (or Young
diagram) $\lambda$ is a sequence of non-negative integers
$\lambda_1\ge\lambda_2\ge\dots\ge 0$, such that $|\lambda|:=\sum_{i=1}^{\infty}
\lambda_i<\infty$. The number of non-zero parts in $\lambda$ is denoted
$\ell(\lambda)$. Introduce the notation
$$
 \partial_{\lambda}[\ln S_{\rho_N}]=\left(\prod_{i=1}^{\ell(\lambda)}
 (\partial_i)^{\lambda_i}\right)[\ln S_{\rho_N}].
$$
$|\lambda|$ is the \emph{order} of the derivative. We will also need to compare the
derivatives of the same order. For that we use the \emph{dominance order} on
partitions: for $\lambda$, $\mu$ such that $|\lambda|=|\mu|$ we write
$\lambda\succeq \mu$ if
$$
 \lambda_1+\lambda_2+\dots+\lambda_k \ge \mu_1+\dots+\mu_k, \quad k=1,2,\dots.
$$
Note that the dominance order is only a partial order, as many partitions are
incomparable. One way to extend it to a complete (i.e.\ linear) order is through the
lexicographic ordering\footnote{Recall that a vector $(u_1,u_2,\dots)$ is larger
 than $(v_1,v_2,\dots)$ in the lexicographic order, if for some $k>0$, $u_1=v_1$, $u_2=v_2$,\dots $u_{k-1}=v_{k-1}$, and $u_k>v_k$.}

The expansions of $\mathcal P_k$ and $\K_{k_1,\dots,k_r}$ in partial derivatives of
$\ln(S_{\rho_N})$ are summarized in the following two theorems.

\begin{theorem} \label{Theorem_P}
 $\mathcal P_k$ can be written as a finite (\emph{independent} of $N$) linear combination of
products of partial derivatives of $\ln(S_{\rho_N})$ at $x_1=\dots=x_N=1$, whose
coefficients are polynomials in
 $N$. The highest order of partial derivatives in this expansion is $k$, and the
 total
 contribution of the derivatives of order $k$ can be written as
 \begin{equation}
 \label{eq_P_equation}
N^{k} \sum_{\lambda: |\lambda|=k} \alpha_k(\lambda) \cdot \partial_{\lambda} [\ln
(S_{\rho_N})]_{x_1=\dots=x_N=1} + O(N^{k-1}),
 \end{equation}
 where the summation goes over all Young diagrams $\lambda$ with $|\lambda|=k$,
  $\alpha_k(\lambda)$ are some coefficients (independent of $N$) and
 $O(N^{k-1})$ is a linear combination of partial derivatives of $\ln(S_{\rho_N})$ of
 order $k$ and with all coefficients being smaller in absolute value than
 $const \cdot N^{k-1}$. Finally, $\alpha_k(k,0,\dots)\ne 0$.
\end{theorem}
\begin{theorem} \label{Theorem_M}
 Take a partition $\mathbf k=(k_1\ge k_2 \ge\dots k_r)$ with $\ell(\mathbf k)=r$.
 $\K_{k_1,\dots,k_r}$ can be written as a finite (\emph{independent} of $N$) linear combination of
products of partial derivatives of $\ln(S_{\rho_N})$ at $x_1=\dots=x_N=1$, whose
coefficients are polynomials in
 $N$. The highest order of partial derivatives in this expansion is $|\mathbf k|$, and the
 contribution of the derivatives of order $\mathbf k|$ can be written as
 \begin{equation}
 \label{eq_M_equation}
  N^{k_1+\dots+k_r} \sum_{\lambda \preceq \mathbf k} \alpha_{k_1,\dots,k_r}(\lambda)
  \cdot \partial_\lambda [\ln (S_{\rho_N})]_{x_1=\dots=x_N=1} + O(N^{k_1+\dots+k_r-1}),
 \end{equation}
 where the summation goes over all Young diagrams $\lambda$ with $|\lambda|=|\mathbf k|$, which are smaller in
 the dominance order than $\mathbf k$, $\alpha_{k_1,\dots,k_r}(\lambda)$ are some coefficients (independent of $N$) and
 $O(N^{k_1+\dots+k_r-1})$ is a linear combination of
  partial derivatives of $\ln(S_{\rho_N})$ of
 order $k_1+\dots+k_r$, and with all coefficients being smaller in absolute value than
 $const \cdot N^{k_1+\dots+k_r-1}$. Finally, $\alpha_{k_1,\dots,k_r}(k_1,\dots,k_r)\ne 0$.
\end{theorem}

The proof of Theorems \ref{Theorem_P}, \ref{Theorem_M} is given in the next Section
\ref{Section_Highest_order}.

\subsection{Highest order components of differential operators}

\label{Section_Highest_order}

We start discussing the structure of \eqref{eq_uncentered_moments},
\eqref{eq_centered_moments}. Expanding by the Leibnitz rule, the application of
$\mathcal D_k$ to a function $f(\x)$ can be written as a sum of partial derivatives
of $f$ divided by linear factors $(x_i-x_j)$, and multiplied by momonial factors
$x_i^m$. In more details, we get a sum of the terms of the form
\begin{equation}
\label{eq_D_k_expansion}
  \frac{(x_i)^q (\partial_i)^n f(\x)}{\prod_{j\in A}(x_i-x_j)},
\end{equation}
indexed by $i$, $A\subset\{1,\dots,N\}\setminus \{i\}$, $q \ge 0$, $n\ge 0$.

When differentiating $S_{\rho_N}$, we write $S_{\rho_N}=\exp(\ln(S_{\rho_N}))$ and
therefore
$$
 \frac{\partial}{\partial x_i} S_{\rho_N} = S_{\rho_N} \cdot \frac{\partial}{\partial x_i}
 (\ln (S_{\rho_N})).
$$
Hence, both $\mathcal P_k$ and $\K_{k_1,\dots,k_r}$ can be written as sums of terms,
each of which is $S_{\rho_N}$ raised to some power, multiplied by partial
derivatives of $(\ln (S_{\rho_N}))$, multiplied by some collection of variables
$x_i$ and divided by some linear factors $1/(x_i-x_j)$. In the end, we need to plug
in $x_1=\dots=x_N=1$ into the differential operators giving $\mathcal P_k$ and
$\K_{k_1,\dots,k_r}$. Since $S_{\rho_N}=1$ upon such substitution, we can ignore its
powers in the end. The conclusion from this discussion is that the result, i.e.\
$\mathcal P_k$ and $\K_{k_1,\dots,k_r}$ can be written as an expression involving
only various partial derivatives of $\ln( S_{\rho_N}(\x))$ and nothing else. We
introduce the notations for the latter derivatives by writing,
\begin{equation}
\label{eq_S_expansion}
 \ln(S_{\rho_N}(x_1,\dots,x_N))=\sum_{k_1,k_2,\dots,k_N=0}^\infty {\mathfrak
 s}^{k_1,\dots,k_N} \prod_{i=1}^N (x_i-1)^{k_i},
\end{equation}
$$
{\mathfrak
 s}^{k_1,\dots,k_N}= \left[ \prod_{i=1}^N \frac{(\partial_i)^{k_i}}{k_i!} \right]
 \ln(S_{\rho_N}(x_1,\dots,x_N)) \bigl|_{x_1=\dots=x_N=1}.
$$
The number $k_1+\dots+k_N$ is the order of the partial derivative corresponding to
${\mathfrak
 s}^{k_1,\dots,k_N}$.

Since the measure ${\rho_N}$ is smooth, the series \eqref{eq_S_expansion} is
uniformly converging to an analytic function in a neighborhood of $1^N$, and there
exists $R>0$ such that a uniform bound holds:
$$
|{\mathfrak
 s}^{k_1,\dots,k_N}|\le R^{k_1+\dots+k_N},
$$
Further, since $S_{\rho_N}$ is symmetric, so are the coefficients $\mathfrak
s^{k_1,\dots,k_N}$. Finally, $S_{\rho_N}(1,\dots,1)=1$ implies ${\mathfrak
 s}^{0,\dots,0}=0$.

Propositions \ref{Proposition_P_as_sum}, \ref{Proposition_Product_as_sum} and
Corollary \ref{Corollary_M_as_sum} summarize our basic point of view on $\mathcal
P_k$ and $\K_{k_1,\dots,k_r}$.
\begin{proposition} \label{Proposition_P_as_sum}
$\mathcal P_k$ can be written as a finite (\emph{independent} of $N$) sum of
products of
 $\mathfrak s^{m_1,\dots,m_h,0,\dots,0}$, whose coefficients are polynomials in
 $N$. The largest order, i.e.\ the value of $m_1+\dots+m_h$, appearing in this sum is $k$.
\end{proposition}
\begin{proof}
 We write $\mathcal D_k$ as a sum of the terms of the form \eqref{eq_D_k_expansion}.
 Note that in this formula $0\le n+|A|+(k-q) \le k$. Further, note that $\mathcal
 D_k$ is a symmetric expression in $x_1,\dots,x_N$, therefore, together with each
 term of the form \eqref{eq_D_k_expansion} all possible its permutations also arise.
 Thus, $\mathcal D_k$ can be written as a linear combination (with coefficients depending only on $k$,
 but not on $N$ or anything else) of the terms of the form:
\begin{equation}
\label{eq_operator_D_k_sym}
 \sum_{ \{i, j_1,\dots,j_m\}\subset \{1,\dots,N\}} \Sym_{i,j_1,\dots,j_m} \frac{
 (x_i)^q (\partial_i)^{n}}{(x_i-x_{j_1})(x_i-x_{j_2})\cdots (x_i-x_{j_m})},
\end{equation}
with $0\le q,n,m\le k$ and $n+m+(k-q) \le k$, and where for an expression
$\Psi(x_{i_1},\dots,x_{i_p})$, we use a notation
$$
 \Sym_{i_1,\dots,i_p} \Psi=\sum_{\sigma\in {\mathfrak S}(p)} \Psi \bigl(x_{i_{\sigma(1)}},\dots,
 x_{i_{\sigma(p)}} \bigr),
$$
and $\mathfrak S (p)$ is the group of all permutations of $\{1,\dots,p\}$.

When we apply \eqref{eq_operator_D_k_sym} to $S_{\rho_N}(\x)$, we write
\begin{multline}
\label{eq_x8}
 (\partial_i)^n S_{\rho_N}= (\partial_i)^n \exp( \ln (S_{\rho_N}))=
 (\partial_i)^{n-1} \left[ \exp(\ln(S_{\rho_N})) \cdot \partial_i (\ln(S_{\rho_N})) \right]
 =\dots\\ = S_{\rho_N} \sum_{\lambda \vdash n}c_\nu \prod_{i=1}^{\infty}
 \Bigl[(\partial_i)^{\nu_i}
 \ln(S_{\rho_N}) \Bigr],
\end{multline}
where the sum goes over all partitions $\nu=(\nu_1\ge \nu_2\ge\dots)$ of $n$, and
$c_\nu$ is a combinatorial coefficient.

Therefore, $\mathcal D_k [S_{\rho_N}(\x)]$ is a finite linear combination of the
terms
\begin{equation}
\label{eq_x1}
 S_{\rho_N}(\x)
  \sum_{ \{i, j_1,\dots,j_m\}\subset \{1,\dots,N\}} \Sym_{i,j_1,\dots,j_m} \frac{
 (x_i)^q \prod_{i=1}^{\infty}
 \bigl[(\partial_i)^{\nu_i} \ln(S_{\rho_N}) \bigr] }{(x_i-x_{j_1})(x_i-x_{j_2})\cdots (x_i-x_{j_m})},
\end{equation}
with $m+(k-q)+\sum_i \nu_i \le k$. Due to symmetry of $\ln(S_{\rho_N})$, all terms
in the sum of \eqref{eq_x1} are essentially the same. Let us analyze the one
corresponding to $i=1$, $\{j_1,\dots,j_m\}=\{2,\dots,m+1\}$. We plug in the
expansion \eqref{eq_S_expansion}, differentiate and multiply to get a power series
expansion for $\prod_{i=1}^{\infty}
 \bigl[(\partial_i)^{\nu_i} \ln(S_{\rho_N}) \bigr]$. Using the symmetry of $\ln(S_{\rho_N})$, the computation reduces to
 studying
\begin{equation}
\label{eq_x2}
 \Sym_{1,\dots,m+1} \frac{
 (x_1)^q \prod\limits_{i=1}^{m+1} (x_i-1)^{\mu_i}
   }{(x_1-x_{2})(x_1-x_{3})\cdots (x_1-x_{m+1})},
\end{equation}
for any integers  $\mu_1,\dots,\mu_{m+1}\ge 0$.

{\bf Claim 1}: \eqref{eq_x2} is a symmetric \emph{polynomial} in
$(x_1-1),\dots,(x_{m+1}-1)$ of degree at most $q-m+\sum_i \mu_i$. Indeed, if we
multiply \eqref{eq_x2} by $V(x_1,\dots,x_{m+1})$ inside the symmetrization, then the
result is skew--symmetric on one hand, and a polynomial on the other hand. Thus, it
should be divisible (as a polynomial) by $V(x_1,\dots,x_{m+1})$, and we obtain the
desired statement.

\smallskip

{\bf Claim 2}: If we expand \eqref{eq_x2} into monomials in $(x_1-1)$, \dots,
$(x_{m+1}-1)$, then each monomial has degree at least $\mu_1+\dots+\mu_{m+1}-m$. (In
particular, if $\mu_1+\dots+\mu_{m+1}>m$, then the polynomial vanishes at the point
$x_1=\dots=x_{m+1}=1$.) Indeed, choose $m+1$ distinct real numbers
$\alpha_1,\dots,\alpha_{m+1}$ and set $x_i=1+ \eps \cdot \alpha_i$, $i=1,\dots,m+1$,
where $\eps$ is a small number. Then each term in \eqref{eq_x2} becomes
$O(\eps^{\sum_i \mu_i- m})$, and therefore the sum is also $O(\eps^{\sum_i \mu_i-
m})$, which implies the claim.

\smallskip

We now plug in $x_1=\dots=x_N=1$ into \eqref{eq_x1}. Since $S_{\rho_N}(\x)$ is
symmetric and $S_{\rho_N}(1,\dots,1)=1$, the result is
\begin{equation}
\label{eq_x3}
 {N \choose {m+1}} \cdot \Sym_{1,\dots,m+1} \frac{
 (x_1)^q \prod_{i=1}^{\infty}
 \bigl[(\partial_1)^{\nu_i} \ln(S_{\rho_N}) \bigr] }{(x_1-x_{2})(x_1-x_{3})\cdots (x_1-x_{m+1})} \Biggr|_{x_1=\dots=x_N=1},
\end{equation}
Note that we can set $x_{m+2}=\dots=x_{N}=1$ in \eqref{eq_x3} even before computing
derivatives and symmetrizing and then use
$$
 \ln(S_{\rho_N}(x_1,\dots,x_{m+1},1,\dots,1)=\sum_{\ell_1,\dots,\ell_{m+1}=0}^{\infty}
 {\mathfrak s}^{\ell_1,\dots,\ell_{m+1},0,\dots,0} \prod_{i=1}^{m+1} (x_i-1)^{\ell_i},
$$
which leads to
\begin{equation}
\label{eq_x4} \prod_{i=1}^{\infty}
 \bigl[(\partial_1)^{\nu_i} \ln(S_{\rho_N}(x_1,\dots,x_{m+1},1,\dots,1))
 \bigr]=\sum_{\mu_1,\dots,\mu_{m+1}=0}^{\infty} c(\mu_1,\dots,\mu_{m+1})
 \prod_{i=1}^{m+1} (x_i-1)^{\mu_i},
\end{equation}
where each $c(\mu_1,\dots,\mu_{m+1})$ is a polynomial in coefficients ${\mathfrak
s}^{\ell_1,\dots,\ell_{m+1}}$, such that if $\mathfrak s^{\ell(1)} \mathfrak s^{\ell(2)}\cdots
\mathfrak s^{\ell(t)}$ is a monomial entering this polynomial, then
$$
 \sum_{j=1}^t \sum_{i=1}^{m+1} \ell(j)_i= \sum_{i=1}^{m+1} \mu_i + \sum_{j=1}^{\infty} \nu_i
$$
 We plug \eqref{eq_x4} into
\eqref{eq_x3}. Claim 2 implies that the result (and therefore, $\mathcal P_k$) does
not depend on terms in \eqref{eq_x4} with $\sum_i \mu_i>m$. Thus, $\mathcal P_k$
depends only on ${\mathfrak s}^{\ell_1,\dots,\ell_{m+1}}$ with
$\ell_1+\dots+\ell_{m+1} \le m +\sum_i \nu_i \le k$.
\end{proof}

\begin{proposition} \label{Proposition_Product_as_sum} For any integers $k_1,\dots,k_r\ge 0$ with
$\sum_{a=1}^r k_a=K$, the expression
$$
  \left(\prod\limits_{a=1}^r \mathcal D_{k_a}\right)
 [S_{\rho_N}(\x)]\Bigr|_{x_1=\dots=x_N=1}
$$
 can be written
 as a finite (\emph{independent} of $N$) sum of products of
 $\mathfrak s^{m_1,\dots,m_h,0,\dots,0}$, whose coefficients are polynomials in
 $N$. The largest order of $m_1+\dots+m_h$ appearing in this sum is
 $K$.
\end{proposition}

\begin{proof}[Proof of Proposition \ref{Proposition_Product_as_sum}]
 We start by writing each $D_{k_a}$ in $ \left(\prod\limits_{a=1}^r \mathcal D_{k_a}\right)
 [S_{\rho_N}(\x)]$ as a sum of the form \eqref{eq_operator_D_k_sym}. Each term thus comes with a subset
 $\mathcal M_a\subset \{1,\dots,N\}$ (that was $\{i,j_1,\dots,j_m\}$ in \eqref{eq_operator_D_k_sym}
 ) of at most $k_a$ elements, after we multiply the expressions we get

 \begin{equation}
 \label{eq_x11}
 \sum_{\mathcal M_1,\dots, \mathcal M_r \subset\{1,\dots,N\}} \text{``expression in variables from } \mathcal M_1\cup \mathcal
 M_2 \cup \dots \cup \mathcal M_r\text{''}.
\end{equation}
It is convenient to define the combinatorial type $\mathfrak{Co}(\mathcal
M_1,\dots,\mathcal M_r)$ as the set of numbers $\bigl(|\bigcap_{i\in \mathcal R}
\mathcal M_i|\bigr)_{\mathcal R\subset \{1,\dots,r\}}$. The combinatorial type
$\mathfrak c$ defines through the inclusion--exclusion principle the number
$|\bigcup_{a=1}^r \mathcal M_r|$ which we denote $M(\mathfrak c)$.

 We choose the sets $\mathcal M_1,\dots,\mathcal M_r$ in \eqref{eq_x11} by a three--step procedure: we
first choose their combinatorial type $\mathfrak c$, then a $M(\mathfrak
c)$--element subset $\overline{\mathcal M}\subset \{1,\dots,N\}$, and then $\mathcal
M_1,\dots,\mathcal M_r \subset \overline{\mathcal M}$. Note $M(\mathfrak c)\le K$.
We will further apply $\eqref{eq_x11}$ to $S_{\rho_N}(\x)$, and then set all
variables equal to $1$. Observe that, due to symmetry in $x_i$, the result is
independent of the choice of $\overline{\mathcal M}$, and therefore we can write it
as a result of setting $x_1=\dots=x_{M(\mathfrak c)}=1$ in
 \begin{multline}
\label{eq_x7}
 \sum_{\mathfrak{c}} { N \choose M(\mathfrak c) } \sum_{\begin{smallmatrix} \mathcal M_1\bigcup \dots \bigcup \mathcal M_r =\{1,\dots,M(\mathfrak c)\} \\ \mathfrak{Co}(\mathcal M_1,\dots,\mathcal M_r)=\mathfrak c\end{smallmatrix}}
\\  \left( \prod_{a=1}^{r} \Sym_{\mathcal M_a=(i,j_1,\dots,j_m)} \frac{
 (x_i)^{q_a} (\partial_i)^{n_{a}}}{(x_i-x_{j_1})(x_i-x_{j_2})\cdots (x_i-x_{j_m})}  \right) \Bigr[
S_{{\rho_N}}(x_1,\dots,x_{M(\mathfrak c)},1,\dots,1)\Bigl],
\end{multline}
where $n_a+|\mathcal M_a|-1\le k_a$.
 The two important features of \eqref{eq_x7} is that the dependence on $N$ is isolated in ${N
\choose M(\mathfrak c)}$ prefactors (and therefore, this dependence is polynomial,
as desired), and that all but $M(\mathfrak c)$ variables in $S_{\rho_N}(\x)$ are set
to $1$ before taking any derivatives.

\smallskip
Fix a number $M$ and a combinatorial type $\mathfrak c$ such that $M(\mathfrak c)\le M$. Consider
the expression
 \begin{multline}  \sum_{\begin{smallmatrix} \mathcal M_1, \dots, \mathcal M_r \subset \{1,\dots,M(\mathfrak c)\} \\ \mathfrak{Co}(\mathcal M_1,\dots,\mathcal M_r)=\mathfrak c\end{smallmatrix}}
\label{eq_x9}   \left( \prod_{a=1}^{r} \Sym_{\mathcal M_a=(i,j_1,\dots,j_m)} \frac{
 (x_i)^{q_a} (\partial_i)^{n_{a}}}{(x_i-x_{j_1})(x_i-x_{j_2})\cdots (x_i-x_{j_m})}  \right) \\ \Bigr[
S_{{\rho_N}}(x_1,\dots,x_M,1,\dots,1)\Bigl],
\end{multline}
where $n_a+|\mathcal M_a|-1\le k_a$. We prove by induction in $r$ the following properties:
 \begin{enumerate}[label=(\Alph*)]
  \item \label{en_x1} \eqref{eq_x9} can be decomposed into a series, which is uniformly and absolutely
  convergent in a neighborhood of
  $(1,\dots,1)$, of the form
  \begin{equation}
  \label{eq_x5}
  S_{\rho_N}(x_1,\dots,x_M,1,\dots,1) \sum_{m_1,\dots,m_M\ge 0} c(m_1,\dots,m_M) \prod_{i=1}^M (x_i-1)^{m_i}
  \end{equation}
 \item \label{en_x2} $c(m_1,\dots,m_M)$ is symmetric in the first $M(\mathfrak c)$ indices $m_1,\dots,m_{M(\mathfrak c)}$.
 \item \label{en_x3} $c(m_1,\dots,m_M)$ is polynomial in
 coefficients $\mathfrak s^{\ell_1,\dots,\ell_M}$ arising
 in the decomposition
$$
 \ln(S_{\rho_N}(x_1,\dots,x_{M},1,\dots,1)=\sum_{\ell_1,\dots,\ell_{M}=0}^{\infty}
 {\mathfrak s}^{\ell_1,\dots,\ell_{M},0,\dots,0} \prod_{i=1}^{M} (x_i-1)^{\ell_i},
$$
and such that if $\mathfrak s^{\ell(1)} \mathfrak s^{\ell(2)}\cdots \mathfrak
s^{\ell(t)}$ is a monomial entering this polynomial (where $\ell(j)$ is an
$M$--tuple $\ell(j)_1$, \dots, $\ell(j)_M$), then
\begin{equation}
\label{eq_degrees_inequality}
 \sum_{j=1}^t \sum_{i=1}^{M} \ell(j)_i\le (m_1+\dots+m_M) + (k_1+\dots+k_r).
\end{equation}
 \end{enumerate}

 In the base case $r=1$ the properties \ref{en_x1}, \ref{en_x2}, \ref{en_x3} are checked in Proposition
 \ref{Proposition_P_as_sum}. For general $r$, first, note that the symmetry of $c(m_1,\dots,m_M)$ follows from the symmetry of the expression \eqref{eq_x9}.
 Further, for the induction step, we \emph{fix} a set $\mathcal M_{r+1}$, which we assume without
 loss of generality to have $w$ elements $\{M(\mathfrak c)-w+1, M(\mathfrak c)-w+2,\dots, M(\mathfrak c)
\}$. We will apply the symmetrization and derivatives in this set \emph{after} doing
so for $\mathcal M_1,\dots,\mathcal M_r$. Suppose
 that $\bigcup_{a=1}^r \mathcal M_r= M'$; note that $M(\mathfrak c)-M'\le w$.
 We first use the induction assumption for the sets $\mathcal M_1,\dots,\mathcal M_r$ (the
 parameter $M$ is unchanged), take the resulting expression of \eqref{eq_x5}, then symmetrize
 it over the \emph{last}
 $w$ variables $\{M(\mathfrak c)-w+1, M(\mathfrak c)-w+2,\dots, M(\mathfrak c)
 \}$ --- this symmetrization is implicit in \eqref{eq_x9}, since it involves the sum over all
 possible choices of $\mathcal M_1,\dots,\mathcal M_r$. We arrive at an expression
   \begin{equation}
  \label{eq_x12}
  S_{\rho_N}(x_1,\dots,x_M,1,\dots,1) \sum_{m_1,\dots,m_M\ge 0} c'(m_1,\dots,m_M) \prod_{i=1}^M
  (x_i-1)^{m_i},
  \end{equation}
  with coefficients $c'(m_1,\dots,m_M)$ being symmetric in the indices $\{M(\mathfrak c)-w+1, M(\mathfrak c)-w+2,\dots, M(\mathfrak
  c)$ (it is also symmetric in the first $M(\mathfrak c)-w$ indices, but not in all the first
  $M(\mathfrak c)$ indices, because we fixed $\mathcal M_{r+1}$). We apply
\begin{multline*}
\Sym_{\mathcal M_{r+1}=(M(\mathfrak c), M(\mathfrak c)-1,\dots, M(\mathfrak c)-w+1)} \\ \frac{
 (x_{M(\mathfrak c)})^q (\partial_{M(\mathfrak c)})^{n}}{(x_{M(\mathfrak c)}-x_{M(\mathfrak c)-1})(x_{M(\mathfrak c)}-x_{M(\mathfrak c)-2})\cdots (x_{M(\mathfrak c)}-
 x_{M(\mathfrak c)-w+1})},
\end{multline*}
to \eqref{eq_x12}, where as before, $m+n\le k_{r+1}$. When we act by
$(\partial_{M(\mathfrak c)})^{n}$, each derivative can either act on $S_{\rho_N}$,
or on the sum in \eqref{eq_x12}. Using \eqref{eq_x8} and symmetry of
$c(m_1,\dots,m_M)$, we arrive at a sum of the terms of the form
\begin{multline} \label{eq_x10}
 S_{\rho_N}(x_1,\dots,x_M,1,\dots,1) \cdot \Sym_{\mathcal M_{r+1}} \Biggl[ \frac{(x_{M(\mathfrak c)})^q \cdot
 \prod_{i=1}^\infty [(\partial_{M(\mathfrak c)})^{\nu_i} \ln S_{\rho_N}(x_1,\dots,x_M,1,\dots,1)] } {
 (x_{M(\mathfrak c)}-x_{M(\mathfrak c)-1})\cdots(x_{M(\mathfrak c)}-x_{M(\mathfrak c)-m+1})}
 \\ \times (\partial_{M(\mathfrak c)})^{\nu_0} \sum_{m_1,\dots,m_M\ge 0} c'(m_1,\dots,m_M) \prod_{i=1}^M
 (x_i-1)^{m_i} \Biggr],
\end{multline}
were $\nu_0+\nu_1+\nu_2+\dots=n$. It remains to use Claim 1 and Claim 2  in the proof of
Proposition \ref{Proposition_P_as_sum} (in variables $\{M(\mathfrak c)-w+1,\dots,M(\mathfrak
c)\}$), and then sum over all other possible choices of the set $\mathcal M_{r+1}$. This finishes
the proof of properties \ref{en_x1}, \ref{en_x2}, \ref{en_x3}. We now use these  properties  (for
the case $M=M(\mathfrak c)$) when plugging $x_1=\dots=x_{M(\mathfrak c)}=1$ in \eqref{eq_x7} to
finish the proof of the proposition.
\end{proof}

\begin{corollary} \label{Corollary_M_as_sum}
 For any integers $k_1,\dots,k_r\ge 0$ with
$\sum_{a=1}^r k_a=K$, $\K_{k_1,\dots,k_r}$
 can be written
 as a finite (\emph{independent} of $N$) sum of products of
 $\mathfrak s^{m_1,\dots,m_K,0,\dots,0}$, whose coefficients are polynomials in
 $N$. The highest order of $m_1+\dots+m_K$ appearing in this sum is
 $K$. All the highest order terms arise from $  \left(\prod_{a=1}^r  \mathcal D_{k_a} \right)
 \, [S_{\rho_N}]\bigr|_{x_1=\dots=x_N=1}.$
\end{corollary}

We now would like to analyze the highest order terms in the expansions of Corollary
\ref{Corollary_M_as_sum} and Proposition \ref{Proposition_P_as_sum}. Due to symmetry, these terms
are enumerated by partitions of $K=k_1+\dots+k_r$. For convenience, we also reorder $\mathbf
k=(k_1\ge k_2\ge \dots\ge k_r)$.

By Corollary \ref{Corollary_M_as_sum}, the highest order terms arise from $  \left(\prod_{a=1}^r
\mathcal D_{k_a} \right)
 \, [S_{\rho_N}]\bigr|_{x_1=\dots=x_N=1}$, and therefore are given as the results of setting
 $x_1=\dots=x_M=1$ in \eqref{eq_x7}. The expression \eqref{eq_x7} is further analyzed inductively
 in $r$ by using the formula \eqref{eq_x10}.

\begin{lemma} \label{Lemma_highest_terms}
 When we compute the highest order terms in $\K_{k_1,\dots,k_r}$ (or in $\mathcal P_k$) inductively as in Proposition \ref{Proposition_Product_as_sum},  then at the
 base step $r=1$, the partition $\nu$ in \eqref{eq_x1} is such that $\nu_2=\nu_3=\dots=0$, $\nu_1>0$, and in
 the induction step in \eqref{eq_x10}, we should have $\nu_1=\nu_2=\dots=0$, $\nu_0>0$.
\end{lemma}
\begin{proof}
 We obtain the highest order term if and only if on each step of the inductive procedure of Proposition
 \ref{Proposition_Product_as_sum} the inequality \eqref{eq_degrees_inequality} turns into an
 equality and $t=1$ (note $\mathfrak s^{0,\dots,0}=0$). Therefore, there should
 be no multiplications of power series in \eqref{eq_x1}, \eqref{eq_x10}, which implies that
 $\nu_2=\nu_3=\dots=0$ in \eqref{eq_x1} and \eqref{eq_x10}. If in \eqref{eq_x1} also $\nu_1=0$,
 then the result does not depend on $\ln(S_{\rho_N})$ at all, and therefore
 \eqref{eq_degrees_inequality} is strict. Thus, we should have $\nu_1>0$ in \eqref{eq_x1}. For
 \eqref{eq_x10}, the absence of multiplications implies that also $\nu_1=0$. Finally, if $\nu_0$
 also vanishes, then due to symmetry of $c'(m_1,\dots, m_M)$, we can factor \eqref{eq_x10} as
 \begin{multline}
 \Sym_{\mathcal M_{r+1}} \Biggl[ \frac{(x_{M(\mathfrak c)})^q} {
 (x_{M(\mathfrak c)}-x_{M(\mathfrak c)-1})\cdots(x_{M(\mathfrak c)}-x_{M(\mathfrak c)-m+1})}
 \Biggr]
 \\ \times  \sum_{m_1,\dots,m_M\ge 0} c'(m_1,\dots,m_M) \prod_{i=1}^M
 (x_i-1)^{m_i},
\end{multline}
and observe that the first factor does not depend on $\ln(S_{\rho_N})$, and
\eqref{eq_degrees_inequality} is again strict.
\end{proof}

\begin{proof}[Proof of Theorem \ref{Theorem_P}]
 Lemma \ref{Lemma_highest_terms} implies that the highest order terms of $ \mathcal
 D_k[S_{\rho_N}]$ are the same as those for $\mathcal
 D_k[\ln(S_{\rho_N})]$, and therefore it can be a written as a linear combination of partial derivatives parameterized
 by Young diagrams, as in \eqref{eq_P_equation}. The $N$--dependence of the
 coefficients in this linear combination comes from the binomial coefficients in
 \eqref{eq_x3}, and therefore, the leading (in $N$) contribution comes from the
 maximal value of $m$. Recall that by definitions, $m\le k$. If $m=k$, then all
 $\nu_i$ in \eqref{eq_x3} are equal to $0$ --- this does not contribute to the highest order
 by Lemma \ref{Lemma_highest_terms}, hence, it is not what we want. We conclude that for
 the leading (in $N$) contribution we should have $m=k-1$, and the contributing term
 can be then written as
\begin{equation} \label{eq_x13}
 \frac{N^k}{k!} \cdot \Sym_{1,\dots,k} \frac{
 (x_1)^k  \partial_1[\ln(S_{\rho_N})]  }{(x_1-x_{2})(x_1-x_{3})\cdots (x_1-x_k)} \Biggr|_{x_1=\dots=x_N=1},
\end{equation}
We already know (cf.\ Claim 1 and Claim 2 in the proof of Proposition
\ref{Proposition_P_as_sum}) that \eqref{eq_x13} is a linear combination of partial
derivatives of $\ln(S_{\rho_N})$ at $1^N$ of order at most $k$. Therefore, it only
remains to show that the derivative $(\partial_1)^k \ln(S_{\rho_N})$ enters with a
non-zero coefficient. Expanding $\ln(S_{\rho_N})$ into a power series in $(x_i-1)$,
this is the same as checking the non-vanishing of
\begin{equation} 
 \frac{1}{k!} \cdot \Sym_{1,\dots,k} \frac{
 (x_1)^k  \partial_1[ (x_1-1)^k]  }{(x_1-x_{2})(x_1-x_{3})\cdots (x_1-x_k)} \Biggr|_{x_1=\dots=x_N=1},
\end{equation}
which follows from the explicit formula for this expression of \cite[Lemma 5.5]{BG}.
\end{proof}

\begin{proof}[Proof of Theorem \ref{Theorem_M}]
  Lemma \ref{Lemma_highest_terms} implies that the highest order terms of $ \K_{k_1,\dots,k_r}$ are the same as those for $\left(\prod_{a=1}^r\mathcal
 D_{k_a}\right)[\ln(S_{\rho_N})]$,
 and therefore it can be a written as a linear combination of partial derivatives parameterized
 by Young diagrams, as in \eqref{eq_M_equation}. The $N$--dependence of the
 coefficients in this linear combination comes from the binomial coefficients in
 \eqref{eq_x7}, and therefore, the leading (in $N$) contribution comes from the
 maximal value of $M(\mathfrak c)$. Lemma \ref{Lemma_highest_terms} implies that
 $n_a\ge 1$ 
 in each factor of \eqref{eq_x7}, therefore, $|\mathcal M_a|\le k_a$.
 $M(\mathfrak c)$ is then maximized when all $\mathcal M_a$ are disjoint, and its
 maximal value is $k_1+\dots+k_r$. Then the leading (in $N$) contribution to the
 highest order terms becomes
 \begin{multline}
\label{eq_x14} \frac{N^{k_1+\dots+k_r}}{(k_1+\dots+k_r)!}
 \sum_{\begin{smallmatrix} \mathcal M_1\bigcup \dots \bigcup \mathcal M_r =\{1,\dots,k_1+\dots+k_r\} \\ \mathcal M_i \cap \mathcal M_j=\emptyset\text{ for all } i\ne j\end{smallmatrix}}
\\  \left( \prod_{a=1}^{r} \Sym_{\mathcal M_a=(i,j_1,\dots,j_{k_a-1})} \frac{
 (x_i)^{k_a} (\partial_i)}{(x_i-x_{j_1})(x_i-x_{j_2})\cdots (x_i-x_{j_{k_a-1}})}  \right)\\
  \Bigr[\ln(S_{{\rho_N}}(x_1,\dots,x_{k_1+\dots+k_r},1,\dots,1))\Bigl]_{x_1=\dots=x_{k_1+\dots+k_r}=1},
\end{multline}
Each of $r$ differential factors in \eqref{eq_x14} acts on its own group of
variables in $S_{{\rho_N}}(x_1,\dots,x_{k_1+\dots+k_r},1,\dots,1)$, and produces a
combination of partial derivatives in this group of variables. Therefore, together
they produce a partial derivative of $\ln(S_{\rho_N})$ parameterized by a partition
$\lambda$, whose parts are union of parts of $r$ partitions: of size $k_1$, of size
$k_2$, \dots, of size $k_r$. Therefore, $\lambda\le \mathbf k$ in the dominance
order as desired, and we arrive at a sum of the form \eqref{eq_M_equation}. The fact
that $\alpha_{\mathbf k} (\mathbf k)\ne 0$ again follows from its explicit
computation by \cite[Lemma 5.5]{BG}.
\end{proof}

\subsection{Proof of Theorem \ref{Theorem_main}}
\label{Section_proof_of_inversion}

We argue inductively, proving the statements of \eqref{exp_1}, \eqref{exp_2},
\eqref{exp_3} of Definition \ref{Def_CLT_appropriate} for all partial derivatives of
order $R$, given that for all orders up to $R-1$ it was already proven.

We know by Lemma \ref{Lemma_eigenrelation} and Definition \ref{Def_CLT} that
\begin{equation}
\label{eq_LLN}
 \mathcal P_{k}= N^{k+1} \mathfrak p (k) + o(N^{k+1}),
\end{equation}
\begin{equation}
\label{eq_CLT_cov}
 \K_{k_1,k_2}=\cov(k_1,k_2) N^{k_1+k_2} +
 o(N^{k_1+k_2}),
\end{equation}
\begin{equation}
\label{eq_CLT_cum}
 \K_{k_1,\dots,k_r}= o(N^{k_1+\dots+k_r}), \quad r>2
\end{equation}
as $N\to\infty$.

\bigskip

Note that for a given degree $R$, we have 1 equation \eqref{eq_LLN} with $k=R$ and
equations \eqref{eq_CLT_cov}, \eqref{eq_CLT_cum} enumerated by all partitions of $R$
with at least two rows. So altogether we have as many equations as there are
partitions of $R$. On the other hand, the number of $R$--degree derivatives of $\ln
S_{\rho_N}$ is also the same. Our idea is to show that this system of equations has
a unique solution as $N\to\infty$, which would imply \eqref{exp_1}, \eqref{exp_2},
\eqref{exp_3}. As we will see, the equations are linear, and have a non-degenerate
matrix of coefficients, which would imply the uniqueness of the solution.

We proceed to the detailed proof.

\bigskip

Assume that the formulas \eqref{exp_1}, \eqref{exp_2}, \eqref{exp_3} are true for
all
 partial derivatives up to order $R-1$ and we aim to prove them for partial derivatives of order
 $R$. Let $p(R)$ be the total number of Young diagrams with $R$ boxes. Divide the expression
 \eqref{eq_P_equation} for $k=R$ by $N^{R+1}$. We will view it as a linear expression in $p(R)$ variables
 \begin{equation}
 \label{eq_variables}
  \frac{(\partial_1)^R}{N}[\ln(S_{\rho_N}(\x))]_{x_1=\dots=x_N=1},
   \quad \partial_{\lambda}[\ln(S_{\rho_N}(\x))]_{x_1=\dots=x_N=1}, \quad
  \lambda:\, |\lambda|=R,\, \lambda\ne (R,0,\dots).
 \end{equation}
 What we know from Theorem \ref{Theorem_P} is that the coefficient of $\frac{(\partial_1)^R}{N}$ is bounded away from $0$ and
 from $\infty$, while all other coefficients are $O(N^{-1})$.

 Further, we take $p(R)-1$ expressions \eqref{eq_M_equation}, divide them by $N^{R}$ and again view as
 a linear expression in variables \eqref{eq_variables}. Then we know from Theorem \ref{Theorem_M} the following:
 \begin{itemize}
 \item The coefficient of $\lambda=\mathbf k$ is bounded away from $0$ and from $\infty$.
 \item The coefficients of $\lambda \prec \mathbf k$ are bounded from above
 \item The coefficient of $\frac{(\partial_1)^R}{N}$ is bounded from above
 \item All other coefficients are $O(N^{-1})$.
\end{itemize}
Altogether, if we denote the vector \eqref{eq_variables} through $\mathbf v$, then the linear
expressions can be brought into a form of a single $p(R)\times p(R)$ matrix $(A+N^{-1} B)$:
$$
 (A+N^{-1} B) \mathbf v,
$$
where $B$ might depend on $N$ (yet, remains uniformly bounded), while $A$ is independent of $N$ and
has the following structure: $A_{11}$ is non-zero, all other elements in the first row $A_{1i}$ are
zero, the elements of the first column can be arbitrary. The principle submatrix $A_{ij}$, $i,j>1$
is upper-triangular, if we order $\lambda$'s by lexicographic order so that the second row/column
corresponds to the smallest Young diagram $1^{R}$ and the last row/column corresponds to the
largest one, $(R-1,1,0,\dots)$.

It is immediate to see that the determinant of matrix $A$ is a product of its diagonal elements and
therefore is non-zero. We conclude that $A$ is invertible. Hence for large enough $N$ also
$(A+N^{-1} B)$ is invertible and all the matrix elements of $(A+N^{-1} B)^{-1}$ are bounded
(uniformly over $N>N_0$).

In these notations, the equations \eqref{eq_LLN}, \eqref{eq_CLT_cov},
\eqref{eq_CLT_cum} can be rewritten in the form
\begin{equation}
\label{eq_new_equation}
 (A+N^{-1} B) \mathbf v = \mathbf o^{R-1}+ \mathbf z + \eps,
\end{equation}
where $\mathbf z$ represents the limits of renormalized right--hand sides of
\eqref{eq_LLN}, \eqref{eq_CLT_cov}, \eqref{eq_CLT_cum}, $\eps$ is the remainder
$o(1)$, $N\to\infty$, in these right--hand sides and $o^{R-1}$ is the contribution
of partial derivatives of $\ln(S_{\rho_N})$ of orders up to $R-1$ in the expansion
of $\mathcal P_R$, $\K_{k_1,\dots,k_r}$ of Proposition \ref{Proposition_P_as_sum},
Corollary \ref{Corollary_M_as_sum}.

\bigskip

We next use the main technical result of \cite{BG_CLT}. It says that if we choose
all the partial derivatives of $\ln S_{\rho_N}$ up to order $R$ satisfying
\eqref{exp_1}, \eqref{exp_2}, \eqref{exp_3} and according to the formulas
\eqref{eq_LLN_1_formula}, \eqref{eq_CLT_1_formula} (assuming all $\p(k)$,
$\cov(k,m)$ to be known; such a choice exists by Lemma \ref{Lemma_as_invert}) , then
for such choice of $\mathbf v=\mathbf v_0$ we would have
\begin{equation}
\label{eq_old_equation}
 (A+N^{-1} B) \mathbf v_0 = \mathbf o^{R-1}+ \mathbf z + \eps',
\end{equation}
where $\eps'$ is another remainder which tends to $0$ as $N\to\infty$; in other
words, the asymptotic expansions \eqref{eq_LLN}, \eqref{eq_CLT_cov},
\eqref{eq_CLT_cum}  will be satisfied. Now we subtract \eqref{eq_new_equation} and
\eqref{eq_old_equation} to get
$$
 (A+N^{-1} B) (\mathbf v- \mathbf v_0)=\eps-\eps'.
$$
It remains to use the fact that $(A+N^{-1} B)^{-1}$ is uniformly bounded for $N>N_0$
to conclude that $\mathbf v - \mathbf v_0$ tends to $0$ as $N\to\infty$. Therefore,
$\mathbf v$ satisfies \eqref{exp_1}, \eqref{exp_2}, \eqref{exp_3}, as desired.

\section{Multi-dimensional CLT}

\label{Section_multi}

In this section we extend Theorem \ref{Theorem_main} to a setting of several random
signatures.

\subsection{Formulation}

Fix $H\ge 1$ and $H$--tuple of positive integers $N_1,\dots,N_H$.

For a probability measure ${\tau}$ on $\prod_{h=1}^H \GT_{N_h}$ we define its
$H$--dimensional Schur generating function, through
\begin{multline}
 S_{\tau}(x_{1}^1,\dots,x_{N_1}^1;\, \dots;\, x_{1}^H,\dots,x_{N_H}^H)\\=
 \sum_{\lambda^{1}\in \GT_{N_1},\dots,\lambda^{H}\in \GT_{N_H}} {\tau}\bigl(\lambda^{1},\dots,\lambda^{H}\bigr)
 \prod_{h=1}^H \frac{s_{\lambda^{h}}(x_{1}^h,\dots,x_{N_h}^h)}{s_{\lambda^{h}}(1^{N_h})}.
\end{multline}

\begin{definition} We call a probability measure $\tau$ on $\prod_{h=1}^H \GT_{N_h}$ \emph{smooth}, if
there exists $r>1$ such that
$$
 \sum_{\lambda^{1}\in\GT_{N_1},\dots,\lambda^{H}\in \GT_{N_H}} \tau(\lambda^{1},\dots,\lambda^{H}) \cdot r^{\sum_{h=1}^H \sum_{i=1}^{N_h} |\lambda_i^h|}
 <\infty.
$$
\end{definition}

An immediate analogue of Lemma \ref{Lemma_smooth_SGF} holds and has the same proof.
\begin{lemma}
 If $\tau$ is a smooth measure on $\prod_{h=1}^H \GT_{N_h}$, then its $H$--dimensional Schur--generating function $S_{\rho_N}$ is
 uniformly convergent in a neighborhood of $1$.
\end{lemma}

Let $L$ be a large parameter and suppose that $N_h=N_h(L)$ are such that
$$
 \lim_{L\to\infty} \frac{N_h}{L}=\mathcal N_h>0,\quad h=1,\dots,H.
$$
For each $L=1,2,\dots$, let $\rho_L$ be a smooth probability measure on
$\prod_{h=1}^H \GT_{N_h}$, and let $\lambda^{1},\dots,\lambda^{H}$ be the
corresponding random element of $\prod_{h=1}^H \GT_{N_h}$. Define
$\{p_{k;h}^N\}_{k=1,2,\dots;\, h=1,\dots,H}$ to be a countable collection of random
variables via
$$
 p_{k,h}^L= \sum_{k=1}^{N_h} \left(\frac{\lambda_i^{h}+N_h-i}{N_h}\right)^k.
$$

\begin{definition} \label{Definition_CLT_multi}
 We say that smooth measures $\rho_L$ satisfy a CLT as $L\to\infty$, if there exist
 two countable collection of numbers $\mathfrak p(k,h)$, $\mathfrak {cov}(k,h;m,f)$,
 $k,m=1,2,\dots;$ $h,f=1,\dots,H$, such that
\begin{enumerate}[label=(\Alph*)]
\item \label{ass_multi_1} For each $k=1,2,\dots$, $h=1,\dots,H$
$$
 \lim\limits_{L\to\infty} \frac{1}{N_h} \E [p_{k,h}^L] = {\mathfrak p}(k,h),
$$
\item \label{ass_multi_2} For each $k,m=1,2,\dots$; $h,f=1,\dots,H$,
$$
 \lim\limits_{L\to\infty}\Bigl( \E [p_{k,h}^L\, p_{m,f}^L] -\E[p_{k,h}^L]\E [p_{m,f}^L] \Bigr)= \mathfrak {cov}(k,h;m,f),
$$
\item \label{ass_multi_3}For each $r>2$ and any collection of pairs
$(k_1,h_1)$, \dots, $(k_r,h_r)$,
$$\lim_{L\to\infty} \kappa_{ (k_1,h_1),\dots,(k_r,h_r)}\bigl( p^L_{k,h}, h=1,\dots,H,\, k=1,2,\dots\bigr)=0.
$$
\end{enumerate}

\end{definition}

\begin{definition}
\label{Def_appropriate_multi}
 We say that smooth measures $\rho_L$ are CLT--appropriate as $L\to\infty$, if there
 exist two countable collections of numbers $\c_{k,h}$, $\d_{k,h;m,f}$,
 $k,m=1,2,\dots$; $h,f=1,\dots,H$,
 such that
\begin{enumerate}
\item \label{exp_multi_1} For each $k=1,2,\dots$:
$$\lim\limits_{L\to\infty} \frac{1}{N_h} \left(\frac{\partial}{\partial x_{i}^h}\right)^k \ln( S_{\rho_L})
\Bigl|_{x_{1}^1=\dots=x_{N_H}^H=1}= \c_{k,h},$$
\item For each two pairs $(k,h)\ne(m,f)$:\label{exp_multi_2}
$$\lim\limits_{L\to\infty}
\left(\frac{\partial}{\partial x_{i}^h}\right)^k \left(\frac{\partial}{\partial
x_{j}^f}\right)^m  \ln( S_{\rho_L}) \Bigl|_{x_{1}^1=\dots=x_{N_H}^H=1}=
\d_{k,h;m,f},
$$
\item \label{exp_multi_3} For each set of pairs $\{(i_1,h_1),\dots,(i_s,h_s)\}$ with at least
three distinct elements:
$$\lim\limits_{L\to\infty} \frac{\partial}{\partial x_{i_1}^{h_1}}
\dots \frac{\partial}{\partial x_{i_s}^{h_s}}  \ln( S_{\rho_L})
\Bigl|_{x_{1}^1=\dots=x_{N_H}^H=1}= 0.$$
\end{enumerate}
 \end{definition}

\begin{theorem} \label{Theorem_main_multi}
Smooth measures $\rho_L$ satisfy CLT as $L\to\infty$ if and only if they are
CLT--appropriate. The numbers $\p(k,h)$, $\cov(k,h;m,f)$, $k,m=1,\dots;$
$h,f=1,\dots,H$ are polynomials in $\c_{k,h}$, $\d_{k,h;m,f}$, $k,m=1,\dots$,
$h,f=1,\dots,H$ and can be computed through the following formulas:

\begin{equation}
\label{eq_LLN_multi_formula} \mathfrak p(k,h) = [z^{-1}] \frac{1}{(k+1)(z+1)} \left(
(1+z) \left(\frac1z  +  \sum_{a=1}^\infty \frac{\c_{a,h} z^{a-1} }{(a-1)!}\right)
\right)^{k+1},
\end{equation}
\begin{multline}
\label{eq_CLT_multi_formula} \mathfrak {cov}(k_1,h_1;k_2,h_2)\\ =
  [z^{-1} w^{-1}]
 \left( (1+z)
\left(\frac1z  +  \sum_{a=1}^\infty \frac{\c_{a,h_1} z^{a-1}
}{(a-1)!}\right)\right)^{k_1} \, \left( (1+w)
\left(\frac1w  +  \sum_{a=1}^\infty \frac{\c_{a,h_2} w^{a-1} }{(a-1)!}\right) \right)^{k_2} \\
\times \left(  \delta_{h_1=h_2}\left(\sum_{a=0}^{\infty}
\frac{z^a}{w^{1+a}}\right)^2 +\sum_{a,b=1}^{\infty} \frac{\d_{a,h_1;b,h_2}}{(a-1)!
(b-1)!} z^{a-1} w^{b-1} \right),
\end{multline}
where $[z^{-1}]$ and $[z^{-1} w^{-1}]$ stay for the coefficients of $z^{-1}$ and
$z^{-1}w^{-1}$, respectively, in the Laurent power series given afterwards.
\end{theorem}

The proof of Theorem \ref{Theorem_main_multi} is similar to Theorem
\ref{Theorem_main} and we present it in the next sections. One complication is that
the results of \cite{BG_CLT} do not cover the generality we need here, and therefore
we have to present their extension.

\subsection{Theorem \ref{Theorem_main_multi}, part 1: Being appropriate implies CLT}
\label{Section_Th_Multi_1} In this section we show one direction of Theorem
\ref{Theorem_main_multi} by proving that CLT--appropriate measures satisfy CLT.
Although this is an extension of \cite[Theorem 2.8]{BG_CLT}, yet  we present a
mildly different proof.

We will shorten the notations by assuming throughout the proof that
$L=N_1=N_2=\dots=N_H=N$. In the general case the argument is the same.

\begin{proposition} Suppose that smooth measures $\rho_N$ on $\GT_N^H$ are
CLT--appropriate, then they satisfy condition \ref{ass_multi_1} of Definition
\ref{Definition_CLT_multi}.
\end{proposition}
\begin{proof}
 After we set in the $H$--dimensional Schur generating function of $\rho_N$ all $x^j_i$ with $j\ne h$ equal to $1$,
 then we get the (one--dimensional) Schur generating function of the marginal
 $\lambda^h$. Therefore, we can apply Theorem \ref{Theorem_main}.
\end{proof}

\begin{proposition} \label{Proposition_multi_Gauss_direct} Suppose that smooth measures $\rho_N$ on $\GT_N^H$ are
CLT--appropriate, then they satisfy condition \ref{ass_multi_3} of Definition
\ref{Definition_CLT_multi}.
\end{proposition}

The proof of Proposition \ref{Proposition_multi_Gauss_direct} is split into several
lemmas.

For $h=1,\dots,H$, we let $\x^h$ be the set of variables $(x_1^h,\dots,x_N^h)$ and
define
$$
 V(\x^h)=\prod_{1\le i<j \le N} (x_i^h-x_j^h).
$$
Introduce the following differential expression in $N$ variables:
$$
 \D_{k,h} = V(\x^h)^{-1} \left(\sum_{i=1}^N \left({x_i^h} \frac{\partial}{\partial
 x_i^h} \right)^k \right) V(\x^h).
$$
We will generally apply this operator to $S_{\rho_N}$ to compute the moments of
${\rho_N}$. We use the notation
\begin{equation}
\label{eq_multi_uncentered_moments}
 \mathcal P_{k,h}:=\D_{k,h}[S_{\rho_N}] \bigr|_{x_1=\dots=x_N=1}
\end{equation}

We also need a version for the cumulants. Using the notation
\eqref{eq_Set_partition}, for a collection of pairs $(k_1,h_1),\dots,(k_r,h_r)$ we
set
\begin{equation}
\label{eq_multi_cumulants}
 \K_{k_1,h_1;\dots;k_r,h_r}:=
 \sum_{\left\{ U_1, \cdots, U_t \right\} \in \Theta_r} (-1)^{t-1}
(t-1)! \prod_{i=1}^t \left[ \left( \prod_{j \in U_i} \mathcal D_{k_j,h_j} \right)
[S_{\rho_N}] \Bigr|_{x_1^1=\dots=x_N^H=1} \right].
\end{equation}

\begin{lemma} We have
$$
\mathcal P_{k,h}= \E \left[ N^k p^N_{k,h} \right], \quad
\K_{k_1,h_1;\dots;k_r,h_r}=\kappa_{(k_1,h_1),\dots(k_r,h_r)} \bigl( N^k p^N_{k,h},\,
h=1,\dots,H,\, k=1,2,\dots\bigr)
$$
\end{lemma}
\begin{proof}
 This is a corollary of the eigenrelarion $\D_{k,h} (s_\mu(\x^h))=\sum_{i=1}^N
 (\mu_i+N-i)^k s_\mu(\x^h)$.
\end{proof}

We start analyzing $\K_{k_1,h_1;\dots;k_r,h_r}$. By its very definition, the
expression is a large sum, where each term is obtained by several applications of
the operators $\D_{k,h}$ to $S_{\rho_N}$.

 As in the previous sections, when we need to differentiate $S_{\rho_N}$, we
use the formula
\begin{equation}
\label{eq_S_diff}
 \partial_i [S_{\rho_N}] = S_{\rho_N} \partial_i [\ln S_{\rho_N}],
\end{equation}
Therefore, $S_{\rho_N}$ is factored away in each step, and disappears when we plug
in all $x_i^h$ equal to $1$ in the end.

In this way, if we act with $\D_{k,h}$ only once, then we can write
\begin{equation}
\label{eq_act_once}
 \D_{k,h} S_{\rho_N} = A_{k,h} \cdot S_{\rho_N},
\end{equation}
where $A_{k,h}$ is a (symmetric in each $N$ variables $\x^h$, $h=1,\dots,H$) sum
involving $\ln S_{\rho_N}$, factors of the form $\frac{x^h_i}{x^h_i-x^h_j}$, and
various partial derivatives with respect to variables $x_i^h$, $i=1,\dots,N$.

Further, if we need to act with two operators $\D_{k_1,h_1} \cdot \D_{k_2,h_2}$,
then we first act with $\D_{k_2,h_2}$ as in \eqref{eq_act_once}, and then act with
$D_{k_1,h_1}$ on the result and expand into an even larger sum by using the Leibnitz
rule. There are two cases here: either some of the derivations from $D_{k_1,h_1}$
are applied to $A_{k_2,h_2}$, or all of them are applied to $S_{\rho_N}$, and then
the factor $A_{k_1,h_1}$ appears. It will be important for us to distinguish these
two cases, and we now develop a book-keeping for them.

In general situation, we apply sequentially $r$  operators $\D_{k_i,h_i}$ (starting
from $i=r$ and ending with $i=1$), use Leibnitz rule and expand the result into a
huge sum. Let us explain, what we mean by a \emph{term} in such sum.

As in Section \ref{Section_Highest_order}, we can write the operator $\D_{k,h}$ as a
sum of the differential operators of the form
\begin{equation}
\label{eq_D_k_expansion_multi}
  \frac{(x_i^h)^q  }{\prod_{j\in A}(x_i^h-x_j^h)} \left(\frac{\partial}{\partial_i^h}\right)^n,
\end{equation}
where $A$ is a finite set of indices. Expression \eqref{eq_D_k_expansion_multi}
involves $|A|+1$ different variables; what is important is that this number does not
grow with $N$. Since each $\D_{k,h}$ is actually symmetric in variables $\x^h$, the
expression \eqref{eq_D_k_expansion_multi} shows up together with all possible
permutations of the involved variables. Put it otherwise, $\D_{k,h}$ is a sum of
differential operators of the form
\begin{equation}
\label{eq_D_k_expansion_multi_sym}
  \Sym_{\{i\}\bigcup A} \frac{(x_i^h)^q  }{\prod_{j\in A}(x_i^h-x_j^h)}
  \left(\frac{\partial}{\partial_i^h}\right)^n,
\end{equation}
where $\Sym$ is the symmetrization operation, as in Section
\ref{Section_Highest_order}. Now take two functions $f$ and $g$, which are symmetric
in variables $\{x_j^h\}, j\in \{i\}\bigcup A$. The Leibnitz rule for
\eqref{eq_D_k_expansion_multi_sym} takes the form:
\begin{multline}
\label{eq_Symmetric_Leibnitz}
  \left(\Sym_{\{i\}\bigcup A} \frac{(x_i^h)^q  }{\prod_{j\in A}(x_i^h-x_j^h)}
  \left(\frac{\partial}{\partial_i^h}\right)^n\right) [fg]\\=\sum_{m=0}^n \Sym_{\{i\}\bigcup A}\left(
\frac{(x_i^h)^q  }{\prod_{j\in A}(x_i^h-x_j^h)} \cdot {n \choose m} \cdot
   \left( \frac{\partial}{\partial_i^h}\right)^m f  \cdot  \left(\frac{\partial}{\partial_i^h} \right)^{n-m} g \right),
\end{multline}
Note that we could also symmetrize over variables in an arbitrary larger set
$B\supset\{i\}\bigcup A$ instead of $\{i\}\bigcup A$ in
\eqref{eq_D_k_expansion_multi_sym}, \eqref{eq_Symmetric_Leibnitz}. Therefore, using
\eqref{eq_S_diff} for the derivations, the application of  $D_{k,h}  D_{k',h}$  to
$S_{\rho_N}$ reduces to a sum of the expressions of the form
\begin{multline}
\label{eq_Leibnitz_computations}
 \left(\Sym_{B} \frac{(x_i^h)^q  }{\prod_{j\in A}(x_i^h-x_j^h)}
  \left(\frac{\partial}{\partial_i^h}\right)^n\right)  \left(\Sym_{B} \frac{(x_{i'}^h)^{q'}  }{\prod_{j'\in A'}(x_{i'}^h-x_{j'}^h)}
  \left(\frac{\partial}{\partial_{i'}^h}\right)^{n'}\right) [S_{\rho_N}]
  \\ =  \left(\Sym_{B} \frac{(x_i^h)^q  }{\prod_{j\in A}(x_i^h-x_j^h)}
  \left(\frac{\partial}{\partial_i^h}\right)^n\right) \left[ \left(\Sym_{B} \frac{(x_{i'}^h)^{q'}  }{\prod_{j'\in A'}(x_{i'}^h-x_{j'}^h)}
   \frac{\left(\frac{\partial}{\partial_{i'}^h}\right)^{n'} S_{\rho_N}}{S_{\rho_N}} \right) \cdot S_{\rho_N} \right]
   \\ =\sum_{m=0}^n {n \choose m} \Sym_{B}\Biggl( \frac{(x_i^h)^q  }{\prod_{j\in A}(x_i^h-x_j^h)} \cdot  \left(\frac{\partial}{\partial_i^h}\right)^m \left(\Sym_{B} \frac{(x_{i'}^h)^{q'}  }{\prod_{j'\in A'}(x_{i'}^h-x_{j'}^h)}
   \frac{\left(\frac{\partial}{\partial_{i'}^h}\right)^n S_{\rho_N}}{S_{\rho_N}} \right)
   \\\cdot \left[\frac{\left(\frac{\partial}{\partial_i^h}\right)^{n-m} S_{\rho_N}}{S_{\rho_N}}\right] \cdot S_{\rho_N}
   \Biggr)
\end{multline}
where $B=\{i\}\bigcup A\bigcup \{i'\}\bigcup A'$. We observe that $D_{k,h} D_{k',h}
[S_{\rho_N}]$ therefore expanded into a large sum over the possible choices of $q$,
$q'$, $A$, $A'$, $n$, $n'$, and $m$. Further, if we study $D_{k,h}
D_{k',h'}[S_{\rho_N}]$ with $h\ne h'$, then the computation is very similar, only we
do not need to introduce additional symmetrization over $B$ --- it is enough to
symmetrize over sets $\{i\}\bigcup A$ and $\{i'\}\bigcup A'$, as the operators act
in disjoint variables.

Summing up, all we did was the application of the Leibnitz rule, with a technical
detail that we must symmetrize at every step --- as we saw in Section
\ref{Section_Highest_order}, the symmetrization is important to ensure the
well--defined values when we plug in all the variables equal to $1$.

In general, when we apply more than two operators $D_{k,h}$ to $S_{\rho_N}$, we do
the same: Leibnitz rule expansion with additional symmetrization on each step.
That's how we arrive at terms of the expansion. Our next aim is to assign to each
such term a certain graph.

\begin{definition}
 A \emph{coding graph} is a graph with $r$ vertices $V_1$, \dots, $V_r$.
 The vertex $V_i$ corresponds to the pair $(k_i,h_i)$ of
 $\K_{k_1,h_1;\dots;k_r,h_r}$. There are two types of edges: single and bold. The
 edges encode the derivations in a term of $\K_{k_1,h_1;\dots;k_r,h_r}$. We say that
 each edge is oriented from smaller index to larger index.
\end{definition}

Let us describe in details how a coding graph is constructed from a term. The
vertices are prefixed, while the edges show the dependencies in derivations: they
are constructed by an inductive in $r$ procedure. When $r=1$, then we have a
one-vertex graph without any edges.

There is one important property, which is preserved in the steps of the
construction. Suppose that a term of the form $T\cdot S_{\rho_N}$ corresponds to a
graph $G$. If $G$ can be split into two disconnected parts $G=G_1\bigsqcup G_2$,
with $G_1$ corresponding to vertices $\{(k_i,h_i)\}_{i\in I_1}$ and $G_2$
corresponding to vertices $\{(k_i,h_i)\}_{i\in I_2}$, then $T$ is factorized,
$T=T_1\cdot T_2$. Moreover, $T_1 S_{\rho_N}$ is one of the terms in the expansion of
$$
 \left[\prod_{i\in I_1} \D_{k_i,h_i}\right]\left[ S_{\rho_N}\right],
$$
and $T_2 S_{\rho_N}$ is one of the terms in the expansion of
$$
 \left[\prod_{i\in I_2} \D_{k_i,h_i}\right]\left[ S_{\rho_N}\right].
$$
Similarly, if $G$ is split into $q$ disconnected parts, then $T$ should be split
into $q$ factors.

Clearly, this property tautologically holds when $r=1$.

We proceed to the inductive description of the construction of the graph. After
$\D_{k_{i+1},h_{i+1}},\dots,\D_{k_r,h_r}$ were applied, we have an expression of the
form
$$
 F(\x^1,\dots,\x^H) S_{\rho_N},
$$
where $F$ is a (symmetric in each $N$ variables $\x^h$, $h=1,\dots,H$) expression,
which is a sum involving $\ln S_{\rho_N}$, factors of the form
$\frac{1}{x^h_i-x^h_j}$, and various partial derivatives. If we argue inductively,
then each term in $F$ already has a graph on vertices $i+1,\dots,r$ assigned to it.
So take one term, $T$, corresponding to a graph $G$ and apply $\D_{k_i,h_i}$ to $T
\cdot S_{\rho_N}$. We assume by induction hypothesis that $T$ has a factorized form,
with each term corresponding to a connected component of $G$.

The application of $\D_{k_i,h_i}$ in the form of the sum of terms
\eqref{eq_D_k_expansion_multi_sym} involves differentiations in variables from the
set $\x^{h_i}$. By the Leibnitz rule, as in the example of
\eqref{eq_Leibnitz_computations}, each derivative is applied to either of:
\begin{enumerate}
\item $S_{\rho_N}$
\item One of the factors of $T$ corresponding to a connected  component $\mathcal C$ of $G$.
\end{enumerate}
The first case does not lead to appearance of new edges in the constructed graph. On
the other hand, in the second case, we take $G$ and add to it edges joining the
$i$th vertex $(k_i,h_i)$ with each vertex of the connected component $\mathcal C$.
It is immediate to see that after this is done, the factorized form with respect to
a new graph on vertices $\{i,i+1,\dots,r\}$ still holds.

It remains to specify whether new edges are single or bold. For each connected
component $\mathcal C$, the edges (if any) joining the $i$th vertex $(k_i,h_i)$ with
this connected component are of the same type. The single edges arise when we
differentiate an expression in a set of variables with respect to a new variable.
For instance, if we have $\partial_1 \ln S_{\rho_N}$ and then apply $\partial_2$ to
it. The bold edges arise when we instead differentiate an expression in a set of
variables with respect to a variable already present in it. For instance, if we
apply $\partial_1$ to $\partial_1 \ln S_{\rho_N}$ or to $\frac{\partial_2\ln
S_{\rho_N}}{x_1-x_2}$.

This finishes the description of the inductive procedure. Let us remark, that
although a graph describes some features of the corresponding term, but this
description is not complete, i.e.\ there might be lots of different terms
corresponding to the same graphs. In particular, sometimes these terms cancel out.

\begin{lemma} \label{Lemma_only_connected}
 After we expand $\K_{k_1,h_1;\dots;k_r,h_r}$ into a sum using
 the above procedure and collect the terms, only those corresponding to connected
 graphs give non-zero contribution.
\end{lemma}
\begin{proof}
 Indeed, take any term corresponding to a graph with connected components
 $V_1,\dots,V_q$. Then in the sum \eqref{eq_multi_cumulants}, we can get such a
 graph from any set--partition $U_1,\dots,U_t$, whose sets are unions of the sets
 $V_1,\dots,V_q$. Therefore, the total coefficient of the terms corresponding to
 this graph is
 $$
  \sum_{\begin{smallmatrix} \text{set partitions }S_1,\dots,S_m \\ \text{of
  }\{1,\dots,q\}\end{smallmatrix}} (-1)^{m-1} (m-1)!.
 $$
 If $q>1$, then this sum is zero, as this is an expression for the $q$--th cumulant
 of the random variable identically equal to $1$.
\end{proof}

\begin{proof}[Proof of Proposition \ref{Proposition_multi_Gauss_direct}]
We fix $r>2$ and prove that $\K_{k_1,h_1;\dots;k_r,h_r}=o(N^{\sum_{i=1}^r k_r})$ as
$N\to\infty$. We proceed by taking a term corresponding to a connected graph $G$ and
analyze its contribution. The estimate $o(N^{\sum_{i=1}^r k_r})$ is obtained by
combining a combinatorial enumeration of the number of terms corresponding to $G$
with bounds on partial derivatives in Definition \ref{Def_appropriate_multi}.
 We remark that due to the arguments of Section \ref{Section_Highest_order},
 we already know that $\K_{k_1,h_1;\dots;k_r,h_r}$ is a combination of products of
 partial derivatives of $\ln(S_{\rho_N})$.

Note that $G$ has at least $r-1$ edges and at least one leaf (as an oriented graph,
i.e.\ leaf is a vertex with no outgoing edges).

Also, let us assign to a term a new number $g$, which is the number of times that we
differentiated $S_{\rho_N}$ (i.e.\ we had a term of the form $T\cdot S_{\rho_N}$ and
in the next derivation we differentiated $S_{\rho_N}$ rather than $T$). This number
is split between the vertices, $g=\sum_{i=1}^r g_i$, where $g_i$ means the number of
times that we differentiated $S_{\rho_N}$ when applying $D_{k_i,h_i}$.

Note that when applying $\D_{k_r,h_r}$, we need to differentiate $S_{\rho_N}$ at
least once. Otherwise, after symmetrization we get ${\rm const} \cdot S_{\rho_N}$
and on the next differentiation, this constant would lead to vanishing contribution.
Therefore, $g_r\ge 1$. The same argument shows that if $(k_m,h_m)$ is a leaf of $G$,
then $g_m\ge 1$.

Let us now consider the most \emph{basic} case of a graph $G$ with all single edges
and $g=1$. Then the single derivation of $S_{\rho_N}$ needs then to happen at a
single leaf of this graph, which is the last vertex, $(k_r,h_r)$. Since there are no
other leaves (and due to the way the graph is being constructed), $G$ is then the
complete graph on $r$ vertices. In this case, the term corresponding to this graph
involves at most $\sum_{i=1}^r k_i$ different variables $x^h_i$. Indeed, the leaf
gives at most $k_r$ different variables, as we need to differentiate $\ln
S_{\rho_N}$ once, and remaining differentiations can be applied to $\prod_{a<b}
(x^{h_r}_a-x^{h_r}_b)$ leading to a new variable each. Each other vertex
$(k_i,h_i)$, $i<r$ leads to at most $k_i$ new variables appearing, we need to
differentiate something other that $\prod_{a<b}(x_a^{h_i}-x_b^{h_i})$ at least once
(as there should be an outgoing edge from this vertex). Choosing these $\sum_{i=1}^r
k_i$ variables in all possible ways leads to a combinatorial factor
$O(N^{{\sum_{i=1}^r k_i}})$. Further, as all edges are single, and there are at
least $2$ of them (since $r>2$), we will differentiate in at least three distinct
variables. The conclusion is that we have $O(N^{{\sum_{i=1}^r k_i}})$ terms, each of
which is a combination of partial derivatives of case \ref{ass_multi_3} in
Definition \ref{Def_appropriate_multi}. Therefore, the total contribution of such
graph is $o(N^{\sum_{i=1}^r k_i})$, as desired.

\begin{figure}[h]
\begin{center}
 {\scalebox{1.2}{\includegraphics{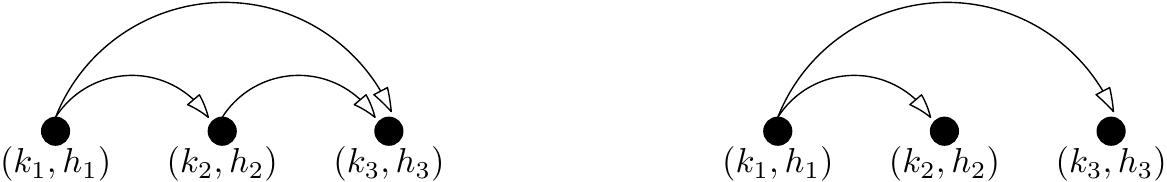}}}
 \caption{When $r=3$, two different connected graphs with all single edges are
 possible in our construction: with one leaf $(k_3,h_3)$, as in the left panel, and with two leaves
 $(k_2,h_2)$ and $(k_3,h_3)$, as in the right panel.
 \label{Figure_graphs}}
\end{center}
\end{figure}

A bit more general \emph{basic} graph has $m$ leaves and $g=m$, with each leaf
corresponding to $g_i=1$, i.e.\ single derivation of $S_{\rho_N}$, cf.\ Figure
\ref{Figure_graphs}. We still assume that all edges are single. Suppose that $m>1$,
as $m=1$ has been already discussed above. The combinatorial coefficient in this
case is $O(N^{\sum_{i=1}^r k_i}-1)$. Indeed, we again count the different variables
as in $g=1$ case. The new feature is that we need at some point (in the inductive
construction of the graph $G$) to joint a connected component of the vertex
$(k_r,h_r)$ with connected component of each other leaf. When we connect two
components, we differentiate the factors coming from them with respect to the same
variable, and therefore, the total number of variables decreases (at least) by $1$.

On the other hand, each term is a product of partial derivatives of cases
\ref{ass_multi_2} and \ref{ass_multi_3} in Definition \ref{Def_appropriate_multi},
with a partial derivative of case \ref{ass_multi_3}. Therefore, the total
contribution is again $o(N^{\sum_{i=1}^r k_i})$.

More general graphs are obtained from the basic ones by two operations: turning a
single edge into a bold edge, or increasing one of $g_i$ by $1$. Let us explain that
each operation does not change the conclusion of total $o(N^{\sum_{i=1}^r k_i})$
contribution. If we increase $g_i$, then a new partial derivative of $\ln
S_{\rho_N}$ appears; a priori it might be of order $N$, if we fit into the case
\ref{ass_multi_1} of Definition \ref{Definition_CLT_multi}. However, increasing
$g_i$ also means that the total number of variables in our term is decreased by $1$.
Therefore, the combinatorial factor is divided by $N$, compensating the partial
derivative.

Similarly, if we turn a single edge into a bold edge, then we can jump from case
\ref{ass_multi_3} to case \ref{ass_multi_2}, or from case \ref{ass_multi_2} to case
\ref{ass_multi_1} of Definition \ref{Definition_CLT_multi}, thus increasing the
order by $N$. Simultaneously, the total number of variables decreases by $1$, and
therefore, the combinatorial factor is divided by $N$. We conclude that the
contribution is still $o(N^{\sum_{i=1}^r k_i})$.
\end{proof}

\begin{proposition} \label{Proposition_multi_covariance_direct} Suppose that smooth measures $\rho_N$ on $\GT_N^H$ are
CLT--appropriate, then they satisfy condition \ref{ass_multi_2} of Definition
\ref{Definition_CLT_multi}.
\end{proposition}
\begin{proof}
 The case $h_1=h_2$ is a particular case of Theorem \ref{Theorem_main}, and
 therefore we do not deal with it and assume $h_1\ne h_2$. Lemma
 \ref{Lemma_only_connected} implies that we need to deal with terms corresponding to
 the connected graph on two vertices. Note that since $h_1\ne h_2$, the edge joining
 these vertices can not be bold (as the operators act in distinct sets of variables).
 For the graph with single edge, the computation still needs efforts, but it is the same
  as in \cite[Sections 6.1, 6.2]{BG_CLT}.
 (the bold edge produced the term involving $1/(z-w)^2$ there, which no longer appears in our last case).
\end{proof}

\subsection{Theorem \ref{Theorem_main_multi}, Part 2: CLT implies being appropriate}
\label{Section_Th_Multi_2} In this section we prove the second implication of
Theorem \ref{Theorem_main_multi} by showing that CLT implies being CLT--appropriate.
The proof shares all key ideas with that of Theorem \ref{Theorem_main}, so we omit
some details. As in Section \ref{Section_Th_Multi_1} we restrict ourselves to the
case $L=N_1=\dots=N_H=N$ in order to shorten the notations.

We argue inductively, proving the statements of \eqref{exp_multi_1},
\eqref{exp_multi_2}, \eqref{exp_multi_3} of Definition \ref{Def_appropriate_multi}
for all partial derivatives of order $R$, given that for all orders up to $R-1$ it
was already proven.

What we know is that
\begin{equation}
\label{eq_LLN_multi}
 \mathcal P_{k,h}= N^{k+1} \mathfrak p (k,h) + o(N^{k+1}),
\end{equation}
\begin{equation}
\label{eq_CLT_cov_multi}
 \K_{(k_1,h_1),(k_2,h_2)}=\cov(k_1,h_1;\, k_2,h_2) N^{k_1+k_2} +
 o(N^{k_1+k_2}),
\end{equation}
\begin{equation}
\label{eq_CLT_cum_multi}
 \K_{(k_1,h_1),\dots,(k_r,h_r)}= o(N^{k_1+\dots+k_r}), \quad r>2
\end{equation}
as $N\to\infty$. As in the proof of Theorem \ref{Theorem_main}, the idea of what
follows is to treat \eqref{eq_LLN_multi}, \eqref{eq_CLT_cov_multi},
\eqref{eq_CLT_cum_multi} as a system of linear equations on partial derivatives.

\bigskip

Let us take an $r$--tuple $(k_1,h_1),\dots,(k_r,h_r)$ and assign to it $H$
partitions $\mu^{(1)},\dots,\mu^{(H)}$. The parts of $\mu^{(h)}$ are all $k_i$ such
that $h_i=h$. Note that $r$--tuples corresponding to the same partitions
$\mu^{(1)},\dots,\mu^{(H)}$ give essentially the same
$\K_{(k_1,h_1),\dots,(k_r,h_r)}$ -- the only freedom is the order of applying the
operators $D_{h,r}$, which is irrelevant, as the operators commute; we can agree on
some order, e.g.\ by imposing the condition that $k_i$ and $h_i$ weakly grow with
$i$. Set $R_h=|\mu^{(h)}|=\sum_{i=1}^{\infty} \mu^{(h)}_i$, $R=\sum_{h=1}^H R_h$ is
the \emph{order} of $\mu$ or of $(k_1,h_1),\dots,(k_r,h_r)$.

We can similarly encode various partial derivatives of $\ln(S_{\rho_N})$ by
$H$--tuples of partitions. Given such $H$--tuple
$\lambda=(\lambda^{(1)},\lambda^{(2)},\dots,\lambda^{(H)})$, we consider the partial
derivative
$$
 \partial_{\lambda}[\ln S_{\rho_N}]:=\left(\prod_{h=1}^{H} \prod_{i=1}^{\infty} \left(\frac{\partial}{\partial
 x^h_i}\right)^{\lambda^{(h)}_i}\right) [\ln S_{\rho_N}]\Bigr|_{x_1^1=\dots=x^H_N=1}.
$$
Due to symmetry of $\ln(S_{\rho_N})$, all possible partial derivatives are encoded
in this way. $\sum_{h=1}^H \sum_{i=1}^\infty \lambda^{(H)}_i$ is the order of the
partial derivative.

\begin{theorem}
\label{Theorem_multi_expansion}
 Take an $r$--tuple $(k_1,h_1),\dots,(k_r,h_r)$ corresponding to $H$ partitions
 $\mu=(\mu^{(1)},\dots,\mu^{(H)})$ of order $R$. Then
 $\K_{(k_1,h_1),\dots,(k_r,h_r)}$ or $\mathcal P_{k_1,h_1}$ (if $r=1$) is a sum of partial
 derivatives of $\ln S_{\rho_N}$ of order up to $R$. The highest (i.e.\ $R$th) order terms
 are given by the following expression
\begin{equation}
\label{eq_multi_expansion}
 N^{R} \sum_{\lambda \le \mu} \alpha_{\mu}(\lambda) \partial_{\lambda} [\ln S_{\rho_N}]
 + O(N^{R-1}),
\end{equation}
where the summation goes over all $H$--tuples of partitions
$\lambda=(\lambda^{(1)},\dots,\lambda^{(H)})$ such that
$|\lambda^{(h)}|=|\mu^{(h)}|$ and $\lambda^{(h)}\preceq \mu^{(h)}$ in the dominance
order for each $h=1,\dots,H$. $\alpha_{\mu}(\lambda)$ are some coefficients
(independent of $N$) and $O(N^{R-1})$ is a finite linear combination of other
partial derivatives of $\ln S_{\rho_N}$ of orders up to $R$ and with all
coefficients being smaller in the absolute value than ${\rm const} \cdot N^{R-1}$.
Finally, $\alpha_{\mu}(\mu)\ne 0$.
\end{theorem}
\begin{proof}
 The proof is the same as that of Theorems \ref{Theorem_P}, \ref{Theorem_M}.
\end{proof}

We now finish the proof of Theorem \ref{Theorem_main_multi} by showing that if
$\rho_N$ satisfy CLT as $N\to\infty$ (recall that we assumed $N=L$), then the
conditions \eqref{exp_multi_1}, \eqref{exp_multi_2}, \eqref{exp_multi_3} of
Definition \ref{Def_appropriate_multi} hold.

First, note that each of the $H$ components of $\rho_N$--random element of $\GT_N^H$
itself satisfies a CLT in the sense of Definition \ref{Def_CLT}. Therefore, we can
invoke Theorem \ref{Theorem_main}. This implies that condition \eqref{exp_multi_1}
of Definition \ref{Def_appropriate_multi} holds, and moreover, the conditions
\eqref{exp_multi_2}, \eqref{exp_multi_3} hold for partial derivatives parameterized
by $H$--tuples of partitions $\lambda=(\lambda^{(1)},\dots,\lambda^{(H)})$ such that
all but one partition are empty.

Consider all possible $H$--tuples of partitions
$\lambda=(\lambda^{(1)},\dots,\lambda^{(H)})$ of order $R$ except for those tuples,
which have $\lambda^{(h)}=(R,0,\dots)$ for some $h=1,\dots,H$, and all other
partitions empty. Let us denote the set of all such $H$--tuples through $\Omega_R$.
Consider $|\Omega_R|$--dimensional vector $\mathbf v$, whose coordinates should be
though of as partial derivatives $\partial_\lambda \ln S_{\rho_N}$, $\lambda\in
\Omega_R$.

Further, for each $\mu\in \Omega_R$ consider an affine-linear expression in $\mathbf
v$, given by the expansion of $N^{-R}\K_{(k_1,h_1),\dots,(k_r,h_r)}$, corresponding
to $\mu$, as in \eqref{eq_multi_expansion}. By \eqref{eq_CLT_cov_multi},
\eqref{eq_CLT_cum_multi}, we know an asymptotic behavior of
$N^{-R}\K_{(k_1,h_1),\dots,(k_r,h_r)}$. Therefore, we get $|\Omega_R|$ linear
equations on the vector $\mathbf v$. According to Theorem
\ref{Theorem_multi_expansion} the equations can be compactly written as
\begin{equation}
\label{eq_multi_defining_equation} (A+N^{-1} B) \mathbf v = \mathbf z + \eps +
\mathbf o^{R-1}+ \mathbf o',
\end{equation}
where $\mathbf z$ represents the limits of renormalized right--hand sides of
 \eqref{eq_CLT_cov_multi}, \eqref{eq_CLT_cum_multi}, $\eps$ is the remainder
$o(1)$, $N\to\infty$, in these right--hand sides, $o^{R-1}$ is the contribution of
partial derivatives of $\ln(S_{\rho_N})$ of orders up to $R-1$ in the expansion
\eqref{eq_multi_expansion}, and $\mathbf o'$ is the contribution of partial
derivatives $(\partial/\partial x_1^h)^R\ln S_{\rho_N}$, $h=1,\dots,H$. $A$ is an
independent of $N$ triangular $|\Omega_R|\times |\Omega_R|$ matrix with respect to
the component-wise dominance order, as in Theorem \ref{Theorem_multi_expansion},
with non-zero diagonal elements. $B$ might depend on $N$, but its matrix elements
remain uniformly bounded as $N\to\infty$. It follows that $(A+N^{-1}B)$ is
invertible for large $N$ and its inverse is uniformly bounded.

At this point we follow the argument in Section \ref{Section_proof_of_inversion}.
The part of Theorem \ref{Theorem_main_multi} proven in Section
\ref{Section_Th_Multi_1} implies the existence of a vector $\mathbf v_0$, satisfying
the asymptotics of \eqref{exp_multi_2}, \eqref{exp_multi_3} and such that
\begin{equation}
\label{eq_multi_defining_equation_2} (A+N^{-1} B) \mathbf v_0 = \mathbf z + \eps' +
\mathbf o^{R-1}+ \mathbf o',
\end{equation}
for another $\eps'$ tending to $0$ as $N\to\infty$. Subtracting
\eqref{eq_multi_defining_equation} from \eqref{eq_multi_defining_equation_2} and
using the invertibility of $(A+N^{-1}B)$ we conclude that the difference $\mathbf
v-\mathbf v_0$ tends to $0$ as $N\to\infty$, which finishes the proof.

\subsection{Extension to multilevel setting.}

\label{Section_CLT_extension}

We give here a slight extension of the first implication of Theorem
\ref{Theorem_main_multi}.

Let $0<a_1 \le a_2 \le \dots \le a_M=1$ be fixed positive reals, and let $\rho_{N}$
be a probability measure on $\GT_{N}$. For each $m=1,\dots,M-1$ take a symmetric
function $g_m(x_1,\dots,x_{\lfloor a_m N \rfloor})$ in $\lfloor a_m N \rfloor$
variables. We will further assume that as $N\to\infty$, $g_m$ are CLT--appropriate
Schur generating functions in the sense of Definition \ref{Def_CLT_appropriate} of
some smooth probability measures.

For $\la \in \GT_{\lfloor a_m N\rfloor }, \mu \in \GT_{\lfloor a_{m-1} N\rfloor}$,
with $k_1 \ge k_2$, let us introduce the coefficients ${\mathfrak p}_{m \to m-1}
(\la \to \mu)$ via
\begin{multline}
\label{eq:Schur-branching-pr-coef} \frac{s_{\lambda} (x_1, \dots, x_{\lfloor
a_{m-1}N\rfloor}, 1^{\lfloor a_{m}N\rfloor-\lfloor a_{m-1}N\rfloor})}{s_{\la}
(1^{\lfloor a_{m}N\rfloor})} \cdot g_{m-1}(x_1,\dots,x_{\lfloor a_{m-1}N\rfloor})
\\ = \sum_{\mu \in \GT_{\lfloor a_{m-1} N\rfloor}} \mathfrak p_{m \to m-1} (\la \to \mu) \frac{s_{\mu}
(x_1, \dots, x_{\lfloor a_{m-1} N\rfloor})}{s_{\mu} (1^{\lfloor a_{m-1} N\rfloor})}.
\end{multline}
The branching rule for Schur functions and the combinatorial formulas for the
Littlewood--Richardson coefficients assert that the coefficients $\mathfrak p_{m \to
m-1} (\la \to \mu)$ are non-negative for all $\la$, $\mu$ (see \cite[Chapter
I]{Mac}). Plugging in $x_1=\dots=x_{\lfloor a_{m-1} N\rfloor}=1$, we see that
$\sum_{\mu \in \GT_{\lfloor a_{m-1}N\rfloor}} \mathfrak{p}_{m \to m-1} (\la \to
\mu)=1$.

Let us introduce the probability measure on the set $\GT_{\lfloor a_1 N\rfloor}
\times \GT_{\lfloor a_2 N\rfloor} \times \dots \times \GT_{\lfloor a_M N\rfloor}$
via
\begin{equation}
\label{eq:def-projections-general} \mathrm{Prob} ( \la{[1]}, \la{[2]}, \dots,
\la{[M]} ) := \rho ( \la{[M]} ) \prod_{m=1}^{M-1} \mathfrak p_{m+1  \to m } (
\la[m+1] \to \la[m]).
\end{equation}
(the fact that all these weights are summed up to $1$ can be straightforwardly
checked by induction over $M$.) In words, $\lambda[m]$, $m=M,M-1,\dots,1$ is a
Markov chain with initial distribution $\rho$ and transition probabilities
$\mathfrak p_{m+1  \to m }$.

More generally, take a smooth measure ${\tau}$ on $\prod_{h=1}^H \GT_{N_h}$, and for
each $h=1,\dots,H$, fix $M$ reals  $0<a_{h,1} \le a_{h,2} \le \dots \le a_{h,M}=1$
and $M-1$ symmetric functions $g_{h,m}(x_1,\dots,x_{\lfloor a_{m;h} N_h \rfloor})$,
$m=1,\dots,M-1$ in $\lfloor a_{h,m} N_h \rfloor$ variables each. We will further
assume that all these functions are CLT--appropriate. We further define the
coefficients ${\mathfrak p}_{m \to m-1;h} (\la \to \mu)$ as in
\eqref{eq:Schur-branching-pr-coef} but using $g_{h,m}$ instead of $g_m$.

Let us introduce the probability measure on the set $\prod_{h=1}^H \prod_{m=1}^M
\GT_{\lfloor a_{h,m} N_h\rfloor}$ via
\begin{multline}
\label{eq:def-projections-general_multi} \mathrm{Prob} ( \lambda^h[m],\,
h=1,\dots,H,\, m=1,\dots,M) := \tau ( \la^1{[M]},\dots,\la^H[M] ) \\
\times \prod_{h=1}^H \prod_{m=1}^{M-1} \mathfrak p_{m+1 \to m;h } ( \la^h[m+1] \to
\la^h[m]).
\end{multline}
In words, we use the procedure of \eqref{eq:def-projections-general} separately and
independently for each coordinate of $\tau$--distributed $(\lambda^1,\dots,
\lambda^H)$.

Let $L$ be a large parameter and suppose that $N_h=N_h(L)$ are such that
$$
 \lim_{L\to\infty} \frac{N_h}{L}=\mathcal N_h>0,\quad h=1,\dots,H.
$$
For each $L=1,2,\dots$, let $\rho_L$, be a smooth probability measure on
$\prod_{h=1}^H \GT_{N_h}$, and let $(\lambda^{h}[m])$, $h=1,\dots,H$, $m=1,\dots,M$
be the corresponding \eqref{eq:def-projections-general_multi}--random element of
$\prod_{h=1}^H \prod_{m=1}^M \GT_{\lfloor a_{h,m} N_h\rfloor}$.

Define $\{p_{k;h,m}^L\}_{k=1,2,\dots;\, h=1,\dots,H, m=1,\dots,M}$ to be a countable
collection of random variables via
$$
 p_{k;h,m}^L= \sum_{i=1}^{\lfloor a_{h,m} N_h\rfloor} \left(\frac{\lambda_i^{h}[m]+\lfloor a_{h,m} N_h\rfloor-i}{\lfloor a_{h,m} N_h \rfloor }\right)^k.
$$

Let us extend the definition of CLT--appropriate measures to the present setting. We
require $\rho_L$ to be CLT--appropriate in the sense of Definition
\ref{Def_appropriate_multi} as $L\to\infty$. In addition, each function $g_{h,m}$
should be CLT--appropriate in the sense of Definition \ref{Def_CLT_appropriate}.

Let us introduce an encoding for the various limits of partial derivatives of the
Schur generating functions. Recall that
$$
S_{\rho_L}(x_{1}^1,\dots, x_{N_1}^1;\, x_1^1,\dots,x_{N_2}^2;\,
x_{1}^{H},\dots,x_{N_H}^H)
$$
is the Schur--generating function of $\rho_L$. Set
$$
 c_{k;h,m}:=\lim_{L\to\infty} \frac{1}{a_{h,m} N_h} \left( \frac{\partial}{\partial
 x_{i;h}}\right)^k \left[\ln S_{\rho_L} + \sum_{q=m}^{M-1} \ln\bigl( g_{h,q}(x_1^{h},x_2^h\dots)\bigr)
 \right]_{x_{1}^1=\dots=x_{N_H}^H=1}
$$
\begin{multline*}
 d_{k,h,m; k',h',m'}:= \lim_{L\to\infty} \\ \begin{cases}   \left( \frac{\partial}{\partial
 x_{i;h}}\right)^k \left( \frac{\partial}{\partial
 x_{i';h'}}\right)^{k'} \left[\ln S_{\rho_L}
 \right]_{x_{1}^1=\dots=x_{N_H}^H=1},& h\ne h',\\
  \left( \frac{\partial}{\partial
 x_{i;h}}\right)^k \left( \frac{\partial}{\partial
 x_{i';h'}}\right)^{k'} \left[\ln S_{\rho_L} + \sum_{q=\max(m,m')}^{M-1} \ln\bigl( g_{h,q}(x_1^{h},x_2^h\dots)\bigr)
 \right]_{x_{1}^1=\dots=x_{N_H}^H=1}, & h=h'.
 \end{cases}
\end{multline*}

\begin{theorem}
\label{Theorem_CLT_multi_ext} Suppose that the measures $\rho_L$ are
CLT--appropriate and the functions $g_{h,m}$ are Schur generating functions of
smooth CLT--appropriate measures. Then the following CLT for $(p_{k;h,m}^L)$ holds:
\begin{enumerate}[label=(\Alph*)]
\item \label{ass_multi_ext_1} For each $k=1,2,\dots$, $h=1,\dots,H$, $m=1,\dots,M$
$$
 \lim\limits_{L\to\infty} \frac{1}{a_{h,m} N_h} \E [p_{k;h,m}^L] = {\mathfrak p}(k;h,m),
$$
\item \label{ass_multi_ext_2} For each $(k,h,m)$, $(k',h',m')$
$$
 \lim\limits_{L\to\infty}\Bigl( \E [p_{k;h,m}^L\, p_{k';h',m'}^L] -\E[p_{k;h,m}^L]\E [p_{k';h',m'}^L] \Bigr)= \mathfrak {cov}(k,h,m;k',h',m'),
$$
\item \label{ass_multi_ext_3}For each $r>2$ and any collection of pairs
$(k_1,h_1,m_1)$, \dots, $(k_r,h_r,m_r)$,
$$\lim_{L\to\infty} \kappa_{ (k_1,h_1,m_1),\dots,(k_r,h_r,m_r)}\bigl( p^L_{k;h,m},\,  k=1,2,\dots,\, h=1,\dots,H,\, m=1,\dots,M\bigr)=0.
$$
\end{enumerate}
Where
\begin{equation}
\label{eq_LLN_multi_formula_ext} \mathfrak p(k;h,m) = [z^{-1}] \frac{1}{(k+1)(z+1)}
\left( (1+z) \left(\frac1z  +  \sum_{a=1}^\infty \frac{\c_{a;h,m} z^{a-1}
}{(a-1)!}\right) \right)^{k+1},
\end{equation}
\begin{multline}
\label{eq_CLT_multi_formula_ext} \mathfrak {cov}(k_1,h_1,m_1;k_2,h_2,m_2)\\ =
  [z^{-1} w^{-1}]
 \left( (1+z)
\left(\frac1z  +  \sum_{a=1}^\infty \frac{\c_{a;h_1,m_1} z^{a-1}
}{(a-1)!}\right)\right)^{k_1} \, \left( (1+w)
\left(\frac1w  +  \sum_{a=1}^\infty \frac{\c_{a;h_2,m_2} w^{a-1} }{(a-1)!}\right) \right)^{k_2} \\
\times \left(  \delta_{h_1=h_2}\left(\sum_{a=0}^{\infty}
\frac{z^a}{w^{1+a}}\right)^2 +\sum_{a,b=1}^{\infty}
\frac{\d_{a,h_1,m_1;b,h_2,m_2}}{(a-1)! (b-1)!} z^{a-1} w^{b-1} \right),
\end{multline}
where $[z^{-1}]$ and $[z^{-1} w^{-1}]$ stay for the coefficients of $z^{-1}$ and
$z^{-1}w^{-1}$, respectively, in the Laurent power series given afterwards.
\end{theorem}

The proof of Theorem \ref{Theorem_CLT_multi_ext} is the same, as the one presented
in Section \ref{Section_Th_Multi_1}, and therefore we omit it. See also
\cite[Theorems 2.9, 2.10, 2.11]{BG_CLT} for similar extensions.

\section{Application to random tilings}

\label{Section_tilings}

In this section we apply Theorems \ref{Theorem_main}, \ref{Theorem_main_multi},
\ref{Theorem_CLT_multi_ext} to the study of global macroscopic fluctuations of
uniformly random lozenge and domino tilings of polygonal domains. A general
conjecture of Kenyon and Okounkov \cite{KO_Burgers} predicts that fluctuations
should be universally governed by the Gaussian Free Field. We prove this conjecture
for several classes of domains, leading to GFF in \emph{multiply--connected}
regions. We will not aim at the most general accessible situations, instead
concentrating on two basic cases, illustrating the basic principles: these are
lozenge tilings of hexagons with a single hole and domino tilings of Aztec
rectangles with several holes along a
single axis. 

Let us outline a general strategy. Each domain that we study will have a specific
singled out vertical section, along which the distribution of the particles can be
computed and identified with a discrete log--gas. The results of \cite{BGG} then
imply the Central Limit Theorem along the section. After we cut the domain along
this vertical line, it splits into several trapezoids. Using the only if part of
Theorems \ref{Theorem_main}, \ref{Theorem_main_multi} we then obtain the asymptotic
for the partial derivatives of the Schur generating functions, and Theorem
\ref{Theorem_CLT_multi_ext} leads to the extension of CLT to the entire domain.
Theorem \ref{Theorem_CLT_multi_ext} outputs the covariance in a contour integral
form, and efforts are then required to identify it with the Gaussian Free Field.

\subsection{Lozenge tilings of holey hexagons: formulation}
\label{Section_GFF_hex_formulation}

Consider an $A\times B\times C$ hexagon with $D\times D$ rhombic hole drawn on a triangular grid,
as shown in Figure \ref{Figure_hex_hole}. The bottom point of the hole is distance $t$ from the
left border of the hexagon, and is at distance $E$ in vertical direction from the bottom border of
the hexagon.

\begin{figure}[h]
\begin{center}
 {\scalebox{0.65}{\includegraphics{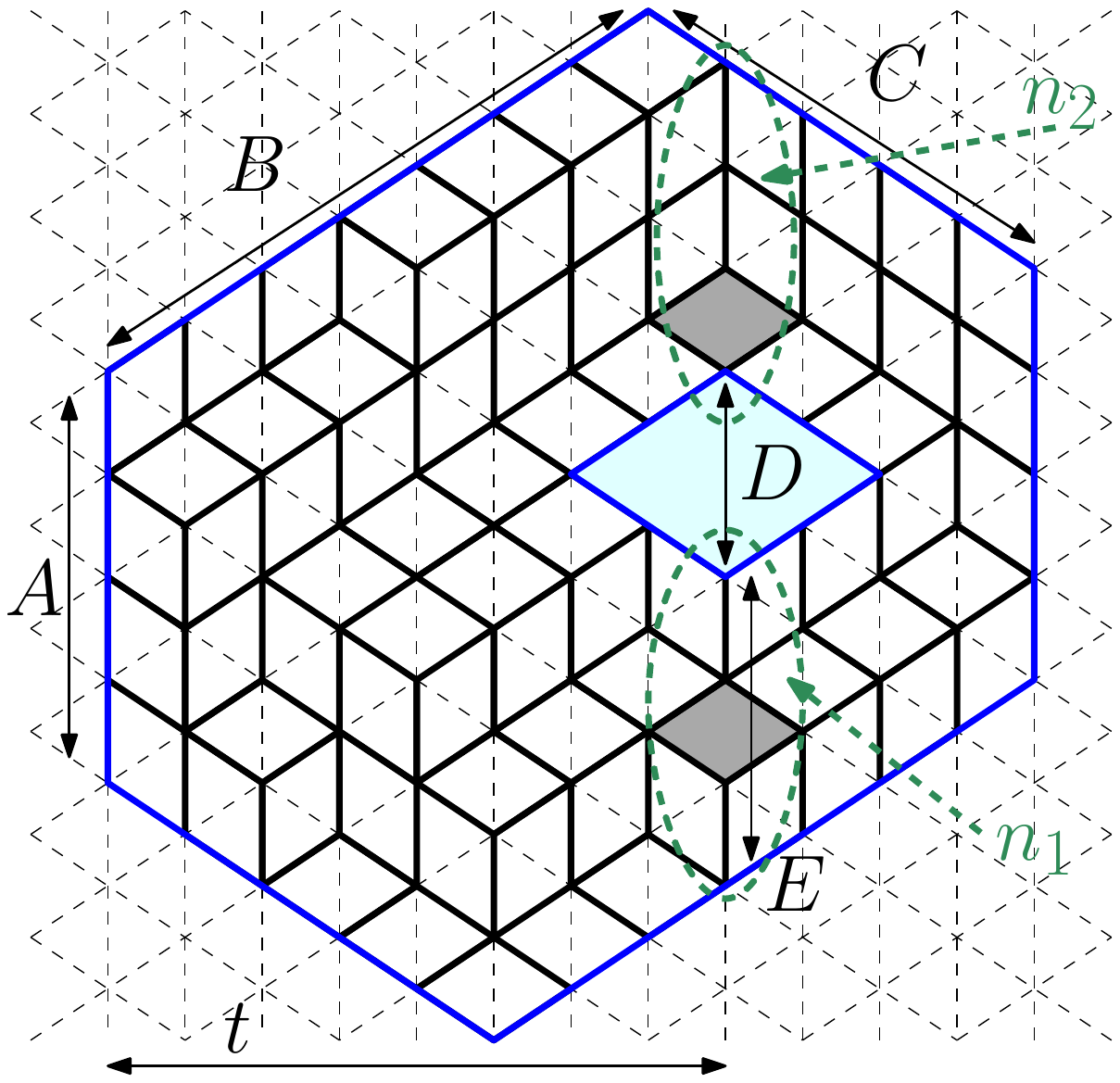}}} \quad
  {\scalebox{0.34}{\includegraphics{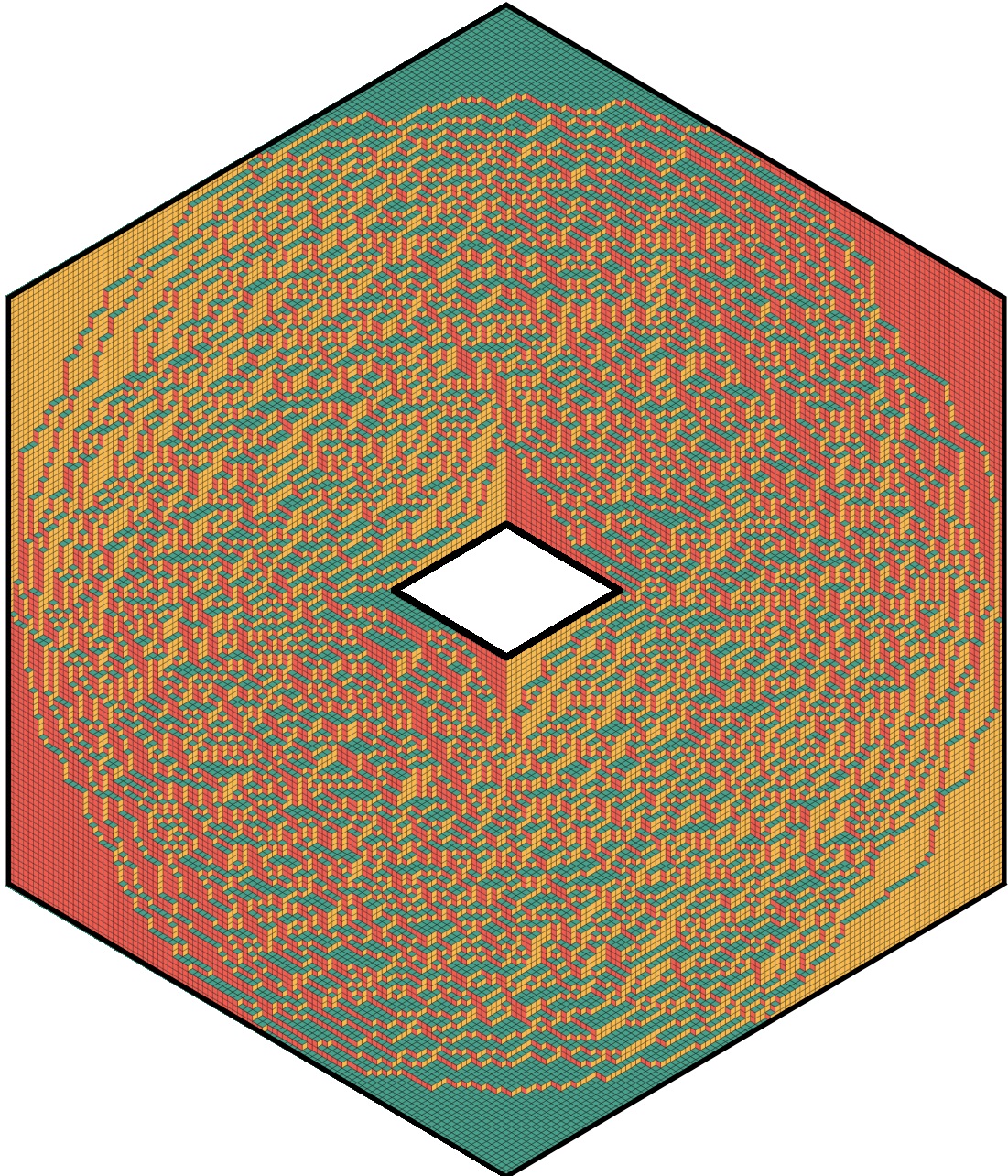}}}
 \caption{Uniformly random lozenge tilings of hexagons with holes. We are grateful
 to Leonid Petrov for the simulations}
 \label{Figure_hex_hole}
\end{center}
\end{figure}

We study \emph{lozenge tilings} of a holey hexagon. Such tilings can be
alternatively viewed as projections in $(1,1,1)$ direction of stepped surfaces, or
piles of cubes.

For a given tiling, consider its section by the vertical line $x=t$ passing through
the center of the hole. The number of \emph{horizontal lozenges}
{\scalebox{0.16}{\includegraphics{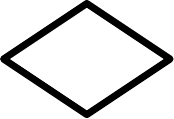}}} on this section is fixed, i.e.\
does not depend on the choice of tilings. For instance, when $C<B<t$, as in the left
panel of Figure \ref{Figure_hex_hole}, then it is $B+C-t-D$. However, the number of
horizontal lozenges {\scalebox{0.16}{\includegraphics{lozenge_hor.pdf}}} below and
above the hole might vary from tiling to tiling. These numbers are called
\emph{filling fractions}. Let us denote them $n_1$ and $n_2$, respectively (so that
$n_1+n_2=B+C-t-D$ in our example).

Fix $A$, $B$, $C$, $D$, $t$, $E$, and $n_1$, $n_2$, and let $\Omega$ be
\emph{uniformly random} lozenge tiling with such parameters. We further identify
$\Omega$ with its \emph{height function}, $H(x,y;\Omega)$.

\begin{definition}
Fix a \emph{scale parameter} $L$. For a given $(x,y)$, the height function
$H_L(x,y)$ counts the number of horizontal lozenges
{\scalebox{0.16}{\includegraphics{lozenge_hor.pdf}}} directly \emph{above}
$(L^{-1}\lfloor L x\rfloor ,L^{-1}\lfloor L y\rfloor )$, cf.\ Figures
\ref{Figure_triangle}, \ref{Figure_hex_hole_coord} for our choice of the $(x,y)$
coordinate system.\footnote{Many articles use another definition of the height
function, counting the number of lozenges of types
{\scalebox{0.16}{\includegraphics{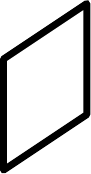}}},
{\scalebox{0.16}{\includegraphics{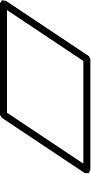}}} below the point $(Nx,Ny)$.
The third definition treats a tiling a projection of a stepped surface in direction
onto the plane $x+y+z=0$ along the direction $(1,1,1)$ --- the stepped surface
itself is then the height function. All definitions  are related to each other by
affine transformations of coordinate systems, and so the asymptotic theorems for
them are equivalent.}
\end{definition}

Note that in terms of the height function, fixing $n_1$, $n_2$ is equivalent to
fixing the heights inside the removed size $D$ rhombus, so this is a natural
conditioning on the deterministic boundary values, from the point of view of stepped
surfaces.

We are interested in the following limit regime:
\begin{align}
\label{eq_hex_limit_regime}
A=L \cdot \hat A, && B= L\cdot \hat B, && C= L\cdot \hat C, && D= L\cdot \hat D,\\
\notag t=L \cdot \tau, && E= L\cdot \hat E, && n_1= L\cdot \hat n_1, && n_2= L\cdot
\hat n_2,&& L\to\infty.
\end{align}
Our aim is to study the asymptotic of the random height function $H_L(x,y)$ in the
regime \eqref{eq_hex_limit_regime}.

Note that all the proportions of the hexagon need to be integers. If all the numbers
in \eqref{eq_hex_limit_regime} are integers themselves, then this is automatic. A
more general parameters (even irrational) are also possible --- then we need to add
integer parts to \eqref{eq_hex_limit_regime}, e.g.\ $A=\lfloor L \cdot \hat
A\rfloor$. For the notational simplicity, we will silently assume that all the
numbers are integers, but the proofs go through for general values of the parameters
without any changes.

\begin{proposition}
\label{Proposition_lozenge_LLN}
 In the limit regime \eqref{eq_hex_limit_regime}, as $L\to\infty$ the random normalized height
 functions $\frac{1}{L} H_L(x,y)$ converge to a non-random limit shape $\mathfrak
 h(x,y)$.
\end{proposition}
The article \cite{CKP} contains a very general variational principle, which
guarantees the existence of the limit shapes for lozenge tilings of almost arbitrary
domains. Note that \cite{CKP} assumes the domains to be \emph{simply connected}, and
therefore, the setting of Proposition \ref{Proposition_lozenge_LLN} is formally out
of the scope of that article; however, the proofs probably go through with very
minor modifications. Nevertheless, we give our own proof of Proposition
\ref{Proposition_lozenge_LLN} in Corollary \ref{Corollary_hex_LLN_CLT} below.

The limit shape $\mathfrak h(x,y)$ has two types of local behavior, as can be seen
in the right panel of Figure \ref{Figure_hex_hole}. Note that the definitions imply
$-1 \le \frac{\partial}{\partial x} \mathfrak h(x,y) \le 0$; the points where the
extreme values are achieved, i.e.\ $\frac{\partial}{\partial x} \mathfrak h(x,y)\in
\{0,-1\}$ form a \emph{frozen region}, other points where $-1
<\frac{\partial}{\partial x} \mathfrak h(x,y)  < 0$ give rise to the \emph{liquid
region} $\mathcal L$. In frozen regions as $L\to\infty$ we see only one type of
lozenges, and therefore, there are no fluctuations. On the other hand, we will show
that in the liquid region $\mathcal L$ the recentered height function $H_L(x,y)-\E
H_L(x,y)$ converges to the \emph{Gaussian Free Field}. This is a particular case of
a general conjecture of Kenyon and Okounkov \cite{KO_Burgers} which we now state.

Near any point $(x,y)$ the partial derivatives of the limit shape $\mathfrak h(x,y)$
uniquely define local proportions of three types of lozenges
$p^{{\scalebox{0.16}{\includegraphics{lozenge_hor.pdf}}}}(x,y)$,
$p^{{\scalebox{0.16}{\includegraphics{lozenge_v_up.pdf}}}}(x,y)$,
$p^{{\scalebox{0.16}{\includegraphics{lozenge_v_down.pdf}}}}(x,y)$. These are
non-negative numbers, which sum up to $1$. In frozen regions one of them equals $1$
and others vanish, while in the liquid region all three are between $0$ and $1$. We
remark that the probabilistic interpretation of the local derivatives as
probabilities of seeing lozenges was rigorously proven for our class of domains only
recently in \cite{G_Bulk}. For the most general domains it remains conjectural.

Consider a triangle on the complex plane with angles proportions $\pi \cdot
p^{{\scalebox{0.16}{\includegraphics{lozenge_hor.pdf}}}}(x,y)$, $\pi \cdot
p^{{\scalebox{0.16}{\includegraphics{lozenge_v_up.pdf}}}}(x,y)$, $\pi \cdot
p^{{\scalebox{0.16}{\includegraphics{lozenge_v_down.pdf}}}}(x,y)$, and vertices $1$, $\xi$, $0$,
respectively, such that $\xi$ lies in the upper halfplane, cf.\ Figure \ref{Figure_triangle}. $\xi$
is called the \emph{complex slope} of the limit shape $\mathfrak h(x,y)$.

\begin{figure}[h]
\begin{center}
 {\scalebox{1.0}{\includegraphics{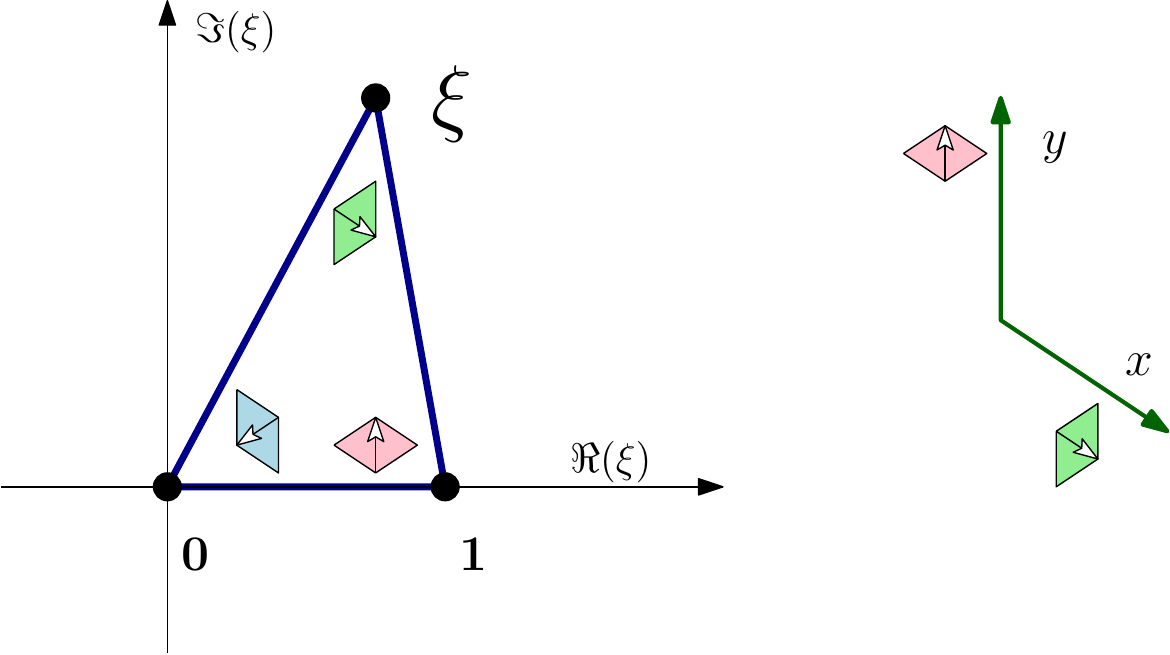}}}
 \caption{Complex slope $\xi$ corresponds to three local proportions of lozenges through a geometric construction. The coordinate
 directions for the complex Burgers equation can be also encoded by lozenges.}
 \label{Figure_triangle}
\end{center}
\end{figure}

The three angles of the triangle are $\arg(\xi)$, $\arg\left(\frac{\xi-1}{\xi}\right)$, and
$\arg\left(\frac{1}{1-\xi}\right)$, corresponding to three types of lozenges.
 Note that there is no canonical order on types of lozenges, and therefore, the angles
of the triangle of Figure \ref{Figure_triangle} can be reshuffled to get five alternative
definitions of the complex slope:
\begin{equation}
\label{eq_alternative_form}
 \xi'=\frac{\xi-1}{\xi},\quad \xi'=\frac{1}{1-\xi}, \quad \xi'=\frac{1}{\bar \xi}, \quad \xi'=1-\bar \xi,\quad \xi'=\frac{\bar \xi}{\bar \xi-1}.
\end{equation}
Kenyon and Okounkov \cite{KO_Burgers} showed that the complex slope satisfies the complex Burgers
equation everywhere inside the liquid region, which in our notations reads:
\begin{equation}
\label{eq_Burgers}
 \xi_y= \xi \cdot \xi_x, \quad (x,y)\in\mathcal L.
\end{equation}
We again note that there is no canonical way to specify the coordinate system on the
triangular lattice: we need to somehow choose two coordinate directions out of the
six lattice directions, and different choices lead to slightly different equations
\eqref{eq_Burgers}. Only if we make this choice in a way agreeing with our choice of
the complex slope, we get precisely \eqref{eq_Burgers}.\footnote{For instance, in
the articles \cite{Petrov-curves, Petrov_GFF}, the
choice of the complex slope is different, which leads to the equation
$\Omega_\chi=\frac{\Omega-1}{\Omega} \Omega_\eta$ on the slope.}

We use $\xi$ to define a \emph{complex structure} on the liquid region $\mathcal L$ as follows:
\begin{definition}
\label{Definition_complex_structure_general}
 A function $f:\mathcal L\mapsto \mathbb C$ is \emph{complex--analytic} (holomorphic), if
 \begin{equation}
 \label{eq_analytic_function}
  f_y = \xi \cdot f_x,\quad (x,y)\in\mathcal L.
 \end{equation}
\end{definition}
If $\xi$ is non-degenerate near a given point $(x,y)$, i.e.\ its Jacobian does not vanish (in our
examples $\xi$ is usually non-degenerate), then we can replace \eqref{eq_analytic_function} by the
condition that locally $f$ is a (conventional) holomorphic function of $\xi$.

The complex structure turns the liquid region into a compact ($1d$) complex manifold
with a boundary. Let $G(x,y)$ denote the Green function of the Laplace operator in
the liquid region (with respect to this complex structure) with Dirichlet boundary
conditions.\footnote{In other words, $G(x,y)$ is the integral kernel for the inverse
operator with Dirichlet boundary conditions to the operator
$\Delta=\frac{\partial}{\partial z} \frac{\partial}{\partial \bar z}$ acting on real
functions, where the complex derivatives (as well as the area form for the
integration) are computed in accord with the complex structure.}

Note that the five other choices of the complex slope in \eqref{eq_alternative_form}
lead to the same complex structure up to complex conjugation, and therefore, to the
same Green function.

\begin{definition}
 The Gaussian Free Field (GFF) in $\mathcal L$ (with respect to complex slope $\xi$ and with Dirichlet boundary conditions)
 is a generalized centered Gaussian field,
 whose covariance is $G(x,y)$.
\end{definition}
Since the Green function has a logarithmic singularity on the diagonal, the values
of GFF are not defined, but its integrals along the curves as well as pairings with
(smooth enough) test functions are bona-fide Gaussian random variables.

Another way to think about the GFF on $\mathcal L$ is through the uniformization. The Koebe
general uniformization theorem (see e.g.\ \cite[Chapter III, Section 4]{AS}) implies that there exists a bijective
conformal (in the sense of Definition \ref{Definition_complex_structure_general})  map $\Theta$
from the liquid region $\mathcal L$ to a domain $\mathbb D\subset \mathbb C$ of the same topology.
$\mathbb D$ is equipped with the conventional complex structure, and we can consider the GFF on
$\mathbb D$ --- generalized centered Gaussian field whose covariance is the Green function for the
Laplace operator with Dirichlet boundary conditions. Then the GFF on $\mathcal L$ is simply the
$\Theta$--pullback of the GFF on $\mathbb D$.

\begin{conjecture}[{\cite[Section 2.3]{KO_Burgers}}]
\label{Conjecture_GFF}
 Fix an arbitrary (tilable) polygonal domain $\Upsilon$ on the triangular grid and let
 $H_L(x,y)$ be the random height function of a uniformly random lozenge tiling\footnote{If $\Upsilon$ is not simply connected, then we additionally deterministically fix heights of each hole, as we did for the holey hexagon}
 of its dilated version $L\cdot \Upsilon$. Then $\sqrt{\pi}\left[H_L(x,y)-\E H_L(x,y)\right]$ converges as $L\to\infty$ in
 the liquid region $\mathcal L$ to the Gaussian Free Field with respect to the complex slope
 $\xi$ and with Dirichlet boundary
 conditions.
\end{conjecture}

There is a class of domains with non-trivial limit shapes for which the validity of
Conjecture \ref{Conjecture_GFF} have been established, see \cite{Kenyon_height,
BorFer, Petrov_GFF, Duits,BG_CLT,Ahn}; however all these domains are
simply--connected. In particular, in these articles the liquid regions are
conformally equivalent to the upper half--plane $\mathbb U$ with the standard
complex structure. Therefore, the limiting GFF is identified with a pullback of the
standard GFF in $\mathbb U$ (whose covariance is explicit: $\mathrm
{Cov}(z,w)=-\frac{1}{2\pi}\ln|\frac{z-w}{z-\bar w}|$) with respect to the
uniformization map identifying the liquid region with $\mathbb U$. This map turned
out to be explicit in \cite{BorFer, Petrov_GFF, Duits,BG_CLT,Ahn}.

The main result of this section leads to the first rigorous appearance of the Gaussian Free Field
in tilings of multiply-connected planar domains.

\begin{theorem} \label{Theorem_holey_hex} Conjecture \ref{Conjecture_GFF} is true for the lozenge tilings $\Omega$
of the holey hexagons in the limit regime \eqref{eq_hex_limit_regime}.
\end{theorem}
We prove the convergence to the Gaussian Free Field for finite-dimensional
distributions of pairings with a specific class of test--functions, which are
$\delta$--functions in $x$--direction and polynomials in $y$--directions. A more
detailed version of Theorem \ref{Theorem_holey_hex} is given below in Proposition
\ref{Proposition_holey_hex_detailed}.

\subsection{Domino tilings of holey Aztec rectangles: formulation}

The story of uniformly random domino tilings is to a large extent parallel to the lozenge one.

In this section we will first state a generalization to dominos of Kenyon--Okounkov's conjecture,
and then we show how to prove this conjecture for a class of Aztec rectangles with holes.

In the most general framework, we would like to consider uniformly random domino tilings of domains
on square grid with $2\times 1$ dominos. Let us describe one interesting class of such domains.

\begin{figure}[t]
\begin{center}
 {\scalebox{1.7}{\includegraphics{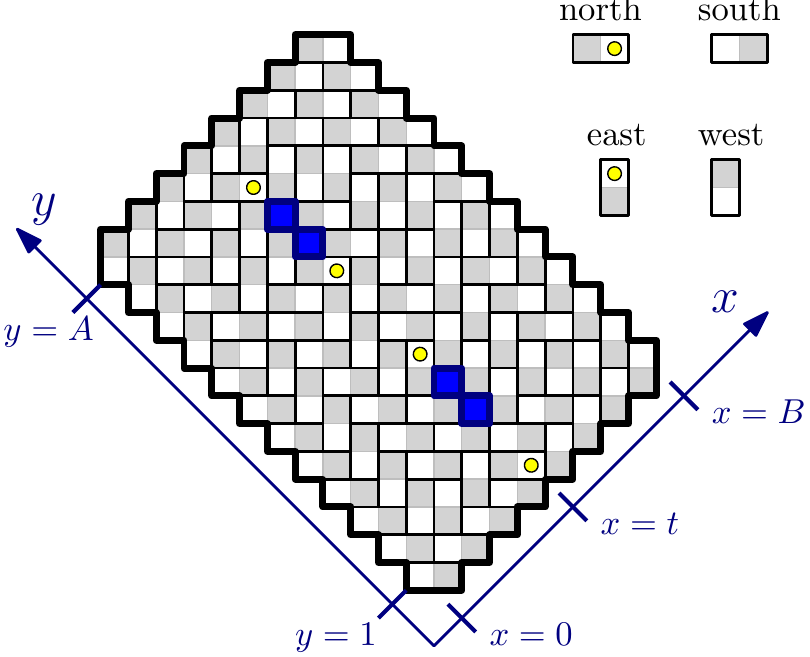}}}
 \caption{Rectangular Aztec diamond with 2 holes. Here $A=12$ and $B=8$. Also four types of dominos are shown:
 north and south in the first row, east and west in the second row. Particles are put on north and east dominos, and we show them
 along the $x=t$ section of the tiling.}
 \label{Figure_domino_domain}
\end{center}
\end{figure}

Consider $A\times B$ rectangle drawn on the square grid in the diagonal direction, i.e.\ such that
its sides are not parallel to the base grid directions, but are rather rotated by $45$ degrees, see
Figure \ref{Figure_domino_domain}. We color the grid in the checkerboard fashion, and the white
squares in the rectangle span coordinates from $y=1$ to $y=A$ and from $x=0$ to $x=B$, so that
there are $A\times (B+1)$ of them. We use the rotated $(x,y)$--coordinate system, as in Figure
\ref{Figure_domino_domain}. We assume $A\ge B$ and along the line $x=t$ we remove $(A-B)$ white
squares from the domain. We call the result \emph{a holey Aztec rectangle}. The total number of
white squares in the domain is $A \times (B+1)- (A-B)= A B +B$. The total number of black squares
is the same, and therefore, the domain is tilable with dominos, Figure \ref{Figure_domino_domain}
shows one possible tiling. Our object of interest is \emph{uniformly random domino tiling} of a
large holey Aztec rectangle.

There are several ways to define a height function of the tiling. All of them are based on
distinguishing not two, but four types of dominos, according to their checkerboard colorings:
horizontal dominos can be \emph{north} and \emph{south} --- the former can fit into the upmost
corner of the recatngle, while latter fits into the bottommost part; similarly, vertical dominos
can be \emph{east} and \emph{west}. We put particles onto white squares of north and east dominos.

\begin{definition}
\label{Def_height_domino}
 Fix a \emph{scale parameter} $L$. For a given $(x,y)$, the height function
$H_L(x,y)$ counts the number of particles (in north and east dominos) directly
\emph{above}, i.e.\ with the same $x$ coordinate and greater $y$ coordinate than
$(L^{-1}\lfloor L x\rfloor ,L^{-1}\lfloor L y\rfloor )$ in the coordinate system of
Figure \ref{Figure_domino_domain}.
\end{definition}

A more traditional definition of the height function for domino tilings (cf.\ \cite{Kenyon_notes})
proceeds as follows. We define the height function $h$ at all vertices of the square lattice, by
imposing the following two local rules:
\begin{itemize}
 \item If a lattice edge $(u,v)$ belong to a domino (i.e.\ cuts it into two squares) of the
     tiling, then $h(v)=h(u)\pm 3$, where the sign is $+$ if $(u,v)$ has a dark square on the
 left and $-$ otherwise.
 \item If a lattice edge $(u,v)$ does not belong to a domino (i.e.\ it borders one of the
     dominos), then $h(v)=h(u)\pm 1$, where the sign is $+$ is $(u,v)$ has a dark square on the
 left and $-$ otherwise.
\end{itemize}
 \cite[Lemma 3.11]{BK} explains how for
\emph{simply connected domains} Definition \ref{Def_height_domino} is matched to the above local
definition of the height function by an affine change of coordinates. Note, however, that for
multiple connected domains the local definition has an inconvenient feature: when we loop around a
hole, the height function will pick up a non-zero increment. Therefore, the height becomes a
multivalued function. This can be fixed\footnote{We would like to thank Rick Kenyon for explaining
this to us.} by noting that the increment depends on the path of integration inside the domain, but
not on the (random) domino tiling and therefore, the fluctuations of the height function (which we
are mostly interested in) are still single--valued.

Definition \ref{Def_height_domino} does not have this problem. But the tradeoff is discontinuity of
$H_L(x,y)$ on the line $x=t$: the height function makes a deterministic jump because of the white
squares cut out of the domain.

\medskip

As the size of the rectangle grows, we would like to vary the positions of the holes
in a regular way. For that we introduce an integer $K>0$, and split  the positions
$\{1,2,\dots,A\}$ into $K$ disjoint groups. The first one is at
$a_1,a_1+1,\dots,b_1$, the second one is at $a_2,a_2+1,\dots,b_2$, etc, until the
last one is at $a_K,a_K+1,\dots,b_K$. We require $a_1=1$ and $b_K=A$. We put holes
in every position \emph{outside} these groups, so that the total number of holes is:
$$
 A-B = A -\sum_{k=1}^K (b_k-a_k+1).
$$

Another parameter is the number of particles in each group. Let us call these numbers
$n_1,n_2,\dots,n_K$. They are \emph{filling fractions}. In principle, they do depend on the
particular choice of the domino tiling, however, we \emph{restrict} the class of considered tilings
by requiring that the filling fractions are deterministically fixed. (Otherwise their fluctuations
would spoil the Gaussian central limit theorem which we aim at). Combinatorics of the model implies
that
$$
 \sum_{k=1}^K n_k= t.
$$

Let $L\to\infty$ be a large parameter. We study unifomly random domino tilings in the asymptotic
regime
\begin{equation}
\label{eq_Aztec_limit_regime_1}
 \frac{A}{L}\to \hat A,\quad \frac{B}{L}\to \hat B,\quad \frac{t}{L}\to \tau,\quad \frac{a_k}{L}\to \hat a_k,\quad
 \frac{b_k}{L}\to \hat b_k,\quad \frac{n_k}{L}\to \hat n_k, \quad 1\le k\le K,
\end{equation}
\begin{equation}
\label{eq_Aztec_limit_regime_2}
 0=\hat a_1 < \hat b_1<\hat a_2<\hat b_2<\dots<\hat a_{K}<\hat b_K=\hat A,\quad \hat n_k>0, \quad 1\le k\le K,
\end{equation}
where $K$ is an arbitrary non-negative integer.

\begin{proposition}
\label{Proposition_LLN_Aztec} In the limit regime \eqref{eq_Aztec_limit_regime_1},
\eqref{eq_Aztec_limit_regime_2}, the rescaled height function $\frac{1}{L} H_L(x,y)$
of a uniformly random domino tiling converges as $L\to\infty$ to a non-random limit
shape $\mathfrak h(x,y)$.
\end{proposition}

As with Proposition \ref{Proposition_lozenge_LLN}, a version of Proposition
\ref{Proposition_LLN_Aztec} for simply connected domains would follow from the variational
principle of \cite{CKP}; for our domain we provide a proof in Section \ref{Section_domino_proofs}.

The limit shape can be alternatively encoded by four local proportions of the four
types of dominos: $\pn$, $\ps$, $\pe$, $\pw$. These four quantities can be uniquely
reconstructed from the derivatives of the limit shape. In our notations they are
found by solving the equations :
\begin{align}
\label{eq_A_height_1} \pn,\ps,\pe,\pw\ge 0,\\
\label{eq_A_height_2} \pn+\ps+\pe+\pw=1,\\
 \label{eq_A_height_3}\frac{\partial h}{\partial y}=-\pn-\pe,\\
 \label{eq_A_height_4}\frac{\partial h}{\partial x}=\pe+\ps,\\
 \label{eq_A_height_5} \sin(\pi \pn) \sin(\pi \ps)=\sin(\pi \pe)\sin(\pi \pw).
\end{align}
 The conditions
\eqref{eq_A_height_1}--\eqref{eq_A_height_4} can be seen from the geometry and our
definition of the height function. The last condition \eqref{eq_A_height_5} follows
from the uniformity of the measure, as explained in \cite{CKP}.

We remark, that, in principle, \eqref{eq_A_height_1}--\eqref{eq_A_height_5} can be
taken as the definition of $\pn,\ps,\pe,\pw$ --- for the purpose of this article
such definition would be enough. In the most general setup, the interpretation of
$\pn,\ps,\pe,\pw$  as the local appearance probabilities for the dominos is
conjectural, see \cite[Section 13]{CKP} for the conjectures. For our particular
class of domains, the interpretation of $\pn+\pe$ as the local probability of the
appearance of particles was recently proven in \cite{G_Bulk}, \cite[Appendix B]{BK}.

We remark that similarly to the height function itself, the limit shape $\mathfrak
h(x,y)$ is discontinuous on the line $x=\tau$.  However, this discontinuity
disappears when we pass to from $\mathfrak h$ to its partial derivatives or,
equivalently, to proportions $\pn,\ps,\pe,\pw$.

The region where all four proportions, $\pn,\ps,\pe,\pw$, are positive is called the
\emph{liquid region}, while degenerate areas are \emph{frozen}. In the liquid region
we encode the proportions by the complex slope $\xi$ --- a complex number in the
upper half--plane. For that we consider a quadrilateral $PQRS$ on the complex plane
with vertices $0,1, 1+\xi, \frac{1+\xi}{1-\xi}$, respectively, which all lie on a
circle. We require $\measuredangle RSQ=\measuredangle RPQ=\pi \pe$, $\measuredangle
SQR=\measuredangle SPR=\pi \ps$, $\measuredangle SRP=\measuredangle SQP=\pi \pw$,
$\measuredangle PRQ=\measuredangle PSQ=\pi \pn$, see Figure \ref{Figure_quad}. The
condition \eqref{eq_A_height_5} is used to guarantee the existence of the
quadrilateral --- it says that $|PS|\cdot |QR|=|PQ|\cdot|SR|$, which must hold, as
in terms of the complex numbers encoding the sides of $PQRS$,
$\bigl|\frac{1+\xi}{1-\xi}\bigr|\cdot \bigl|\xi\bigr|=\bigl|1\bigr| \cdot \bigl|\xi
\frac{1+\xi}{1-\xi}\bigr|$.

\begin{figure}[t]
\begin{center}
 {\scalebox{1.6}{\includegraphics{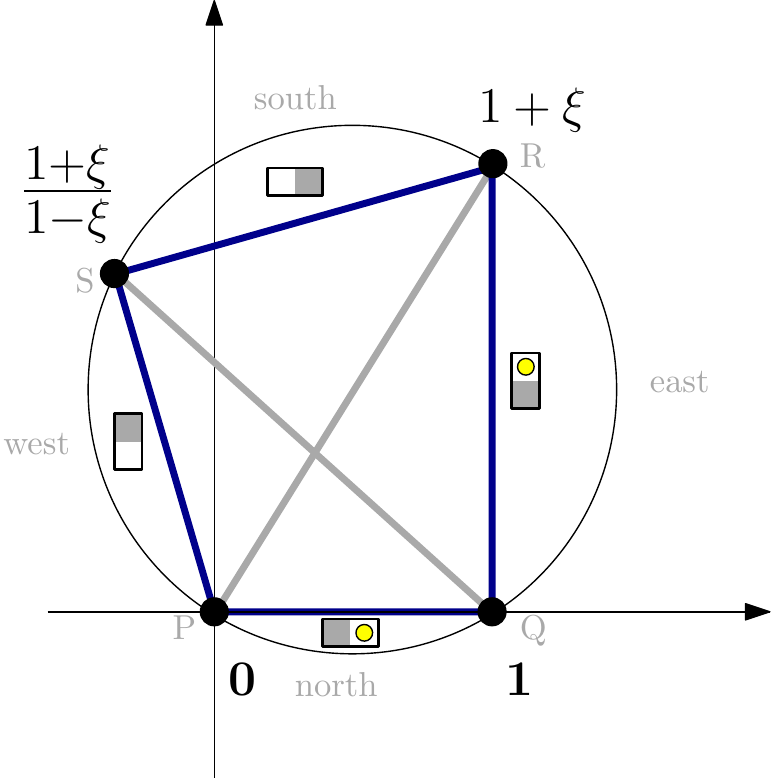}}}
 \caption{A quadrilateral encoding four proportions of dominos and the complex slope $\xi$.
 }
 \label{Figure_quad}
\end{center}
\end{figure}

Similarly to the formulas \eqref{eq_alternative_form} for the lozenge tilings, we have some freedom
in choosing $\xi$, since there is no canonical ordering on the types of dominos. Again, this will
not be important for the final covariance of the Gaussian Free Field, and therefore, we do not
detail this here.

If we denote $z=\xi$, $w=\frac{\xi+1}{\xi-1}$ (these are oriented sides $QR$ and $SP$ of $PQRS$),
then
\begin{equation}
\label{eq_Spectal_domino}
 P(z,w):=1+ z + w -zw=0. 
\end{equation}
The polynomial \eqref{eq_Spectal_domino} is the spectral curve $P(z,w)$ of
\cite{KO_Burgers} for the uniformly random domino tilings. The complex Burgers
equation of \cite{KO_Burgers} reads in this case for $(x,y)$ in the liquid region:
\begin{equation}
\label{eq_Burgers_domino_1}
 \frac{z_x}{z}+\frac{w_y}{w}=0,\quad P(z,w)=0.
\end{equation}
In terms of solely $\xi$ this can be rewritten as
\begin{equation}
\label{eq_Burgers_domino_2}
 \xi_x=\frac{2}{\xi-\xi^{-1}}\cdot \xi_y.
\end{equation}
As with lozenge tilings, $\xi$ gives rise to a complex structure on the liquid
region $\mathcal L$.

\begin{definition}
 A function $f:\mathcal L \to \mathbb C$ is complex--analytic (holomorphic) if
\begin{equation}
\label{eq_complex_structure_domino}
 f_x=\frac{2}{\xi-\xi^{-1}}\cdot f_y.
\end{equation}
\end{definition}
As before, locally near any point where $\xi$ is non-degenerate (i.e.\ its Jacobian
is non-vanishing), \eqref{eq_complex_structure_domino} is equivalent to saying that
$f$ is  a (conventional) holomorphic function of $\xi$.

\begin{conjecture}[{\cite[Section 2.3]{KO_Burgers}}]
\label{Conjecture_GFF_domino}
 Fix an arbitrary (tilable) polygonal (with sides inclined by 45 degrees to the grid directions)
 domain $\Upsilon$ on the square lattice and let
 $H_L(x,y)$ be the random height function of a uniformly random domino tiling\footnote{If $\Upsilon$ is not simply connected, then we additionally
 deterministically fix heights of each hole, as we did for the holey Aztec rectangle.}
 of its dilated version $L\cdot \Upsilon$. Then $\sqrt{\pi}\left[H_L(x,y)-\E H_L(x,y)\right]$ converges as $L\to\infty$ in
 the liquid region $\mathcal L$ to the Gaussian Free Field with respect to the complex slope
 $\xi$ and with Dirichlet boundary
 conditions.
\end{conjecture}

Before this article, the validity of Conjecture \ref{Conjecture_GFF_domino} was established for the
tilings of the Aztec diamond in \cite{CJY}, \cite{BG_CLT} and for a class of rectangles in
\cite{BK}. Here is the first non-simply connected result.

\begin{theorem}
\label{Theorem_holey_Aztec} Conjecture \ref{Conjecture_GFF_domino} is true for domino tilings of
holey Aztec rectangles in the limit regime \eqref{eq_Aztec_limit_regime_1},
\eqref{eq_Aztec_limit_regime_2}.
\end{theorem}

We prove the convergence to the Gaussian Free Field for finite-dimensional
distributions of pairings with a specific class of test--functions, which are
$\delta$--functions in $x$--direction and polynomials in $y$--directions. The
detailed version of Theorem \ref{Theorem_holey_Aztec} is given in Proposition
\ref{Proposition_holey_Aztec_detailed}.

\subsection{Other tiling models and boundary conditions}
\label{Section_further_dimers}

We believe that the general method which we use here for the asymptotic analysis of holey hexagons
and Aztec rectangle is applicable to a wider class of models.

One example is gluings of arbitrary many trapezoids (for lozenge tilings) or
rectangles (for dominos) along a single axis, cf.\ Figure \ref{Figure_others}. This
leads to polygonal domains of high complexity and arbitrary topology. There are
several kinds of the holes, which these domains might have: the main distinction is
that some of them have continuous height function, as in the hexagon example, while
for others the heights are discontinuous. The strategy for these domains would
remain the same (analysis of log-gas on the vertical section and extension to 2d
fluctuations through Schur Generating Functions), however, the ingredients need to
be redeveloped. The results of \cite{BGG} are no longer enough for the analysis of
the necessary log--gases and in \cite{BorotGG} we develop much more general
machinery based on the extensions of Nekrasov (discrete loop) equations. The
asymptotic covariance is also complicated and matching it to the Gaussian Free Field
remains a future  challenge.

\begin{figure}[t]
\begin{center}
 {\scalebox{0.7}{\includegraphics{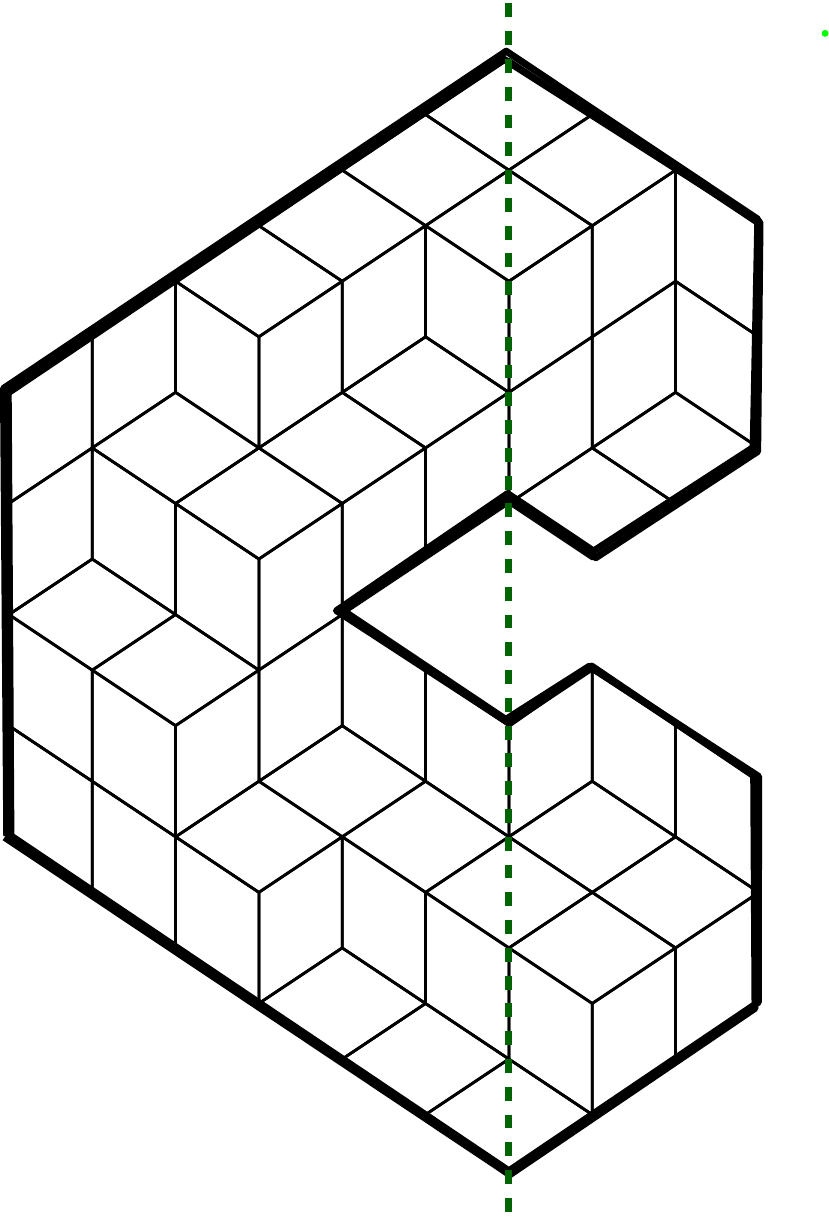}}} \quad {\scalebox{0.7}{\includegraphics{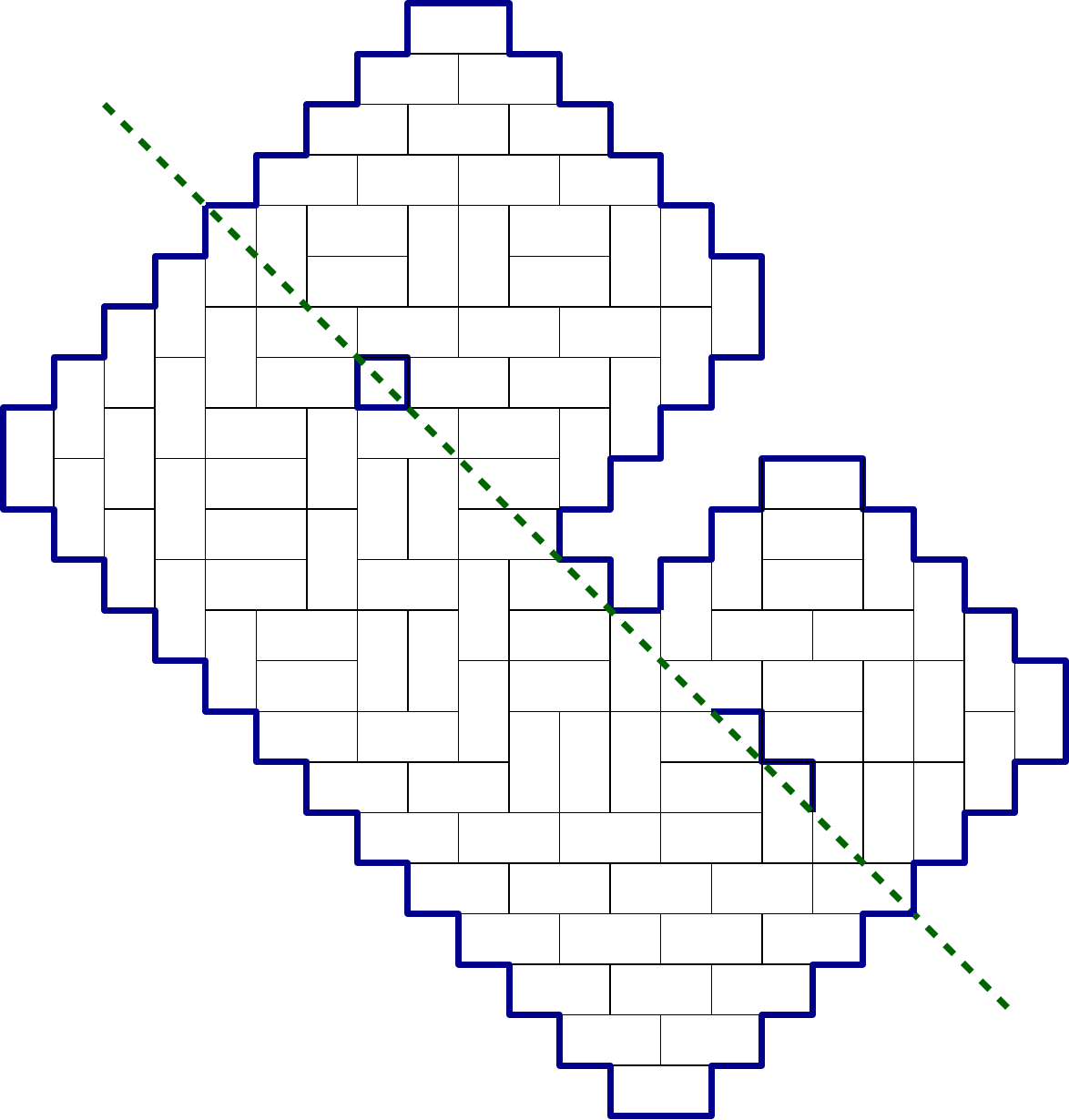}}}
 \caption{Left panel: a simply--connected domain obtained by gluing of 3 trapezoids. Right panel: a domain with two holes
 (of different types) obtained by gluing 3 rectangles}
 \label{Figure_others}
\end{center}
\end{figure}

\bigskip

In another direction, the Schur Generating Functions are applicable to the study of
other tilings and dimer models. One such application was recently demonstrated in
\cite{BL}. For other accessible models see \cite[Section 7]{BouCC}, \cite[Section
4]{BorFer_tile}, \cite{BBCCR}. We expect that the present method will find its way
into analysis of these models.

\subsection{Lozenge tilings of holey hexagons: proofs}
In this section we prove Proposition \ref{Proposition_lozenge_LLN} and Theorem
\ref{Theorem_holey_hex}. The proof proceeds in the following five steps:

\begin{enumerate}
 \item The horizontal lozenges along the section $x=t$ form a discrete log--gas (this was shown
     in \cite[Section 9.2]{BGG}), and therefore, LLN and CLT for them follows from the results of
     \cite{BGG}.
 \item We show how the covariance in CLT of \cite[Theorem 7.1]{BGG} can be used to construct the
     covariance function for the Gaussian Free Field.

 \item We construct the uniformization map, which is a conformal bijection between the liquid
     region $\mathcal L$ (with Kenyon--Okounkov's complex structure) and a Riemann sphere with
     two cuts. This is done by gluing together several applications of theorems of \cite{Petrov_GFF}, \cite{DM}, \cite{BG_CLT} on
     the parameterization of the liquid region in lozenge tilings of (simply--connected)
     trapezoids.

 \item Applying Theorem \ref{Theorem_main_multi} to the result of Step 1, we get the asymptotic
     of the partial derivatives of the logarithm of Schur generating functions, encoding the
     $x=t$ section.  Branching rules for Schur functions have the same combinatorics as lozenge
 tilings. Therefore, the combination of Theorem \ref{Theorem_CLT_multi_ext} with just obtained
 partial derivatives asymptotics leads to LLN and CLT jointly for any finite collection of
  sections $x=t_i$, $i=1,\dots,m$.
 \item We identify the covariance of Theorem \ref{Theorem_CLT_multi_ext} with that of the
     Gaussian Free Field constructed in Step 2. As an ingredient we use computations of
     \cite{BG_CLT} matching the global fluctuations for tilings of (simply-connected) trapezoids
     with the Gaussian Free Field.
\end{enumerate}

Let us start the detailed proof. We assume without loss of generality that $\hat C <  \hat B<\tau$
in \eqref{eq_hex_limit_regime}, so that the picture qualitatively looks like the left panel of
Figure \ref{Figure_hex_hole}. Other cases are studied in a similar way.

Our first task is to turn Theorem \ref{Theorem_holey_hex} into a numeric form by
presenting formulas for the liquid region $\mathcal L$ and the covariance of the
corresponding Gaussian Free Field. For that we need to fix notations and collect
previous results.

For a random tiling $\Omega$ of a hexagon with a hole, consider its section along the vertical line
$x=t$. It has $N=B+C-t-D=n_1+n_2$ horizontal lozenges, where $n_1$ is the number of lozenges below
the hole. Let us denote the positions of these lozenges through $\ell_1>\dots>\ell_N$ in the $y$
coordinate of Figures \ref{Figure_hex_split}, \ref{Figure_hex_hole_coord}. The results of
\cite[Section 9.2]{BGG} yield the Law of Large Numbers and the Central Limit Theorem for the
macroscopic fluctuations of $\ell_i$.

For the CLT we need to introduce the covariance  $\Cov(z,w)$. It depends only on two segments
$(\alpha_1,\beta_1)$, $(\alpha_2,\beta_2)$, which are intersections of the liquid region $\mathcal
L$ with the vertical line $x=\tau$:
\begin{multline}
\label{eq_covariance_two_cut} \Cov(z,w) =
 -\frac{1}{2(w-z)^2}+ \frac{c(z)}{\sqrt{\prod_{i=1}^2 (w-\alpha_i)(w-\beta_i)}}
\\+\frac{{\sqrt{\prod_{i=1}^2(z-\alpha_i)(z-\beta_i)}}}{2\sqrt{\prod_{i=1}^2(w-\alpha_i)(w-\beta_i)}}\Biggl(\frac{1}{(z-w)^2}-\frac{1}{2(z-w)}\sum_{i=1}^2
\left(\frac{1}{z-\alpha_i}+\frac{1}{z-\beta_i}\right)\Biggr),
\end{multline}
where $c(z)$ is a unique function of $z$, such that the integrals of
\eqref{eq_covariance_two_cut} in $w$ around the segments $(\alpha_1,\beta_1)$ and
$(\alpha_2,\beta_2)$ vanish. $c(z)$ can be expressed through elliptic integrals. For
the square root $\sqrt{x}$ we choose the branch which maps large positive real
numbers to large positive real numbers everywhere in \eqref{eq_covariance_two_cut}.

Another form of the same expression for $\Cov(z,w)$, which makes some of its
properties more transparent can be found e.g.\ in \cite{BorotGG}. It reads
\begin{multline}
\label{eq_covariance_two_cut_symmetric} \Cov(z,w) =
 \frac{1}{(z - w)^2}\left(- \frac{1}{2} + \frac{P_4(z,w)}{{\sqrt{\prod_{i=1}^2(z-\alpha_i)(z-\beta_i)}}{\sqrt{\prod_{i=1}^2(w-\alpha_i)(w-\beta_i)}}}\right) \\+
 \frac{P_2(z,w)}{{\sqrt{\prod_{i=1}^2(z-\alpha_i)(z-\beta_i)}}{\sqrt{\prod_{i=1}^2(w-\alpha_i)(w-\beta_i)}}},
\end{multline}
where $P_4$ and $P_2$ are symmetric polynomials in $z,w$ of degrees $4$ and $2$, respectively,
whose coefficients depend on $\alpha_1,\alpha_2,\beta_1,\beta_2$ in an explicit way. $P_4(z,w)$ is
such that
\begin{equation}
\label{eq_covariance_polynomial}
 \frac{P_4(z,w)}{{\sqrt{\prod_{i=1}^2(z-\alpha_i)(z-\beta_i)}}{\sqrt{\prod_{i=1}^2(w-\alpha_i)(w-\beta_i)}}}=\frac{1}{2}+O\bigl( (z-w)^2 \bigr),\quad (z-w)\to 0.
\end{equation}

 Although the formula
\eqref{eq_covariance_two_cut} might look complicated, but only its basic properties are important
for us:
\begin{itemize}
 \item $\Cov(z,w)$ is holomorphic in both variables outside $[\alpha_1,\beta_1]$,
     $[\alpha_2,\beta_2]$, i.e.\ there is no singularity near $z=w$.
 \item $\Cov(z,w)=\Cov(w,z)$.

 \item When $w$ crosses the real axis at point $x\in [\alpha_1,\beta_1]\cap[\alpha_2,\beta_2]$,
     $\Cov(z,w)$ makes a jump
\begin{multline}
\label{eq_covariance_two_cut_jump} \Cov(z,x+\ii 0)-\Cov(z,x-\ii 0) =
\frac{2}{(z - x)^2} \frac{P_4(z,x)}{{\sqrt{\prod_{i=1}^2(z-\alpha_i)(z-\beta_i)}}{\sqrt{\prod_{i=1}^2(x-\alpha_i)(x-\beta_i)}}} \\+
 \frac{ 2 P_2(z,x)}{{\sqrt{\prod_{i=1}^2(z-\alpha_i)(z-\beta_i)}}{\sqrt{\prod_{i=1}^2(x-\alpha_i)(x-\beta_i)}}},
\end{multline}
 \item The $w$--jump at $x\in [\alpha_1,\beta_1]\cup[\alpha_2,\beta_2]$ as a function of $z$ has
     a double pole at $z=x$ of the form $\frac{\pm 1}{(z-x)^2}$, with sign depending on whether
     $z$ approaches $x$ from the upper or lower halfplane.
\end{itemize}

\medskip

 If we lift $\Cov(z,w)$ to the Riemann surface of the equation
 $y^2=(x-\alpha_1)(x-\alpha_2)(x-\beta_1)(x-\beta_2)$, $x,y\in\mathbb C$, then it becomes closely
 related to ``fundamental bidifirential of the second kind'' (or sometimes ``Bergman kernel'') on
 this surface, see \cite{BorotGG} for more details.

\begin{theorem}[{\cite[Theorem 7.1 and Section 9.2]{BGG}}] \label{Theorem_hex_cut} Define the Cauchy--Stieltjes transform of random
particles $\ell_1>\dots>\ell_N$ introduced above, through
$$
 G_L(z)=\sum_{i=1}^N \frac{1}{z- \frac{\ell_i}{L} },
$$
then as $L\to\infty$ in the limit regime \eqref{eq_hex_limit_regime}, we have
\begin{equation}
 \lim_{L\to\infty} \frac{1}{L} G_L(z) = \int_{\mathbb R} \frac{\mu(x)}{z-x} dx, \text{ in
 probability},
\end{equation}
where $\mu(x)$ is a density of a compactly--supported measure of mass $(\hat B+\hat
C-\tau - \hat D)$, which satisfies $0\le \mu(x)\le 1$ everywhere; $0<\mu(x)<1$ only
on two intervals (``bands'') $(\alpha_1,\beta_1)$ and $(\alpha_2,\beta_2)$. Further,
$G_L(z)-\E G_L(z)$, converges as $L\to\infty$ to a Gaussian field in the sense of
moments, jointly and uniformly for finitely many $z$'s belonging to an arbitrary
compact subset of $\mathbb C\setminus \overline{ {\mathrm supp} [\mu(x)]}$. The
covariance $\lim_{L\to\infty}\left[ \E G_L(z) G_L(w) - \E G_L(z) \E G_L(w)\right]$
is given by $\Cov(z,w)$.
\end{theorem}
Note that \cite{BGG} used normalization by $N$, while we normalize by $L$, hence a slight
difference in the statements.

Although, it is possible to get certain formulas for the dependence of $\mu(x)$ and
its bands $(\alpha_1,\beta_1)$, $(\alpha_2,\beta_2)$ on the parameters
\eqref{eq_hex_limit_regime}, we will not address it here; we rather use the
knowledge of $\mu(x)$ and its bands as an input for all other formulas.

\bigskip

Topologically, the liquid region $\mathcal L$ of the holey hexagon is an annulus,
cf.\ Figure \ref{Figure_hex_hole}, and therefore, there is a conformal bijection of
it with the Riemann sphere with two cuts. We fix these cuts in a specific way by
considering the domain $\mathbb D=\mathbb C\setminus \bigl[(-\infty, \alpha_1)\cup
(\beta_1,\alpha_2)\cup (\beta_2,+\infty)\bigr]$, where
$\alpha_1<\beta_1<\alpha_2<\beta_2$ are endpoints of the bands from Theorem
\ref{Theorem_hex_cut}.

Below we construct an explicit uniformization map between $\mathcal L$ and $\mathbb D$. Before
doing that, let us define the Green function of the Laplace operator in $\mathbb D$ with Dirichlet
boundary conditions.

 Fix an arbitrary point $D\in\mathbb C$ and define for $z$ and $w$ in the upper half--plane the
 function

 \begin{multline}
\label{eq_GFF_cov_1}
  \mathcal G(z,w)=\frac{1}{4\pi}\Biggl[ \int_{D}^z \int_D^w \Cov(\mathfrak z,\mathfrak w) d\mathfrak z d\mathfrak w -
 \int_{D}^{\bar z} \int_D^w \Cov(\mathfrak z,\mathfrak w) d\mathfrak z d\mathfrak w
 \\ -\int_{D}^z \int_D^{\bar w} \Cov(\mathfrak z,\mathfrak w) d\mathfrak z d\mathfrak w
 +\int_{D}^{\bar z} \int_D^{\bar w} \Cov(\mathfrak z,\mathfrak w) d\mathfrak z d\mathfrak
 w\Biggr]-\frac{1}{2\pi}\ln\left|\frac{z-w}{z-\bar w}\right|.
 \end{multline}

 Note that $\mathcal G(z,w)$ is a real function (because the formula is unchanged under
conjugations).
 The integration contour in \eqref{eq_GFF_cov_1} lies inside $\mathbb C\setminus \bigcup_{i=1}^2
 [\alpha_i,\beta_i]$ (which is different from the domain $\mathbb D$), and  $\mathcal G(z,w)$ does
 not depend on the choice of such contour, as follows from the fact that integrals of
 $\Cov(\mathfrak z,\mathfrak w)$ around the bands vanish.

 Further,\
 if $z$ is in the upper halfplane, while $w$ is in the lower halfplane, then
 $\mathcal G(z,w)$ has a similar definition:
 \begin{multline}
\label{eq_GFF_cov_2}
  \mathcal G(z,w)=\frac{1}{4\pi}\Biggl[ \int_{D}^z \int_D^{\bar w} \Cov(\mathfrak z,\mathfrak w) d\mathfrak z d\mathfrak w -
 \int_{D}^{\bar z} \int_D^{\bar w} \Cov(\mathfrak z,\mathfrak w) d\mathfrak z d\mathfrak w
 \\ -\int_{D}^z \int_D^{w} \Cov(\mathfrak z,\mathfrak w) d\mathfrak z d\mathfrak w
 +\int_{D}^{\bar z} \int_D^{w} \Cov(\mathfrak z,\mathfrak w) d\mathfrak z d\mathfrak
 w\Biggr].
 \end{multline}
The differences between \eqref{eq_GFF_cov_1} and \eqref{eq_GFF_cov_2} are the
logarithmic term and conjugation of $w$. We further extend the definition to all
other non-real $z$ and $w$ by requiring
 \begin{equation}
 \label{eq_GFF_cov_3}
  \mathcal G(z,w)=\mathcal G(\bar z,\bar
  w).
 \end{equation}

\begin{proposition} \label{Proposition_G_continuous}
 The function $\mathcal G(z,w)$ of \eqref{eq_GFF_cov_1}, \eqref{eq_GFF_cov_2}, \eqref{eq_GFF_cov_3}
 can be uniquely extended to real $z$ and $w$, so that the resulting function of $(z,w)\in\mathbb
 C^2$ is continuous outside the diagonal $z=w$ and continuously differentiable in $\mathbb D^2\setminus \{z=w\}$.
 If $w\in \bigcup_{i=1}^2 [\alpha_i,\beta_i]$, then $G(\bar z,w)= G(z,w)$.
 If $w\in \mathbb R \setminus \bigcup_{i=1}^2 [\alpha_i,\beta_i]$, then $G(z,w)=0$.
\end{proposition}
\begin{proof}

 First, take $x\in \mathbb R \setminus \bigcup_{i=1}^2 [\alpha_i,\beta_i]$. If we send $w\to x$ in
 either upper or lower halfplane, then both formulas \eqref{eq_GFF_cov_1} and \eqref{eq_GFF_cov_2}
 turn into identical zeros. Therefore the continuity holds and $G(x,w)=0$.

 Next, take $x\in [\alpha_1,\beta_1]$ (the case $x\in [\alpha_2,\beta_2]$ is similar). Suppose that
 both $z$ and $w$ are in the upper half--plane and send $w\to x$. The last term involving logarithm
 in \eqref{eq_GFF_cov_1} vanishes. For the part involving the integrals, note that for $w$ in the
 upper half--plane, $\mathcal G(z,\bar w)$ given by \eqref{eq_GFF_cov_2} (with $w$ replaced by $\bar w$)
 coincides with the integral part of $\mathcal G(z,w)$ given in \eqref{eq_GFF_cov_1}. When $w\to
 x$, also $\bar w\to x$, and we conclude the continuity of $\mathcal G(z,w)$ at
 $w=x\in[\alpha_1,\beta_1]$. Note that we used $z\ne x$ here, as otherwise both logarithmic part of
 \eqref{eq_GFF_cov_1} and integrand would be exploding near $x$. Let us also emphasize, that
 $\mathcal G(x,w)$ does not have to vanish (as might seem from looking at \eqref{eq_GFF_cov_1}, \eqref{eq_GFF_cov_2}) --- this is because the integration contours from $D$ to
 $w$ and $\bar w$ are necessarily very different, as they are not allowed to cross
 $[\alpha_1,\beta_1]$.

 We proceed to continuity of the derivatives. Let us compute the (complex) $w$--derivative of
 \eqref{eq_GFF_cov_1}. Note that for the third and forth integrals the derivative vanishes, as they
 anti--holomorphic by definition, while for the first and second integrals the derivative equals to the integrand. For the logarithm we write
 \begin{equation}
 \label{eq_x24}
  \frac{1}{2\pi} \ln\left|\frac{z-w}{z-\bar w}\right|=\frac{1}{4\pi}\bigl( \ln (z-w)+\ln(\bar z-\bar w) - \ln(z-\bar w)-\ln(\bar z-w)\bigr),
 \end{equation}
 where the branches of the logarithms need to be chosen appropriately. The $w$--derivative of
 \eqref{eq_x24} is
 \begin{equation}
  \frac{1}{4\pi} \left( \frac{1}{w-z}-\frac{1}{w-\bar z} \right).
 \end{equation}
 Summing up, when both $z$ and $w$ are in the upper  halfplane we get
 \begin{equation}
 \label{eq_x25}
  \frac{\partial}{\partial w} \mathcal G(z,w)=\frac{1}{4\pi}\Biggl[ \int_{D}^z \Cov(\mathfrak z, w) d\mathfrak z -
 \int_{D}^{\bar z} \Cov(\mathfrak z,w) d\mathfrak z +\frac{1}{w-z}-\frac{1}{w-\bar z}\Biggr]
 \end{equation}
 Similarly, when  $w$ is in the lower halfplane, we get
\begin{equation}
\label{eq_x26}
  \frac{\partial}{\partial w} \mathcal G(z,w)=\frac{1}{4\pi}\Biggl[ -\int_{D}^z \Cov(\mathfrak z, w) d\mathfrak z +
 \int_{D}^{\bar z} \Cov(\mathfrak z,w) d\mathfrak z \Biggr]
 \end{equation}
 We need to show that as $w\to x\in[\alpha_1,\beta_1]$, the expressions \eqref{eq_x25} and
 \eqref{eq_x26} have the same limit. The difference between $w\to x$ limits of \eqref{eq_x25} and
 \eqref{eq_x25} is
\begin{multline}
 \label{eq_x27}
  \frac{\partial}{\partial w} \mathcal G(z,w)= \frac{1}{4\pi}\Biggl[ \int_{D}^z \bigl(\Cov(\mathfrak z, x+\ii 0) +\Cov(\mathfrak z, x-\ii 0)\biggr) d\mathfrak z \\-
 \int_{D}^{\bar z} \bigl(\Cov(\mathfrak z,x+\ii 0)+ \Cov(\mathfrak z,x-\ii 0)\bigr) d\mathfrak z +\frac{1}{x-z}-\frac{1}{x-\bar z}\Biggr]
\\=\frac{1}{4\pi}\Biggl[ -\int_{D}^z \frac{1}{(\mathfrak z-x)^2} +
 \int_{D}^{\bar z} \frac{1}{(\mathfrak z-x)^2}  d\mathfrak z +\frac{1}{x-z}-\frac{1}{x-\bar z}\Biggr]=0,
 \end{multline}
which proves continuity of the derivative. Note that the above argument would not
work for $x\in \mathbb R\setminus \bigcup_{i=1}^2  [\alpha_i,\beta_i]$ --- there are
no reasons for the $w$ derivative to be continuous there.

So far the variable $z$ was non-real, and it remains to study what happens when both $z$ and $w$
tend to the real axis.

First note that for $z$ in the lower half--plane we define $\mathcal G(z,w)$ as
$\mathcal G(\bar z,\bar w)$. When $w=x\in\bigcup_{i=1}^2 [\alpha_i,\beta_i]$,
$w=\bar w$ and we get $\mathcal G(z,x)=\mathcal G(\bar z,x)$. In turn, this readily
implies that the limits of $\mathcal G(z,x)$ as $z$ approaches a real point from
upper or lower halfplanes, coincide. For the complex $z$--derivative, we compute
using $w\to x$ limit of \eqref{eq_GFF_cov_1} for $z$ in the upper half--plane
\begin{equation}
\label{eq_x28}
 \frac{\partial}{\partial z} \mathcal G(z,x)=
\frac{1}{4\pi}\Biggl[ \int_D^{x+\ii 0} \Cov(z,\mathfrak w)  d\mathfrak w -\int_D^{x-\ii 0} \Cov(z,\mathfrak w)  d\mathfrak w
 \Biggr],
 \end{equation}
 and for $z$ in the lower halfplane we get
\begin{equation}
\label{eq_x29}
 \frac{\partial}{\partial z} \mathcal G(z,x)= \frac{\partial}{\partial z} \mathcal G(\bar z,x)=
 \frac{1}{4\pi}\Biggl[ -
  \int_D^{x+\ii 0} \Cov(z,\mathfrak w) d\mathfrak w
 +\int_D^{x-\ii 0} \Cov(z,\mathfrak w) d\mathfrak
 w\Biggr],
 \end{equation}
The $\frac{1}{2(z-\mathfrak w)^2}$ part of $\Cov(z,\mathfrak w)$ integrates to zero
in \eqref{eq_x28}, \eqref{eq_x29} (unless $z=x$), while the remaining part changes
its sign as $z$ crosses the real axis. We conclude that \eqref{eq_x28},
\eqref{eq_x29} have the same limits as $z$ approaches the real axis inside
$\bigcup_{i=1}^2 [\alpha_i,\beta_i]$.
\end{proof}

\begin{proposition} \label{Proposition_G_is_Green} $\mathcal G(z,w)$
(continuously extended to $z,w\in \mathbb D$, as in Proposition
\ref{Proposition_G_continuous}) is the Green function of the Laplace operator in
$\mathbb D$ with Dirichlet boundary conditions.
\end{proposition}
\begin{proof}
Let us check the defining properties of the Green function. First of all, we know
from Proposition \ref{Proposition_G_is_Green} that $\mathcal G(z,w)$ vanishes when
either $z$ or $w$ belong to the boundary of $\mathbb D$. Second, as $z$ or $w$ goes
to $\infty$, $\mathcal G(z,w)$ converges to $0$: for $w\to\infty$
 this is clear from the definition of $\mathcal G(z,w)$, and for $z\to\infty$ we use
 the symmetry of the covariance. Third, for fixed $z$, $\mathcal G(z,w)$ is harmonic
 everywhere in $\mathbb D\setminus{z}$, because it is a combination of holomorphic
 and anti-holomorphic functions. Additional care is required near
 $[\alpha_1,\beta_1]$, $[\alpha_2,\beta_2]$: here we use the fact harmonic function can be
 extended (preserving harmonicity) over a cut, if it is continuously differentiable
 near this cut.

 Finally, locally near
$z$ the function $\mathcal G(z,w)$ can
 be represented as a sum of $-\frac{1}{2\pi}\ln|z-w|$ and a harmonic function.

 At this point, the Green function is uniquely determined by the above four
 properties.
\end{proof}

The next step is to construct a conformal isomorphism $\Theta$ between  the liquid
region $\mathcal L$ with the complex structure given by the complex slope $\xi$ and
$\mathbb D$ with the standard complex structure.

For that we cut the hexagon along the vertical line going though the center of the
hole and extend each of the two pieces to a trapezoid as in Figure
\ref{Figure_hex_split} (cf.\ \cite{BG}, \cite{BG_CLT} for study of the uniformly
random lozenge tilings of trapezoids). The horizontal lozenges along the cut become
a part of both trapezoids: they form the common boundary of the left trapezoid of
width $t$ and right trapezoid of width $B+C-t$. Note that given the lozenges along
the cut, the tilings of left and right trapezoids are independent and uniformly
distributed. This property would allow us to use the results about liquid regions
for random tilings of trapezoids.

We identify $\ell_1>\dots>\ell_N$ of Theorem \ref{Theorem_hex_cut} with two
different signatures. One is $\lambda$ of rank $t$, which encodes the horizontal
lozenges on the border of the left trapezoid.  Another one is $\tilde \lambda$ of
rank $\tilde t:=B+C-t$, which encodes the horizontal lozenges on the border of the
right trapezoid.  Both these signatures are obtained by supplementing
$\ell_1>\dots>\ell_N$ with additional deterministic particles. Note that for $\tilde
\lambda$ the $y$--direction in the coordinate system is flipped, as shown in Figure
\ref{Figure_hex_split}.

\begin{figure}[t]
\begin{center}
 {\scalebox{0.65}{\includegraphics{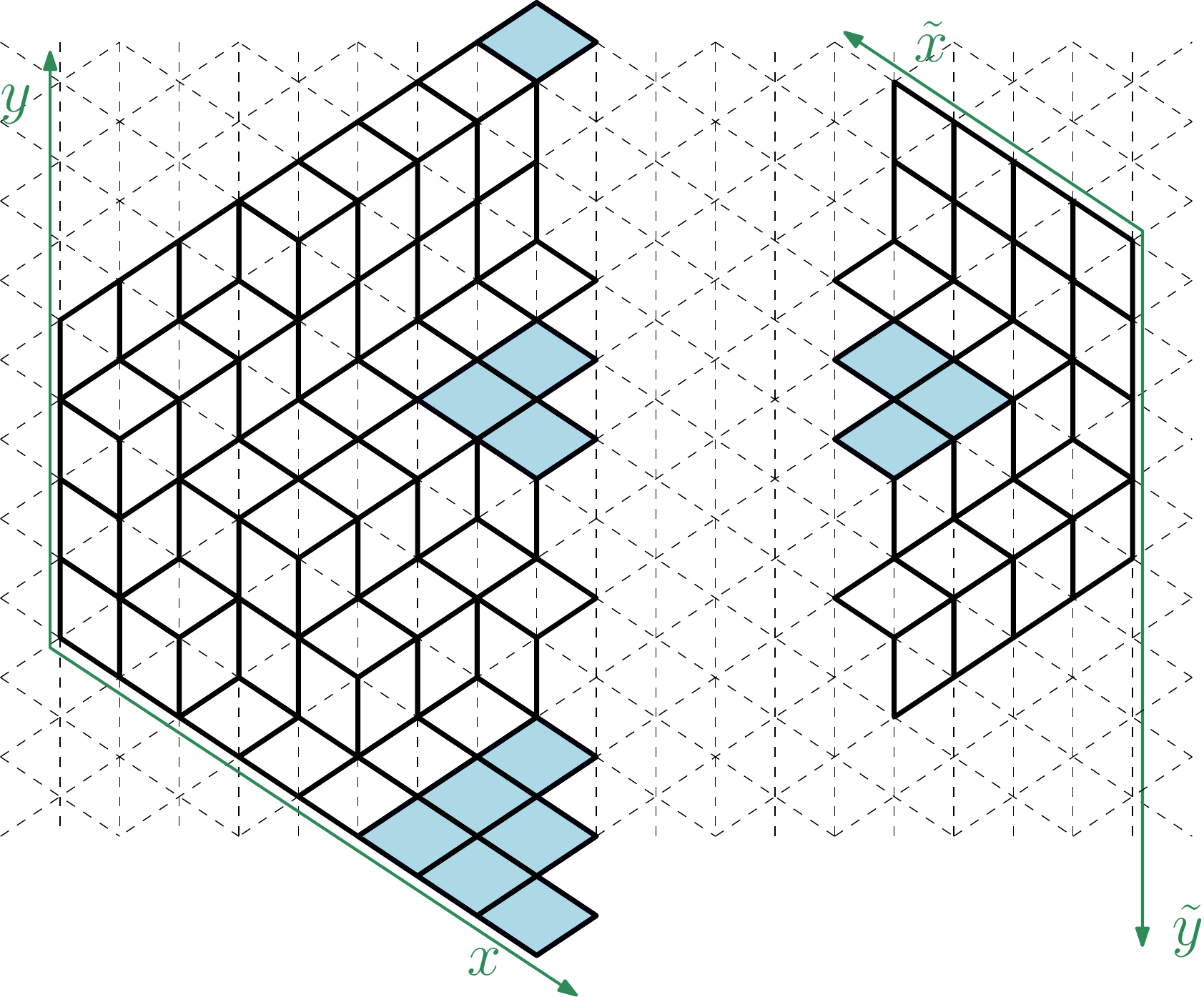}}}
 \caption{Hexagon with a hole is split into two trapezoids. The blue lozenges are deterministic and create parts of the boundary
for the hexagon and the hole. The ``native'' coordinate systems of trapezoids differ
by a central symmetry. \label{Figure_hex_split}}
\end{center}
\end{figure}

We obtain $\Theta$ as a \emph{gluing} of isomorphisms of the liquid region in the
left and right trapezoids with upper and lower half--planes, respectively. For each
trapezoid the isomorphisms are written in terms of Cauchy--Stieltjes transforms:
\begin{equation}
\label{eq_St_left_right}
 G(z)=\lim_{L\to\infty} \frac{1}{t} \sum_{i=1}^t \frac{1}{z-\lambda_i/t}, \quad
 \tilde G(z)=\lim_{L\to\infty} \frac{1}{\tilde t} \sum_{i=1}^{\tilde t} \frac{1}{z-\tilde \lambda_i/{\tilde t}},
\end{equation}
where the limits exist by Theorem \ref{Theorem_hex_cut}.

For $x\le \tau$ consider the equation
\begin{equation}
\label{eq_critical_equation_rescaled}
 y=z + \frac{\tau-x}{\exp\left(-G\left(\frac{z}{\tau}\right)\right)-1},\quad z\in \mathbb C.
\end{equation}
Let $\mathcal L^{\mathrm{left}}$ denote the set of points $(x,y)\subset
(0,\tau)\times \mathbb R$, such that the equation has a solution $z$ in the upper
half--plane $\mathbb U$. It is proven in
\cite{Petrov-curves}, \cite{DM}, \cite{BG_CLT}, \cite{G_Bulk} that $\mathcal
L^{\mathrm{left}}$ is precisely the liquid region of the left trapezoid (we rescaled
$x$, $y$, and $z$ by $\tau$, as compared to those articles). Moreover, for each
$(x,y)\in \mathcal L^{\mathrm{left}}$ \eqref{eq_critical_equation_rescaled} has a
\emph{unique} solution $z(x,y)$ in the upper-halfplane $\mathbb U$, and the map $z:
\mathcal L^{\mathrm{left}}\to \mathbb U$ is a smooth bijection, see \cite{DM}.

Differentiating \eqref{eq_critical_equation_rescaled} in $x$ and $y$, we get
\begin{equation}
\label{eq_critical_equation_rescaled_x}
 0=z_x\left(1+\frac{1}{\tau}G'\left(\frac{z}{\tau}\right)  \exp\left(-G\left(\frac{z}{\tau}\right)\right) \frac{\tau-x}{(\exp\left(-G\left(\frac{z}{\tau}\right)\right)-1)^2}\right) - \frac{1}{\exp\left(-G\left(\frac{z}{\tau}\right)\right)-1}
\end{equation}
\begin{equation}
\label{eq_critical_equation_rescaled_y}
 1=z_y\left(1 + \frac{1}{\tau}G'\left(\frac{z}{\tau}\right)  \exp\left(-G\left(\frac{z}{\tau}\right)\right) \frac{\tau-x}{(\exp\left(-G\left(\frac{z}{\tau}\right)\right)-1)^2}\right),
\end{equation}
which implies the relation
\begin{equation}
\label{eq_burgers_for_map}
 z_y=\left(\exp\left(-G\left(\frac{z}{\tau}\right)\right)-1\right) z_x
\end{equation}
In addition, $z(x,y)$ is related to the complex slope $\xi$ of Section
\ref{Section_GFF_hex_formulation} through
\begin{equation}
\label{eq_complex_structure_def}
 \xi(x,y)= \frac{\tau-x}{y-z(x,y)}=
\exp\left(-G\left(\frac{z}{\tau}\right)\right)-1 , \quad x\le \tau.
\end{equation}
 Note that the articles
\cite{Petrov-curves}, \cite{DM}, \cite{BG_CLT}, \cite{G_Bulk} where such relations
are established all use slightly different coordinate systems and the definitions of
$\xi$ (there is an ambiguity in the identification of angles of the triangle of
Figure \ref{Figure_triangle} with types of lozenges). For instance, in
\cite{Petrov-curves,Petrov_GFF}, the coordinate system $(\eta,\chi)$ is used, which
corresponds to our coordinate system by $\eta=x$, $\chi=-y$. The complex slope
$\Omega$ used there is related to $\xi$ by $\xi=\frac{\Omega-1}{\Omega}$, which is
one of the transformations mentioned in \eqref{eq_alternative_form}.

Either of the relations \eqref{eq_burgers_for_map}, \eqref{eq_complex_structure_def} imply that
$(x,y)\mapsto z(x,y)$ is indeed a conformal map in the complex structure $\xi$.

For the right trapezoid we take $x\ge\tau$ and consider the equation
\begin{equation}
\label{eq_critical_equation_rescaled_right}
 y=\tilde z + \frac{\tau-x}
{\exp\left(-\tilde G\left(\frac{\hat B +\hat A -\tilde z}{\hat B+\hat C-\hat
\tau}\right)\right)-1},\quad \tilde z\in \mathbb C,
\end{equation}
which is the same equation \eqref{eq_critical_equation_rescaled} written in terms of the
coordinates $\tilde x$, $\tilde y$, with $G$ replaced  by $\tilde G$, $\tau$ replaced by $\hat B +
\hat C -\tau$, and then reexpressed back in $x=\hat B+\hat C -\tilde x$, $y=\hat B+\hat A- \tilde
y$, $\tilde z = \hat B +\hat A - z$ (see Figure \ref{Figure_hex_split} for coordinate systems). We
let $\mathcal L^{\mathrm{right}}$ be the set of points $(x,y)\in (\tau,\hat B+\hat C)\times \mathbb
R$ such that \eqref{eq_critical_equation_rescaled_right} has a unique solution in the \emph{lower }
half-plane  $\overline{\mathbb U}$ (as before, it can not have more solutions), and denote this
solution as $\tilde z(x,y)$. Then $\tilde z: \mathcal L^{\mathrm{right}}\to \overline{\mathbb U}$
is a smooth bijection, and $\tilde z$ is related to the complex slope $\xi$ in a conformal way.

The bijectivity and relation to the complex slope is obtained by the above change of coordinates
from the results about the map $z\to z(x,y)$ of \eqref{eq_critical_equation_rescaled}.

\begin{remark}
 General results on the solution of the variational problem for the limit shape of
lozenge tilings guarantee that the complex slope $\xi(x,y)$ continuously depends on
$(x,y)$, see \cite{dSS}. In our particular case the continuity of $\xi(x,y)$ given
 is immediate outside the line $x=\tau$ (due to explicit formulas), while on the latter line one needs additional
arguments to see the continuity.
\end{remark}

Define the \emph{uniformization map} $\Theta:\mathcal L \to \mathbb D$ through
\begin{equation}
\label{eq_Def_theta} \Theta(x,y)=
\begin{cases}
 z(x,y),& x< \tau,\\
 \tilde z(x,y), & x> \tau,\\
 y ,& x=\tau,\, y \in (\alpha_1,\beta_1)\cup (\alpha_2,\beta_2),
\end{cases},
 \end{equation}
where $\alpha_1,\alpha_2,\beta_1,\beta_2$ are endpoints of the bands in Theorem
\ref{Theorem_hex_cut}.

\begin{proposition} The map $\Theta$ of \eqref{eq_Def_theta} is a bijection between
the liquid region $\mathcal L$ of the holey hexagon and $\mathbb D=\mathbb
C\setminus[(-\infty,\alpha_1)\cup (\beta_1,\alpha_2)\cup(\beta_2,+\infty)]$. Moreover, $\Theta$ is
a conformal map between the complex structure in $\mathcal L$ of the complex slope $\xi$ and the
standard complex structure in $\mathbb D$.
\end{proposition}
\begin{proof} According to the above definitions and results of \cite{DM}, \cite{BG_CLT},
\cite{G_Bulk}, inside the left trapezoid the liquid region is $\mathcal
L^{\mathrm{left}}$, inside the right trapezoid the liquid region is $\mathcal
L^{\mathrm{right}}$. Along the cut $x=\tau$, the liquid region is the union of two
bands $(\alpha_1,\beta_1)$ and $(\alpha_2,\beta_2)$. Therefore, $\mathcal L=
\mathcal L^{\mathrm{left}}\cup \mathcal L^{\mathrm{right}} \cup \{(\tau,y)\mid y\in
(\alpha_1,\beta_1)\cup (\alpha_2,\beta_2) \}$ and $\Theta$ maps it to $\mathbb D$.
Since each of the three cases in \eqref{eq_Def_theta} gives a bijection and together
they form a disjoint cover of $\mathcal L$ and $\mathbb D$, we conclude that
$\Theta$ is a bijection.

The restriction of $\Theta$ to $\mathcal L^{\mathrm{left}}$ and $\mathcal L^{\mathrm{right}}$ are
conformal in the complex slope $\xi$ due to definitions and the results of \cite{DM},
\cite{BG_CLT}, \cite{G_Bulk}. In addition, $\Theta$ is continuous near the cut $x=\tau$, since
formulas \eqref{eq_critical_equation_rescaled}, \eqref{eq_critical_equation_rescaled_right} turn
into $y=z$, $y=\tilde z$, respectively, as $x\to \tau$. Since holomorphic function can be extended
(in a holomorphic way) over a cut, where it is continuous, we conclude that $\Theta$ is conformal
in $\mathcal L$.
\end{proof}

\bigskip

With definitions of $\mathcal G$ and $\Theta$ in hand, we state a detailed version
of Theorem \ref{Theorem_holey_hex}.

\begin{proposition}
\label{Proposition_holey_hex_detailed}
 Let $\mathrm{Hex}$ be the interior of the holey hexagon in the coordinate system of
 Figure \ref{Figure_hex_hole_coord}, and let $H_L(x,y)$ be the random height function
of uniformly random lozenge tiling of $\mathrm{Hex}$ in the limit regime
\eqref{eq_hex_limit_regime}. Further, let $\mathcal L$ be the liquid region, let
$\mathbb D=\mathbb C\setminus[(-\infty,\alpha_1)\cup
(\beta_1,\alpha_2)\cup(\beta_2,+\infty)]$, let $\Theta: \mathcal L \to \mathbb D$ be
the uniformization map \eqref{eq_Def_theta}, and let $\mathcal G$ be the Green
function of $\mathbb D$, as in \eqref{eq_GFF_cov_1}, \eqref{eq_GFF_cov_2},
\eqref{eq_GFF_cov_3}.

 For any $m=1,2,\dots$ take a collection of $m$ points $0< x_1,x_2,\dots,x_m< \hat B
+\hat C$ and $m$ polynomials $f_1,\dots,f_m$. Then the random variables
\begin{equation}
\label{eq_polynomial_observables}
 \int_{y:\, (x_i,y)\in \mathrm{Hex}} f_i(y) \bigl( H_L(x_i,y)-\E H_L(x_i,y) )
dy,\quad i=1,\dots,m,
\end{equation}
converge as $L\to\infty$ (in the sense of moments) to a centered Gaussian
$m$--dimensional random vector. The asymptotic covariance of $i$th and $j$th
components is
\begin{equation}
\label{eq_target_covariance}
 \frac{1}{\pi} \int_{y_i:\, (x_i,y_i)\in \mathcal L} \int_{y_j:\, (x_j,y_j)\in \mathcal L}
f_i(y_i) f_j(y_j) \mathcal G\bigl( \Theta(x_i,y_i), \Theta(x_j,y_j) \bigr)    dy_i d
y_j, \quad 1\le i,j \le m.
\end{equation}
\end{proposition}
Note that the polynomial observables \eqref{eq_polynomial_observables} are dense in
any reasonable set of test functions on $\mathcal L$; they uniquely define the
covariance of the limiting Gaussian Free Field.

In the rest of this section we prove Proposition
\ref{Proposition_holey_hex_detailed} by combining Theorems \ref{Theorem_main_multi},
\ref{Theorem_CLT_multi_ext} with Theorem \ref{Theorem_hex_cut}.

We start by explaining that the result of Theorem \ref{Theorem_hex_cut} agrees with
the statement of Proposition \ref{Proposition_holey_hex_detailed}.

\begin{proposition} \label{Prop_covariance_section_identification}
 In the setting of Theorem \ref{Theorem_hex_cut} we have convergence in probability
and in the sense of moments
 $$
  \frac{1}{N}\sum_{i=1}^N  \left(\frac{\ell_i}{N}\right)^k\to \int_{\mathbb R} x^k
  \mu(x) dx,\quad k=1,2,\dots.
 $$
 Moreover, the random variables
 $$
  \sum_{i=1}^N \left[  \left(\frac{\ell_i}{N}\right)^k -\E
  \left(\frac{\ell_i}{N}\right)^k \right], \quad k=1,2,\dots,
 $$
 become Gaussian as $L\to\infty$ with asymptotic covariance of $k$th and $m$th components given
 by
 \begin{equation}
 \label{eq_covariance_hex_section}
  \frac{1}{(2\pi \ii)^2} \oint\oint z^k w^m \Cov(z,w) dz dw,
 \end{equation}
 where integration goes over a large positively oriented contour, enclosing both
 $(\alpha_1,\beta_1)$ and $(\alpha_2,\beta_2)$, and $\Cov(z,w)$ was defined in \eqref{eq_covariance_two_cut}.
 An equivalent form of
\eqref{eq_covariance_hex_section} is
 \begin{equation}
\label{eq_cov_polynomial_section}
 \frac{km}{\pi} \iint_{ [(\alpha_1,\beta_1)\cup (\alpha_2,\beta_2)]^2} x^{k-1} y^{m-1} \mathcal G(x,y)
  \,dx\,
  dy.
 \end{equation}
\end{proposition}
\begin{proof}
 Note that
 $$
   \sum_{i=1}^N  \left(\frac{\ell_i}{N}\right)^k=\frac{1}{2\pi \ii} \oint z^k G_N(z) dz
 $$
 with integration over a large positively oriented contour, and apply Theorem
 \ref{Theorem_hex_cut}. This yields Gaussian convergence and formula
\eqref{eq_covariance_hex_section}.

It remains to show that \eqref{eq_covariance_hex_section} equals
\eqref{eq_cov_polynomial_section}. Let us integrate
\eqref{eq_covariance_hex_section} by parts in $z$ and in $w$.  We get
 \begin{equation}
 \label{eq_x15}
  \frac{km}{(2\pi \ii)^2 } \oint\oint z^{k-1} w^{m-1}
  \left[\int_{D}^z \int_D^w \Cov(\mathfrak z,\mathfrak w) d\mathfrak z d\mathfrak w  \right]dz dw,
 \end{equation}
Deform  both contours in \eqref{eq_x15} to closely follow the real line. In more
details, we choose large $J>0$ and small $\eps>0$, integrate over the segment
$[-J-\ii \eps,J-\ii \eps]$ in positive direction and then over
$[-J+\ii\eps,J+\ii\eps]$ in the negative direction. Let us now send $\eps\to 0$.
Note that the term $1/(w-z)^2$ after integration in $z$ and $w$ becomes $\ln(z-w)$.
whose singularity at $z=w$ is mild, i.e.\ integrable. Therefore, the integral does
not explode and we can send $\eps\to 0$ directly, getting
\begin{multline}
\label{eq_x17}
  \frac{km}{(2\pi \ii)^2 } \iint_{[-J,J]^2} x^{k-1} y^{m-1}
  \Biggl[\\ \int_{D}^{x+\ii 0} \int_D^{y+\ii 0} \Cov(\mathfrak z,\mathfrak w) d\mathfrak z d\mathfrak w-
        \int_{D}^{x-\ii 0} \int_D^{y+\ii 0} \Cov(\mathfrak z,\mathfrak w) d\mathfrak z d\mathfrak w \\-
        \int_{D}^{x+\ii 0} \int_D^{y-\ii 0} \Cov(\mathfrak z,\mathfrak w) d\mathfrak z d\mathfrak w+
        \int_{D}^{x-\ii 0} \int_D^{y-\ii 0} \Cov(\mathfrak z,\mathfrak w) d\mathfrak z d\mathfrak w
   \Biggr]dx dy.
\end{multline}
Outside $[\alpha_1,\beta_1]\cup [\alpha_2,\beta_2]$, $\Cov(\mathfrak z,\mathfrak w) d\mathfrak z
d\mathfrak w$ is analytic, and the values of $\int_{D}^{x\pm \ii 0} \int_D^{y\pm \ii 0}
\Cov(\mathfrak z,\mathfrak w) d\mathfrak z d\mathfrak w$ are all the same, and therefore, the
integrand vanishes. We conclude that integration domain in \eqref{eq_x17} can be restricted from
$[-J,J]^2$ to $([\alpha_1,\beta_1]\cup [\alpha_2,\beta_2])^2$. On the other hand, observing that
$\ln\left|\frac{z-w}{z-\bar w}\right|$ vanishes for real $z$ and $w$, we identify the integrand
with $\mathcal G(x,y)$. Therefore, \eqref{eq_x17} is identical to
\eqref{eq_cov_polynomial_section}, as desired.
\end{proof}

Recall that we identified $\ell_1>\dots>\ell_N$ of Theorem \ref{Theorem_hex_cut}
with two different signatures: $\lambda$ of rank $t$, and $\tilde \lambda$ of rank
$\tilde t:=B+C-t$, according to splitting into two trapezoids of Figure
\ref{Figure_hex_split}. Further, the horizontal lozenges on the vertical line at
distance $q=1,2,\dots, t$ from the left boundary of the hexagon are inside the left
trapezoid and, therefore, are identified with a signature of rank $q$. The
horizontal lozenges on the vertical line at distance $r=1,2,\dots,B+C-t$ from the
right boundary are inside the right trapezoid and, therefore, identified with a
signature of rank $r$.

Let us recall that the branching rules for Schur polynomials match the combinatorics
of the lozenge tilings of trapezoids (see e.g.\ \cite[Section 3.2]{BG},
\cite[Section 3.5]{BG_CLT}). In particular, if a signature $\mu$ encodes $q$
lozenges along the vertical line $x=q$, then the conditional law
$\mathrm{Prob}(\mu\mid \lambda)$ can be found from the decomposition
\begin{equation}
\label{eq_lozenge_through_Schur}
 \frac{s_\lambda(x_1,\dots,x_q,1^{t-q})}{s_\lambda(1^t)}=\sum_{\mu\in\GT_q}
 \mathrm{Prob}(\mu\mid \lambda) \frac{s_\mu(x_1,\dots,x_q)}{s_\mu(1^q)}.
\end{equation}
Thus, the finite--dimensional distributions of all these signatures for
$q=1,2,\dots, t$, $r=1,2,\dots,B+C-t$ fit into the framework of Section
\ref{Section_CLT_extension} with the following specifications:

\begin{equation}
\label{eq_specifications_lozenge}
 H=2, \qquad g_{h,m}=1,\qquad N_1=t,\qquad N_2=B+C-t,
\end{equation}
and therefore, Theorem \ref{Theorem_CLT_multi_ext} holds.

Our main object of interest is the  height function $H_L(x,y)$, as in Section
\ref{Section_GFF_hex_formulation}, so let us link it to the power sums of Theorem
\ref{Theorem_CLT_multi_ext}.

\begin{figure}[t]
\begin{center}
 {\scalebox{0.65}{\includegraphics{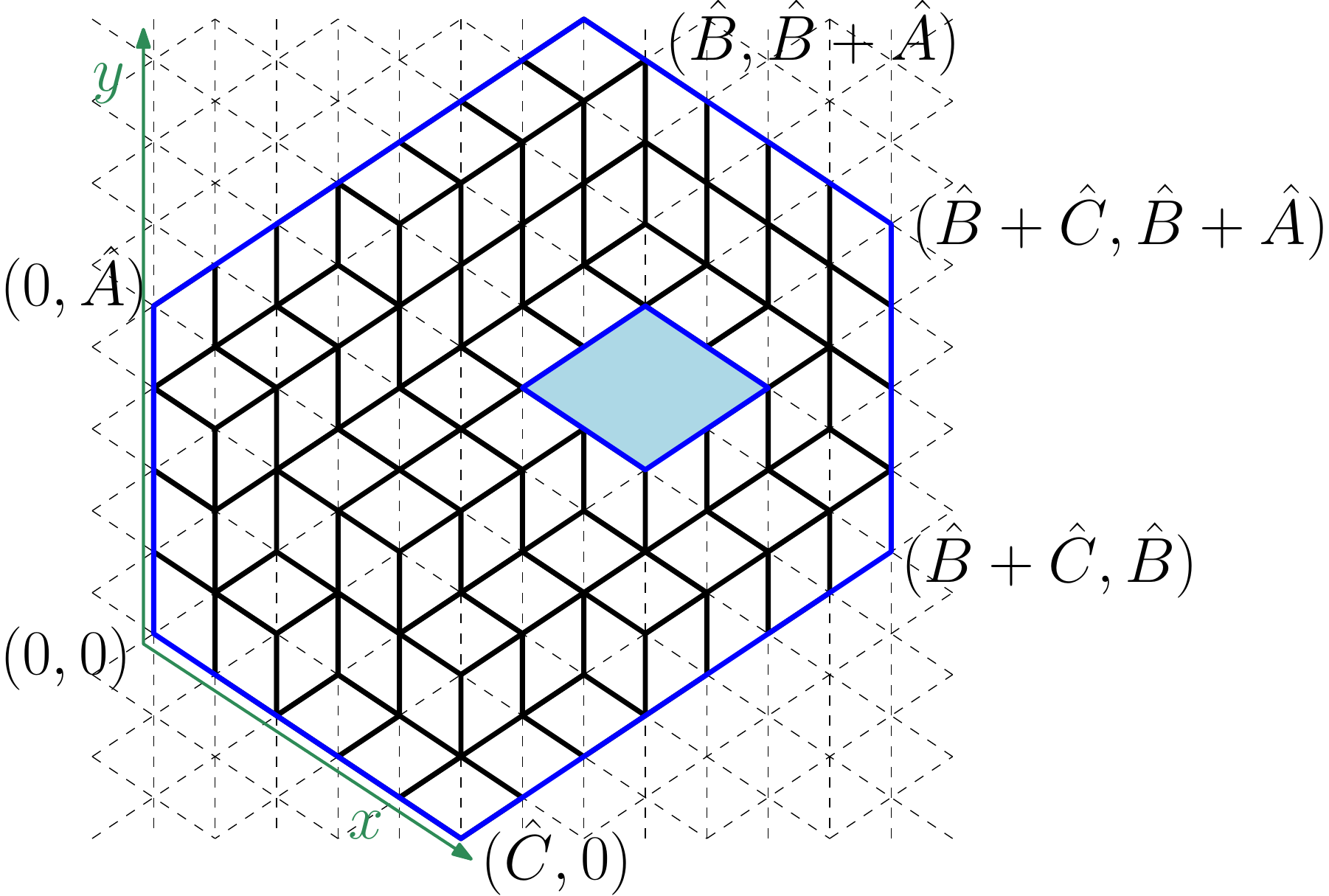}}}
 \caption{Coordinate system for the hexagon with holes. The coordinates of vertices are shown.}
 \label{Figure_hex_hole_coord}
\end{center}
\end{figure}

 We use the rescaled coordinate system, as
in Figure \ref{Figure_hex_hole_coord}. For each $x\in [0,\hat B+\hat C]$ and
$k=1,2\dots$ define
\begin{multline}
\label{eq_polynomial_observables_2}
 p^L(x;k)\\=\begin{dcases} \int\limits_0^{x+\hat A} H_L(x,y) y^k dy, & x\le \hat C,\\
                        \int\limits_{x-\hat C}^{x+\hat A} H_L(x,y) y^k dy +
\int\limits_{x-\hat C}^{x+\hat A} (x-\hat C) y^k dy +\int\limits_0^{x-\hat C} y^k dy
, & \hat C\le x< \hat B,\\
\int\limits_{x-\hat C}^{\hat A+\hat B} H_L(x,y) y^k dy +\int\limits_{x-\hat
C}^{\hat A+\hat B-x} (x-\hat C) y^k dy +\int\limits_0^{x-\hat C} y^k dy
+\int\limits_{\hat A + \hat
B-x}^{\hat A +x} y^k dy, & \hat B\le x< \tau,\\
\int\limits_{x-\hat C}^{\hat A +\hat B} H_L(x,y) \bigl(\hat A+\hat B-y \bigr)^k
dy, & x> \tau.
\end{dcases}
\end{multline}

In all the cases of the definition of $p^L(x;k)$, the random part is the first integral, which is a
pairing of the height function against a $\delta$--function in $x$--direction multiplied by a
polynomial in $y$--direction. The remaining integrals are deterministic terms, which are introduced
for the convenience of matching with Theorem \ref{Theorem_CLT_multi_ext}; they all disappear when
we pass to the centered versions $p^L(x;k)-\E p^L(x;k)$, and therefore play no role in the central
limit theorem. The change from $y$ to $\hat A+\hat B-y$ in the forth case is related to the switch
to $\tilde y$ coordinates in Figure \ref{Figure_hex_split}.

Note that our definitions change in a discontinuous way near $x=\tau$: two different
definitions are related by an affine transformation. Precisely at $x=\tau$ one can
stick to any of the definitions, let us agree to use the $x<\tau$ one for being
definite.

\begin{corollary} \label{Corollary_hex_LLN_CLT}
 For any $0< x<\hat B + \hat C$ and $k=1,2,\dots$, the random variables
$\frac{1}{L}p^L(x;k)$ converge as $L\to\infty$ to constants (in the sense of
moments). Further, the centered random variables $p^L(x;k)-\E p^L(x;k)$ are
asymptotically Gaussian (jointly in finitely many $x$'s and $k$'s).
\end{corollary}
Note that the convergence of $\frac{1}{L}p^L(x;k)$ implies Proposition
\ref{Proposition_lozenge_LLN}.
\begin{proof}[Proof of Corollary \ref{Corollary_hex_LLN_CLT}]
 Assume without loss of generality that $x<\hat C$. Integrating by parts, and using the vanishing of the integrand both at $0$ and at $x+\hat A$,
 we rewrite
\begin{multline}\label{eq_integrate_by_parts}
 p^L(x;k)=
\int\limits_0^{x+\hat A} \left(\sum_{i=1}^{\lfloor Lx \rfloor} {\mathbf
1}_{\lambda^{\lfloor Lx \rfloor}_i +\lfloor Lx \rfloor -i \ge Ly}
 \right) y^k dy
=-\frac{1}{k+1} \sum_{i=1}^{\lfloor Lx \rfloor} \left( \frac{\lambda^{\lfloor Lx \rfloor}_i +\lfloor Lx \rfloor -i}{L}\right)^{k+1}
\end{multline}
At this point, the statement becomes an application of Theorem
\ref{Theorem_CLT_multi_ext}. Indeed, we explained above that lozenge tilings are in
the framework of this theorem with specifications \eqref{eq_specifications_lozenge}.
On the other hand, the measures are CLT--appropriate by the combination of
Proposition \ref{Prop_covariance_section_identification} and Theorem
\ref{Theorem_main_multi}.
\end{proof}

Note that polynomial observables \eqref{eq_polynomial_observables}  of Proposition
\ref{Proposition_holey_hex_detailed} are linear combinations of $p^L(x;k)$.
Therefore, Corollary \ref{Corollary_hex_LLN_CLT} implies the asymptotic Gaussianity
in Proposition \ref{Proposition_holey_hex_detailed}. It now remains to identity the
asymptotic covariance of the random variables $p^L(x;k)-\E p^L(x;k)$ with that of
sections of the Gaussian Free Field, i.e.\ with formulas
\eqref{eq_target_covariance}.

The basic idea is to treat this Gaussian Free Field as a sum of three independent
parts: the GFF (with Dirichlet boundary conditions) in the liquid region of the left
trapezoid, the GFF (with Dirichlet boundary conditions) in the liquid region of the
right trapezoid, and the 1d random field along the common boundary of the trapezoid,
extended as a harmonic function to the liquid regions of the left and right
trapezoids. Such splitting is clearly visible in our formulas \eqref{eq_GFF_cov_1},
\eqref{eq_GFF_cov_2}: the $-\frac{1}{2\pi} \ln\left|\frac{z-w}{z-\bar w}\right|$
part (treated as two different functions, one for $z,w$ in the upper half--plane and
another one for $z,w$ in the lower halfplane) is the Green function of the upper (or
lower) halfplane; the sum of the four integrals is the harmonic extension from the
common boundary of the trapezoids.

This approach is a shadow of the well--known restriction property of the Gaussian Free Field (or,
equivalently, of the Green function encoding its covariance).

\bigskip

\begin{proposition} The covariance in Corollary \ref{Corollary_hex_LLN_CLT}
is $\frac{1}{\pi}$ times that of the $\Theta$--pullback of the Gaussian Free Field
with Dirichlet boundary conditions in $\mathbb D$. In more details, take any
$k_1,k_2=1,2,\dots$. If $0<x_1,x_2\le \tau$, then
\begin{multline}
\label{eq_target_left_left}
 \lim_{L\to\infty} \E p^L(x_1;k_1) p^L(x_2;k_2)-\E p^L(x_1;k_1) \E p^L(x_2;k_2)\\
= \frac{1}{\pi} \int_{y_1:\, (x_1,y_1)\in \mathcal L} \int_{y_2:\, (x_2,y_2)\in
\mathcal L}
 (y_1)^{k_1} (y_2)^{k_2} \mathcal G\bigl( \Theta(x_1,y_1), \Theta(x_2,y_2) \bigr)
dy_1 dy_2.
\end{multline}
If $0<x_1\le \tau < x_2 < \hat B+\hat C$, then
\begin{multline}
\label{eq_target_left_right}
 \lim_{L\to\infty} \E p^L(x_1;k_1) p^L(x_2;k_2)-\E p^L(x_1;k_1) \E p^L(x_2;k_2)\\
= \frac{1}{\pi} \int_{y_1:\, (x_1,y_1)\in \mathcal L} \int_{y_2:\, (x_2,y_2)\in
\mathcal L}
 (y_1)^{k_1} (\hat A+\hat B- y_2)^{k_2} \mathcal G\bigl( \Theta(x_1,y_1), \Theta(x_2,y_2) \bigr)
dy_1 dy_2.
\end{multline}
If $\tau < x_1,x_2 < \hat B+\hat C$, then
\begin{multline}
\label{eq_target_right_right}
 \lim_{L\to\infty} \E p^L(x_1;k_1) p^L(x_2;k_2)-\E p^L(x_1;k_1) \E p^L(x_2;k_2)\\
= \frac{1}{\pi} \int_{y_1:\, (x_1,y_1)\in \mathcal L} \int_{y_2:\, (x_2,y_2)\in
\mathcal L}
 (\hat A+\hat B-y_1)^{k_1} (\hat A+\hat B-y_2)^{k_2} \mathcal G\bigl( \Theta(x_1,y_1), \Theta(x_2,y_2) \bigr)
dy_1 dy_2.
\end{multline}
\end{proposition}
\begin{proof}
 We compute the asymptotic covariance as $L\to\infty$ of $p^L(x_1;k_1)$ and
$p^L(x_2; k_2)$ of \eqref{eq_polynomial_observables_2}. The definition of the random
part of $p^L(x;k)$ depends on whether $x<\tau$ or $x>\tau$, leading to three cases
\eqref{eq_target_left_left}, \eqref{eq_target_left_right},
\eqref{eq_target_right_right}. We start from the second case $0< x_1\le \tau \le
x_2<\hat B+\hat C$.

We use the formulas of Theorem \ref{Theorem_CLT_multi_ext}. Note that the numbers $d_{a,h_1,m_1;b,
h_2,m_2}$ there do not depend on $m_1$, $m_2$ --- this is because all $g$--functions were set to
$1$. Therefore, we will denote these numbers simply $d_{a,h_1; b,h_2}$. We now ``compute'' them by
applying Theorem \ref{Theorem_CLT_multi_ext} to evaluate the asymptotic covariance of $\sum_{i=1}^t
(\frac{\lambda_i+t-i}{t})^k$ and $\sum_{i=1}^{\tilde t} (\frac{\tilde \lambda_i+\tilde t-i}{\tilde
t})^m$ (recall that $\tilde t=B+C-t$). Multiplying the result of \eqref{eq_CLT_multi_formula_ext}
by $u^{-k-1} v^{-m-1}$ with large complex numbers $u$ and $v$, and then summing over all $k$ and
$m$ and rewriting the residue computation as a contour integral we get
\begin{equation}
\label{eq_covariance_Stieltjes}
 \cov\left(\frac{1}{u-x},t; \frac{1}{v-x},\tilde t \right)= \frac{1}{(2\pi\ii)^2}
\oint\oint \frac{1}{u-F(z)} \cdot \frac{1}{v-\tilde F(w)} Q(z,w) dz dw,
\end{equation}
where $\cov\left(\frac{1}{u-x},t; \frac{1}{v-x},\tilde t \right)$ means the
asymptotic covariance of the random variables
$$
 \sum_{i=1}^t
\frac{1}{u- \frac{\lambda_i+t-i}{t}}\text{ and }\sum_{i=1}^{\tilde t}
\frac{1}{v-\frac{\tilde \lambda_i+\tilde t-i}{\tilde t}},
$$
$$
 F(z)=(z+1)\left( \frac{1}{z}+\sum_{a=1}^\infty \frac{\c_{a;1}
z^{a-1}}{(a-1)!}\right),\quad  \tilde F(w)=(w+1)\left( \frac{1}{w}+\sum_{a=1}^\infty
\frac{\c_{a;2} w^{a-1}}{(a-1)!}\right),
$$
and
$$
 Q(z,w)=\sum_{a,b=1}^{\infty} \frac{\d_{a,1;b,2}}{(a-1)!(b-1)!} z^{a-1} w^{b-1}.
$$
The integration goes over small contours around $0$, which enclose both $1/u$ and $1/v$. Note that
the non-deterministic parts of $\lambda$ and $\tilde \lambda$ are essentially the same and can be
identified with random $N$--tuple $\ell_1>\dots>\ell_N$ studied in Theorem \ref{Theorem_hex_cut}.
Therefore, the left-hand side of \eqref{eq_covariance_Stieltjes} is the function $\tau (\hat B+\hat
C-\tau)\Cov\bigl(\tau u, \hat B+\hat A-(\hat B+\hat C-\tau) v \bigr)$, where $\Cov$ is given by
\eqref{eq_covariance_two_cut}. The substraction from $\hat B+\hat A$ in the argument is the change
between $y$ and $\tilde y$ in two coordinate systems of Figure \ref{Figure_hex_hole_coord}. On the
other hand, the integrand in the right-hand side of \eqref{eq_covariance_Stieltjes} has a unique
(simple) pole inside the contours at the point $(z,w)$ satisfying $u=F(z)$, $v-\tilde F(w)$, and
therefore the integral is computed as the residue at this point. We conclude that
\begin{equation}
\label{eq_covariance_through_Q}
  Q(z,w)=\tau (\hat B+\hat
C-\tau) \, \Cov\bigl(\tau F(z), \hat B+\hat A-(\hat B+\hat C-\tau)\tilde F(w)\bigr)\, {F'(z) \tilde F'(w)}.
\end{equation}
With this identity, we can now apply Theorem \ref{Theorem_CLT_multi_ext} and
integration by parts, as in \eqref{eq_integrate_by_parts},  to compute the desired
asymptotic covariance of $p^L(x_1;k_1)$ and $p^L(x_2; k_2)$ as
\begin{multline}
\label{eq_x19}
 \tau (\hat B +\hat C - \tau) \frac{\left(x_1\right)^{k_1+1}  \left(\hat B + \hat C - x_2\right)^{k_2+1} }{(k_1+1)(k_2+1)}  \\ \times  \oint \oint (F_{x_1}(z))^{k_1+1} (\tilde
F_{x_2}(w_2))^{k_2+1}  \Cov\bigl(\tau F(z),\hat B+\hat A-(\hat B+\hat C-\tau)\tilde F(w)\bigr)\, {F'(z) \tilde F'(w)} dz dw,
\end{multline}
where the integration goes over small contours around $0$ and
$$
 F_{x_1}(z)=(z+1)\left( \frac{1}{z}+\frac{\tau }{x_1}\sum_{a=1}^\infty \frac{\c_{a;1}
z^{a-1}}{(a-1)!}\right),\quad  \tilde F_{x_2}(w)=(w+1)\left( \frac{1}{w}+\frac{\hat
B + \hat C -\tau}{\hat B + \hat C - x_2}\sum_{a=1}^\infty \frac{\c_{a;2}
w^{a-1}}{(a-1)!}\right),
$$
so that $F_{\tau}(z)=F(z)$ and $\tilde F_\tau(w)=\tilde F(w)$.

We now relate these expressions to \eqref{eq_critical_equation_rescaled},
\eqref{eq_critical_equation_rescaled_right}. Recall that $G(z)$ and $\tilde G(z)$
are defined in \eqref{eq_St_left_right}. As a corollary of the relation of
\eqref{eq_LLN_multi_formula_ext} between $F(z)$ and $G(z)$ we have (see \cite[proof
of Lemma 9.2]{BG_CLT} for more details of this computation)
$$
 F_{x_1}(z)= \frac{\tau }{x_1} \left(\ddot{z}+
\frac{1-x_1/\tau}{\exp(-G(\ddot{z}))-1} \right), \quad \ddot z=
G^{(-1)}\bigl(\log(1+z)\bigr),
$$
where $G^{(-1)}$ is the functional inverse of the function $G$. Similarly,
$$
 \tilde F_{x_2}(w)= \frac{\hat B+\hat C - \tau }{\hat B + \hat C - x_2} \left(\ddot{w}+
\frac{1-\frac{\hat B + \hat C -x_2}{\hat B + \hat C- \tau}}{\exp(-\tilde
G(\ddot{w}))-1} \right), \quad \ddot w= \tilde G^{(-1)}\bigl(\log(1+w)\bigr).
$$
Note that $G(z)$, $\tilde G(z)$ behave as $1/z$ when $z\to\infty$. Therefore, the
inverse functions map a neighborhood of $0$ to a neighborhood of $\infty$. It
follows that for small $z$ and $w$, the maps $z\mapsto \ddot z$, $w\mapsto\ddot w$
are locally bijective. Therefore, we can change the integration variables in
\eqref{eq_x19} to $\ddot{z}$, $\ddot{w}$, getting
\begin{multline}
\label{eq_x20}
 \frac{\tau (\hat B + \hat C - \tau) }{(k_1+1)(k_2+1)(2\pi\ii)^2}
\oint \oint \Cov\left(\tau \ddot{z}, \hat B+\hat A-(\hat B+\hat C-\tau)\ddot{w} \right) \\ \times \Biggl(\tau \ddot{z}+ \frac{\tau-x_1}{\exp(-G(\ddot{z}))-1} \Biggr)^{k_1+1}
\Biggl((\hat B+\hat C-\tau) \ddot{w}+ \frac{x_2-\tau}{\exp(-\tilde G(\ddot{w}))-1} \Biggr)^{k_2+1}
 d\ddot{z} d\ddot{w}.
\end{multline}
Note that $F(z)=F_\tau(z)=\ddot{z}$, $\tilde F(w)=\tilde F_\tau(w)=\ddot{w}$, and therefore,
$F'(z)$, $\tilde F'(w)$ factors in \eqref{eq_x19} got absorbed into the Jacobian of our change of
variables.

Let us do yet another change of variables, so that $\tau \ddot{z}$ and $\hat B+\hat A-\ddot{w}(\hat
B+\hat C-\tau)$ become the new variables $\ddot{z}$ and $\ddot{w}$, respectively, to get
\begin{multline}
\label{eq_x30}
 \frac{ - 1}{(k_1+1)(k_2+1)(2\pi\ii)^2}
\oint \oint \Cov\left(\ddot{z}, \ddot{w} \right) \\ \times \Biggl(\ddot{z}+ \frac{\tau-x_1}{\exp\left(-G\left(\frac{\ddot{z}}{\tau}\right)\right)-1} \Biggr)^{k_1+1}
\Biggl(\hat A+\hat B-\Biggl(\ddot{w}+ \frac{\tau-x_2}{\exp\left(-\tilde G\left(\frac{\hat B+\hat A-\ddot{w}}{\hat B+\hat C-\tau}\right)\right)-1}\Biggr) \Biggr)^{k_2+1}
 d\ddot{z} d\ddot{w}.
\end{multline}

We further deform the $\ddot{z}$ contour in \eqref{eq_x30} in such a way that its
upper half--plane part becomes the image of the vertical line (more precisely, its
part inside the liquid region) with horizontal coordinate $x_1$ under the map
$\Theta(x,y)$ of \eqref{eq_Def_theta} . The lower half--plane part is the axial
reflection of the upper half--plane part. Similarly, we deform the $\ddot{w}$
contour in such a way that its lower half--plane part is the image of the vertical
line with horizontal coordinate $x_2$ under $\Theta(x,y)$; the upper half--plane
part is the axial reflection of the lower half--plane part.

After deformation of the contours, we integrate by parts in $\ddot{z}$ and $\ddot{w}$. Taking into
account that $\Cov(\ddot{z},\ddot{w})$ is the mixed partial derivative of the first term in
$\mathcal G (z,w)$ of \eqref{eq_GFF_cov_2}, we get

\begin{multline}
\label{eq_x22} \frac{1}{\pi} \int\limits_{\ddot{z}\in\{\Theta(x_1,\cdot)\}}
\int\limits_{\ddot{w}\in\{\Theta(x_2,\cdot)\}} \Biggl(\ddot{z}+ \frac{\tau-x_1}{\exp\left(-G\left(\frac{\ddot{z}}{\tau}\right)\right)-1} \Biggr)^{k_1} \Biggl(\hat A+\hat B- \Biggl(\ddot{w}+ \frac{\tau-x_2}{\exp\left(-\tilde G\left(\frac{\hat B+\hat A-\ddot{w}}{\hat B+\hat C-\tau}\right)\right)-1}\Biggr) \Biggr)^{k_2} \\
\times \mathcal G(\ddot{z},\ddot{w}) \frac{\partial}{\partial \ddot{z}} \left[\ddot{z}+ \frac{\tau-x_1}{\exp\left(-G\left(\frac{\ddot{z}}{\tau}\right)\right)-1}  \right] \frac{\partial}{\partial
\ddot{w}}\left[\ddot{w}+ \frac{\tau-x_2}{\exp\left(-\tilde G\left(\frac{\hat B+\hat A-\ddot{w}}{\hat B+\hat C-\tau}\right)\right)-1}  \right]
 d\ddot{z} d\ddot{w}.
\end{multline}
Note that the $\ddot{z}$
contour can be parameterized by the $y$--coordinate of the vertical slice at $x=x_1$, and
$\ddot{z}+ \frac{\tau-x_1}{\exp(-G(\ddot{z}))-1}$ is equal to this coordinate. Similarly, the
$\ddot{w}$ contour can be parameterized by the $y$ coordinate of the vertical slice at $x=x_2$.
Therefore, changing the coordinates to the $y$--coordinates along these two slices, \eqref{eq_x22}
becomes \eqref{eq_target_left_right}.

\bigskip

We switch to the case $x_1\le x_2 \le \tau$. Following the same scheme as in the previous case, we
compute the numbers $d_{1,h_1; 1,h_2}$ by applying Theorem \ref{Theorem_CLT_multi_ext} to evaluate
the asymptotic covariance of $\sum_{i=1}^t (\frac{\lambda_i+t-i}{t})^k$ and $\sum_{i=1}^{t}
(\frac{\lambda_i+t-i}{ t})^m$. Multiplying the result of \eqref{eq_CLT_multi_formula_ext} by
$u^{-k-1} v^{-m-1}$ with large complex numbers $u$ and $v$, and then summing over all $k$ and $m$
and rewriting the residue computation as a contour integral  we get
\begin{equation}
\label{eq_covariance_Stieltjes_2}
 \cov\left(\frac{1}{u-x},t; \frac{1}{v-x}, t \right)= \frac{1}{(2\pi\ii)^2}
\oint\oint \frac{1}{u-F(z)} \cdot \frac{1}{v-F(w)} Q^{<\tau}(z,w) dz dw,
\end{equation}
where $\cov\left(\frac{1}{u-x},t; \frac{1}{v-x},\tilde t \right)$ means the
asymptotic covariance of the random variables
$$
 \sum_{i=1}^t
\frac{1}{u- \frac{\lambda_i+t-i}{t}}\text{ and }\sum_{i=1}^{t}
\frac{1}{v-\frac{\lambda_i+t-i}{t}},
$$
$F(z)$ is the same as in \eqref{eq_covariance_Stieltjes}, and
$$
 Q^{<\tau}(z,w)=\frac{1}{(z-w)^2}+\sum_{a,b=1}^{\infty} \frac{\d_{a,1;b,1}}{(a-1)!(b-1)!} z^{a-1} w^{b-1}.
$$
The integration goes over small contours around $0$, which enclose both $1/u$ and
$1/v$, and one of the contours is inside another one.

The non-deterministic parts of $\lambda$ coincides with random $N$--tuple $\ell_1>\dots>\ell_N$
studied in Theorem \ref{Theorem_hex_cut}. Therefore, the left-hand side of
\eqref{eq_covariance_Stieltjes} is $\tau^2 \Cov(\tau u, \tau v)$, where $\Cov$ is given by
\eqref{eq_covariance_two_cut}.

The difference with the $x_1\le \tau \le x_2$ argument is that the $1/(z-w)^2$ term
in $Q^{<\tau}$ does not allow us to compute the contour integral
\eqref{eq_covariance_Stieltjes_2} as a single residue. We adjust for that by
introducing a new function
$$
 Q^{(0)} (z,w)= \frac{1}{(z-w)^2}+\sum_{a,b=1}^{\infty} \frac{\d_{a;b}^{(0)}}{(a-1)!(b-1)!} z^{a-1}
w^{b-1},
$$
where the coefficients $d_{a;b}^{(0)}$ are chosen so that \eqref{eq_covariance_Stieltjes_2} with
$Q^{<\tau}$ replaced by $Q^{(0)}$ would give identical zero. In other words, these are the numbers
corresponding to vanishing covariance. Theorem \ref{Theorem_main} implies that there is a unique
choice for such $d_{a;b}^{(0)}$. In particular, $Q^{(0)}(z,w)$ then coincides with $Q_\rho$ of
\cite[Section 9.1]{BG_CLT}, used there in the GFF theorem for trapezoids with deterministic
boundaries. Define
$$
 \Delta Q(z,w)=Q^{<\tau}(z,w)-Q^{(0)}(z,w),
$$
and write
$$
 \Delta Q(z,w)=\sum_{a,b=1}^{\infty} \frac{\Delta \d_{a;b}^{(0)}}{(a-1)!(b-1)!}
z^{a-1} z^{a-1} w^{b-1}.
$$
The definitions imply that $\Delta Q(z,w)$ is analytic (has no singularities) in a
neighborhood of $(0,0)$. Then the identity
\begin{equation}
\label{eq_covariance_Stieltjes_3}
 \cov\left(\frac{1}{u-x},t; \frac{1}{v-x}, t \right)= \frac{1}{(2\pi\ii)^2}
\oint\oint \frac{1}{u-F(z)} \cdot \frac{1}{v-F(w)} \Delta Q(z,w) dz dw,
\end{equation}
implies an analogue of \eqref{eq_covariance_through_Q}
\begin{equation}
\label{eq_covariance_through_Q_2}
  \Delta Q(z,w)= \tau^2 \Cov(\tau F(z),\tau F(w))\, {F'(z) F'(w)}.
\end{equation}
Applying Theorem \ref{Theorem_CLT_multi_ext} and integration by parts, as in
\eqref{eq_integrate_by_parts},  we compute the asymptotic covariance of
$p^L(x_1;k_1)$ and $p^L(x_2; k_2)$ as
\begin{multline}
\label{eq_x23}
   \frac{\tau^2 \left(x_1\right)^{k_1+1}  \left(x_2\right)^{k_2+1} }{(k_1+1)(k_2+1)}   \oint \oint (F_{x_1}(z))^{k_1+1} (
F_{x_2}(w_2))^{k_2+1}\Cov(\tau F(z),\tau F(w))\, {F'(z) \tilde F'(w)} dz dw,
\\ +
\frac{\tau^2 \left(x_1\right)^{k_1+1}  \left(x_2\right)^{k_2+1} }{(k_1+1)(k_2+1)} \oint
\oint (F_{x_1}(z))^{k_1+1} ( F_{x_2}(w_2))^{k_2+1} Q^{(0)}(z,w) dz dw,
\end{multline}
We transform the first line of \eqref{eq_x23} exactly in the same way as in $x_1\le
\tau \le x_2$ case, arriving at an expression of a form similar to \eqref{eq_x22}.
Note that now both $\ddot{z}$ and $\ddot{w}$ contour are in the upper half--plane,
and therefore, there is an additional term $-\frac{1}{2\pi}
\ln\left|\frac{\ddot{z}-\ddot{w}}{\ddot{z} \bar\ddot{w}}\right|$ in the definition
of $\mathcal G (\ddot{z}, \ddot{w})$. However, we get precisely this term after
changing the variables to $\ddot{z}$, $\ddot{w}$ in the second line of
\eqref{eq_x23}, as shown in \cite[Section 9.1]{BG_CLT}. This finishes the proof for
the $x_1\le x_2\le \tau$ case. The final case $\tau \le x_1\le x_2$ is analogous.
\end{proof}

\subsection{Domino tilings of holey Aztex rectangles: proofs}

\label{Section_domino_proofs}

In this section we prove Proposition \ref{Proposition_LLN_Aztec} and Theorem
\ref{Theorem_holey_Aztec}. The proofs are parallel to the lozenge tilings case and we omit many
details. In particular, we follow the same five steps. This time we need to detail Step 1, as it
was not present in the examples of \cite[Section 9]{BGG}.

Let us \emph{fix} all particles along the $x=t$ section of the rectangle. This separates the tiling
into two --- each corresponds to its own rectangle, as in  Figure \ref{Figure_Aztec_split}. Each
domino of the original tilings thus belongs to exactly one of the rectangles, the holes do not
belong to either of them.

\begin{figure}[t]
\begin{center}
 {\scalebox{1.5}{\includegraphics{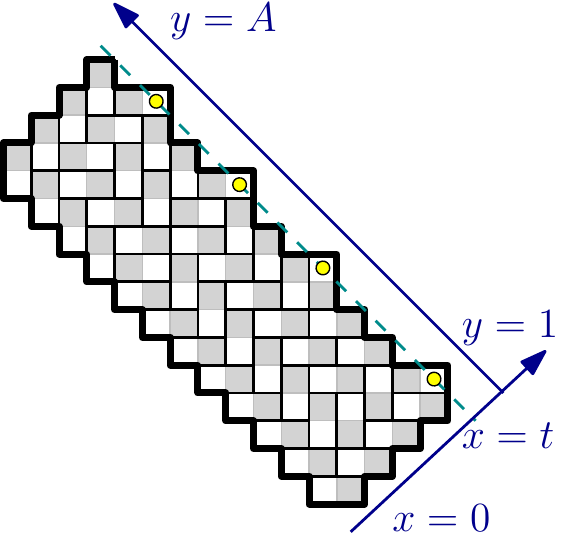}}}{\scalebox{1.5}{\includegraphics{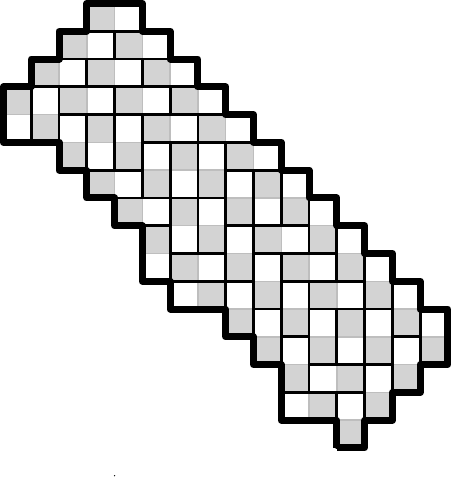}}}
 \caption{Domino tiling of Figure \ref{Figure_domino_domain} is split into two rectangles along $x=t$ line.}
 \label{Figure_Aztec_split}
\end{center}
\end{figure}

For each of the rectangles, the total number of domino tilings (given the boundary) can be
explicitly computed.

\begin{lemma} \label{Lemma_Aztec_rectangle_Z}
 Let $\ell_1<\dots<\ell_t$ denote the particle positions along the line $x=t$, encoding the left
 Aztec rectangle, as in Figure \ref{Figure_Aztec_split}. The total number of domino tilings of this
 rectangle is
 \begin{equation}
  2^{t(t+1)/2} \prod_{i<j} (\ell_j-\ell_i)
 \end{equation}
\end{lemma}
\begin{proof}
 See \cite{Ciucu} and also
 \cite[Corollary 2.14]{BK}.
\end{proof}

\begin{corollary}\label{Corollary_Aztec_law}
 Let $X=\bigcup_{k=1}^K \{a_k,a_k+1,\dots,b_k\}\subset \{1,\dots,A\}$.
 The distribution of $t$ particles $\ell_1<\dots<\ell_t$ along the $x=t$ section of uniformly
 random domino tiling is
 \begin{equation}
 \label{eq_Aztec_law}
  \frac{1}{Z} \prod_{1\le i<j\le t} (\ell_i-\ell_j)^2 \prod_{i=1}^t \left[\frac{1}{ (\ell_i-1)! (A-\ell_i)!}
  \prod_{u\in \{1,\dots,A\}\setminus X} |\ell_i-u|\right],
 \end{equation}
 where $Z>0$ is a normalization constant.
\end{corollary}
\begin{remark}
 It is not important for the computation, whether we fix the filling fractions $n_1,\dots,n_K$ or
 not; only the partition function and the set of admissible configurations $\ell_1,\dots,\ell_t$
 changes.
\end{remark}
\begin{proof}[Proof of Corollary \ref{Corollary_Aztec_law}]
 Let $\tilde \ell_1<\tilde \ell_2<\dots<\tilde \ell_{\tilde
 t}$ denote the particle configuration $X\setminus\{\ell_1,\dots,\ell_t\}$. The combinatorics of
 the model implies $\tilde t=B-t$

  Then combining Lemma \ref{Lemma_Aztec_rectangle_Z}
 with the same statement for the right rectangle, we get the formula for the distribution of
 $\ell_1,\dots,\ell_t$ of the form
 \begin{equation}
 \label{eq_x31}
  \mathrm {Prob}(\ell_1,\dots,\ell_t)=\frac{1}{Z} \prod_{1\le i<j \le t} (\ell_j-\ell_i)
  \prod_{1\le i<j \le \tilde t} (\tilde \ell_j-\tilde \ell_i).
 \end{equation}
 We next use the formula, which is valid for any two disjoint sets $\mathcal A$, $\mathcal B$ with
 $\mathcal A\cup \mathcal B=X$:
 $$
  \prod_{\begin{smallmatrix}a,a'\in \mathcal A,\\ a<a'\end{smallmatrix}}  (a'-a) \cdot \prod_{\begin{smallmatrix}x,x'\in X, \\
  x<x'\end{smallmatrix}} (x'-x)
  =  \prod_{\begin{smallmatrix}b,b'\in \mathcal B,\\ b<b'\end{smallmatrix}}(b'-b)
  \cdot \prod_{a\in \mathcal A} \prod_{\begin{smallmatrix}x\in X,\\ x\ne a\end{smallmatrix}} |x-a|.
 $$
 Applying this formula with $\mathcal A=\{\tilde \ell_1,\dots,\tilde \ell_{\tilde t}\}$, we
 transform \eqref{eq_x31} into
 \begin{equation}
 \label{eq_x31}
  \mathrm {Prob}(\ell_1,\dots,\ell_t)=\frac{1}{Z'} \prod_{1\le i<j \le t} (\ell_j-\ell_i)^2
  \prod_{i=1}^t \prod_{\begin{smallmatrix}x\in X,\\ x\ne \ell_i\end{smallmatrix}} |x-\ell_i|,
 \end{equation}
 which is another form of \eqref{eq_Aztec_law}.
\end{proof}

Let us rewrite \eqref{eq_Aztec_law} in the form
$$
  \frac{1}{Z}\prod_{1\le i<j \le t} (\ell_j-\ell_i)^2 \prod_{i=1}^t w(\ell_i).
$$
Then the weight $w(\cdot)$ satisfies
\begin{equation}
 \label{eq_weight_ration}
 \frac{w(x)}{w(x-1)}=-\prod_{k=1}^K \frac{b_k+1-x}{a_k-x},
\end{equation}
and therefore we can identify it with the weight of multi--cut general $\theta$ Krawtchouk ensemble
of \cite[Section 9.1]{BGG} at $\theta=1$. Therefore, we can apply the results of this article to
get the LLN and CLT for $\ell_1,\dots,\ell_t$.

The covariance in the CLT depends on $K$ segments $(\alpha_k,\beta_k)\subset [\hat
a_k, \hat b_k]$, $k=1,\dots,K$, which are intersections of the liquid region
$\mathcal L$ with the vertical line $x=\tau$:
\begin{multline}
\label{eq_covariance_K_cut} \Cov^K(z,w) =
 -\frac{1}{2(w-z)^2}+ \frac{\sum_{d=0}^{K-2}c_d(z)w^d}{\sqrt{\prod_{i=1}^K (w-\alpha_i)(w-\beta_i)}}
\\+\frac{{\sqrt{\prod_{i=1}^K(z-\alpha_i)(z-\beta_i)}}}{2\sqrt{\prod_{i=1}^K(w-\alpha_i)(w-\beta_i)}}\Biggl(\frac{1}{(z-w)^2}-\frac{1}{2(z-w)}\sum_{i=1}^K
\left(\frac{1}{z-\alpha_i}+\frac{1}{z-\beta_i}\right)\Biggr),
\end{multline}
where $\sum_{d=0}^{K-2}c_d(z)w^d$, is a unique polynomial of degree at most $K-2$ and with
coefficients depending on $z$, such that the integrals of \eqref{eq_covariance_K_cut} in $w$ around
the segments $(\alpha_k,\beta_k)$, $k=1,\dots,K$ vanish. For the square root $\sqrt{x}$ we choose
the branch which maps large positive real numbers to large positive real numbers everywhere in
\eqref{eq_covariance_two_cut}.

\begin{theorem}[{\cite[Theorem 7.1 and Section 9.1]{BGG}}] \label{Theorem_Aztec_cut} Define the Cauchy--Stieltjes transform of random
particles $\ell_1>\dots>\ell_N$ introduced above, through
$$
 G_L(z)=\sum_{i=1}^N \frac{1}{z- \frac{\ell_i}{L} },
$$
then as $L\to\infty$ in the limit regime \eqref{eq_hex_limit_regime}, we have
\begin{equation}
 \lim_{L\to\infty} \frac{1}{L} G_L(z) = \int_{\mathbb R} \frac{\mu(x)}{z-x} dx, \text{ in
 probability},
\end{equation}
where $\mu(x)$ is a density of a compactly--supported measure of mass $\tau$, which
satisfies $0\le \mu(x)\le 1$ everywhere; $0<\mu(x)<1$ only on K intervals
(``bands'') $(\alpha_k,\beta_k)$, $k=1,\dots,K$. Further, $G_L(z)-\E G_L(z)$,
converges as $L\to\infty$ to a Gaussian field in the sense of moments, jointly and
in uniformly for finitely many $z$'s belonging to an arbitrary compact subset of
$\mathbb C\setminus \overline{ {\mathrm supp} [\mu(x)]}$. The covariance
$\lim_{L\to\infty}\left[ \E G_L(z) G_L(w) - \E G_L(z) \E G_L(w)\right]$ is given by
$\Cov^K(z,w)$.
\end{theorem}
Note that \cite{BGG} used normalization by $N$, while we normalize by $L$, hence a slight
difference in the statements.

\bigskip

Topologically, the liquid region $\mathcal L$ of the holey Aztec rectangle is homeomorphic to the
Riemann sphere with $K$ cuts. We fix these cuts in a specific way by considering the domain
$\mathbb D=\mathbb C\setminus \bigl[ \bigcup_{k=0}^K (\beta_{k}, \alpha_{k+1})\bigr]$, where
$\alpha_k,\beta_k$ are endpoint of the bands from Theorem \ref{Theorem_hex_cut}, and we set
$\beta_0=-\infty$, $\alpha_{K+1}=+\infty$

Below we construct an explicit uniformization map between $\mathcal L$ and $\mathbb D$. Before
doing that, let us define the Green function of the Laplace operator in $\mathbb D$ with Dirichlet
boundary conditions.

 Fix an arbitrary point $D\in\mathbb C$ and define for $z$ and $w$ in the upper half--plane the
 function

 \begin{multline}
\label{eq_GFF_cov_1_Aztec}
  \mathcal G(z,w)=\frac{1}{4\pi}\Biggl[ \int_{D}^z \int_D^w \Cov^K(\mathfrak z,\mathfrak w) d\mathfrak z d\mathfrak w -
 \int_{D}^{\bar z} \int_D^w \Cov^K(\mathfrak z,\mathfrak w) d\mathfrak z d\mathfrak w
 \\ -\int_{D}^z \int_D^{\bar w} \Cov^K(\mathfrak z,\mathfrak w) d\mathfrak z d\mathfrak w
 +\int_{D}^{\bar z} \int_D^{\bar w} \Cov^K(\mathfrak z,\mathfrak w) d\mathfrak z d\mathfrak
 w\Biggr]-\frac{1}{2\pi}\ln\left|\frac{z-w}{z-\bar w}\right|.
 \end{multline}
The integration contour in \eqref{eq_GFF_cov_1} lies inside $\mathbb C\setminus \cup_{k=1}^K
 [\alpha_k,\beta_k]$ (which is different from the domain $\mathbb D$), and  $\mathcal G(z,w)$ does
 not depend on the choice of such contour, as follows from the fact that integrals of
 $\Cov^K(\mathfrak z,\mathfrak w)$ around the bands vanish.

Further,\
 if $z$ is in the upper halfplane, while $w$ is in the lower halfplane, then
 $\mathcal G(z,w)$ has a similar definition:
 \begin{multline}
\label{eq_GFF_cov_2_Aztec}
  \mathcal G(z,w)=\frac{1}{4\pi}\Biggl[ \int_{D}^z \int_D^{\bar w} \Cov^K(\mathfrak z,\mathfrak w) d\mathfrak z d\mathfrak w -
 \int_{D}^{\bar z} \int_D^{\bar w} \Cov^K(\mathfrak z,\mathfrak w) d\mathfrak z d\mathfrak w
 \\ -\int_{D}^z \int_D^{w} \Cov^K(\mathfrak z,\mathfrak w) d\mathfrak z d\mathfrak w
 +\int_{D}^{\bar z} \int_D^{w} \Cov^K(\mathfrak z,\mathfrak w) d\mathfrak z d\mathfrak
 w\Biggr].
 \end{multline}
The differences between \eqref{eq_GFF_cov_1} and \eqref{eq_GFF_cov_2} are the logarithmic term and
conjugation of $w$. We further extend the definition to all other non-real $z$ and $w$ by requiring
 \begin{equation}
 \label{eq_GFF_cov_3_Aztec}
  \mathcal G(z,w)=\mathcal G(\bar z,\bar
  w).
 \end{equation}

Repeating the proof of Proposition \ref{Proposition_G_is_Green}, we obtain the following.

\begin{proposition} \label{Proposition_G_is_Green_Aztec} $\mathcal G(z,w)$
(continuously extended to $z,w\in \mathbb D$) is the Green function of the Laplace operator in
$\mathbb D$ with Dirichlet boundary conditions.
\end{proposition}

The next step is to construct a conformal isomorphism $\Theta$ between  the liquid
region $\mathcal L$ with the complex structure given by the complex slope $\xi$ and
$\mathbb D$ with the standard complex structure. This will be done by gluing the
complex structures in the liquid regions of two Aztec rectangles.

Recall that the particles along the $x=t$ line form a border for the left rectangle
and they are encoded by a signature $\lambda\in \GT_t$. We set
\begin{equation}
\label{eq_x32}
 G(z)=\lim_{L\to\infty} \frac{1}{t}\sum_{i=1}^t \frac{1}{z-\lambda_i/t}.
\end{equation}
Similarly the configuration complementary to the particles along the $x=t$ forms a
border for the right rectangle, and it is encoded by a signature $\tilde \lambda\in
\GT_{\tilde t}$, where $\tilde t=B-t$ We set
\begin{equation}
\label{eq_x33}
 \tilde G(z)=\lim_{L\to\infty} \frac{1}{\tilde t}\sum_{i=1}^{\tilde t} \frac{1}{z-\lambda_i/{\tilde t}}.
\end{equation}
The limits (in probability) in \eqref{eq_x32}, \eqref{eq_x33} exist by Theorem
\ref{Theorem_Aztec_cut}.

For $x<\tau$ we consider an equation
\begin{equation}
\label{eq_Aztec_critical_equation_left}
 y=z+\frac{2 (\tau - x)  }{\exp\left(-G\left(\frac{z}{\tau}\right)\right)-\exp\left(G\left(\frac{z}{\tau}\right)\right)}
\end{equation}
It is shown in \cite[Proposition 6.2]{BK} that the above equation has a non-real
complex root in the upper half--plane $\mathbb U$ if and only if $(x,y)$ belongs to
the liquid region $\mathcal L^{\mathrm{left}}$ of the left rectangle. Moreover, the
resulting map $(x,y)\mapsto z(x,y)$ is a smooth bijection between $\mathcal
L^{\mathrm{left}}$ and $\mathbb U$.

Differentiating \eqref{eq_Aztec_critical_equation_left} in $x$ and $y$, we get
\begin{equation}
\label{eq_Aztec_critical_equation_rescaled_x}
 0=z_x\left(1-2\frac{\tau-x}{\tau}G'\left(\frac{z}{\tau}\right) \frac{\exp\left(-G\left(\frac{z}{\tau}\right)\right)-\exp\left(G\left(\frac{z}{\tau}\right)\right)}{\left[\exp\left(-G\left(\frac{z}{\tau}\right)\right)-\exp\left(G\left(\frac{z}{\tau}\right)\right)\right]^2}\right)
 -\frac{2 }{\exp\left(-G\left(\frac{z}{\tau}\right)\right)-\exp\left(G\left(\frac{z}{\tau}\right)\right)}
\end{equation}
\begin{equation}
\label{eq_Aztec_critical_equation_rescaled_y}
 1=z_y\left(1-2\frac{\tau-x}{\tau}G'\left(\frac{z}{\tau}\right) \frac{\exp\left(-G\left(\frac{z}{\tau}\right)\right)-\exp\left(G\left(\frac{z}{\tau}\right)\right)}{\left[\exp\left(-G\left(\frac{z}{\tau}\right)\right)-\exp\left(G\left(\frac{z}{\tau}\right)\right)\right]^2}\right),
\end{equation}
which implies the relation
\begin{equation}
\label{eq_Aztec_burgers_for_map}
 z_x=\frac{2}{\exp\left(-G\left(\frac{z}{\tau}\right)\right)-\exp\left(G\left(\frac{z}{\tau}\right)\right)}
 z_y
\end{equation}
\begin{lemma} \label{Lemma_Aztec_rectangle_slope} For $(x,y)$ in the liquid region,
 the complex slope is computed as $\xi := \exp\left(-G\left(\frac{z(x,y)}{\tau}\right)\right)$.
\end{lemma}
\begin{proof}
We need to show that
$$
\arg (\xi) = - \pi \frac{\partial h}{\partial y} = \pi (p_{north}+p_{east}), \qquad \arg \left( \frac{1 + \xi}{1 - \xi} \right) = \pi \frac{\partial h}{\partial x} = \pi ( p_{east} + p_{south}).
$$
The first equality is \cite[Theorem 4.3]{BK}. To establish the second equality, let us start by
noting that \eqref{eq_Aztec_burgers_for_map} implies
$$
\xi_x = \frac{2}{ \xi -  \xi^{-1}}  \xi_y,
$$
which is the complex Burgers equation \eqref{eq_Burgers_domino_2}. Next, we have
\begin{multline*}
\frac{ \partial h(x,y)}{\partial x} = - \frac{\partial}{ \partial x} \int_{y}^{+\infty} \frac{\partial}{ \partial v} h(x,v) d v =  \frac{\partial}{ \partial x} \int_y^{+\infty} \pi^{-1} \arg ( \xi) d v = \frac{1}{\pi} \frac{\partial}{ \partial x} \Im \int_y^{+\infty} \log ( \xi) d v \\ = \frac{1}{\pi} \Im \int_y^{+\infty} \frac{1}{ \xi} \frac{\partial  \xi}{ \partial x} dv = \frac{1}{\pi} \Im \int_y^{+\infty} \frac{2}{ \xi^2 - 1} \frac{\partial  \xi}{ \partial v} dv = \frac{1}{\pi} \Im \int_y^{+\infty} \left( \frac{1}{ \xi-1} - \frac{1}{ \xi + 1} \right) \frac{\partial  \xi}{ \partial v} dv \\ = \frac{1}{\pi} \left( \pi - \Im \left( \log \frac{ \xi -1}{ \xi +1} \right) \right) = 1 - \frac{1}{\pi} \arg \left( \frac{  \xi -1}{ \xi+1} \right) = \frac{1}{\pi} \arg \left( \frac{1 +  \xi}{1 -  \xi} \right).\qedhere
\end{multline*}
\end{proof}

Given Lemma \ref{Lemma_Aztec_rectangle_slope}, we compare
\eqref{eq_Aztec_burgers_for_map} with \eqref{eq_Burgers_domino_2} and conclude that
the map $(x,y)\mapsto z(x,y)$ is conformal in Kenyon--Okounkov's complex structure
on the left trapezoid.

For the right trapezoid, i.e.\ $x>\tau$, we consider an equation
\begin{equation}
\label{eq_Aztec_critical_equation_right}
 y=\tilde z+\frac{2 (\tau -x)  }{\exp\left(-\tilde G\left(\frac{\hat A-\tilde z}{\hat B-\tau}\right)\right)-\exp\left(G\left(\frac{\hat B-\tilde z}{\hat
 B-\tau}\right)\right)},
\end{equation}
which is the same equation \eqref{eq_Aztec_critical_equation_right}, written for the
right trapezoid (with its own rotated coordinate system) with $\tau$ replaced by
$\hat B-\tau$, and then rewritten back into the original coordinates by replacing
$y$ with $\hat A-y$, $z$ with $\hat A-\tilde z$, and $x$ with $\hat B - x$.

The results for \eqref{eq_Aztec_critical_equation_left} imply that
\eqref{eq_Aztec_critical_equation_right} has a unique complex root in the
\emph{lower} halfplane  $\bar {\mathbb U}$ when $(x,y)$ belongs to the liquid region
$\mathcal L^{\mathrm{right}}$ of the right trapezoid. Then the map $(x,y)\mapsto
\tilde z(x,y)$ is a conformal bijection between $\mathcal L^{\mathrm{right}}$ and
$\bar {\mathbb U}$.

We further glue the definitions in two trapezoids together and define the
uniformization map $\Theta: \mathcal L \to \mathbb D$ through

\begin{equation}
\label{eq_def_theta_Aztec}
 \Theta(x,y)=\begin{cases} z(x,y),& x<\tau,\\ \tilde z(x,y),& x>\tau,\\ y,& x=\tau,
 y\in \bigcup_{k=1}^K (\alpha_k,\beta_k). \end{cases}
\end{equation}

The definitions readily imply that right and left parts glue together into $\Theta$
in a continuous way, and therefore, $\Theta$ is a conformal bijection between
$\mathcal L$ and $\mathbb D$. This leads to a detailed version of Theorem
\ref{Theorem_holey_Aztec}.
\begin{proposition}
\label{Proposition_holey_Aztec_detailed}
 Let $\mathrm{Aztec}$ be the interior of the holey rectangular Aztec rectangle
 in the coordinate system of Figure \ref{Figure_domino_domain}, and
 let $H_L(x,y)$ be the random height function
of uniformly random domino tiling of $\mathrm{Hex}$ in the limit regime
\eqref{eq_Aztec_limit_regime_1}, \eqref{eq_Aztec_limit_regime_2}. Further, let
$\mathcal L$ be the liquid region, let $\mathbb D=\mathbb C\setminus[
\bigcup_{k=0}^K(\beta_{k},\alpha_{k+1})]$, let $\Theta: \mathcal L \to \mathbb D$ be
the uniformization map \eqref{eq_def_theta_Aztec}, and let $\mathcal G^K$ be the
Green function of $\mathbb D$, as in \eqref{eq_GFF_cov_1_Aztec},
\eqref{eq_GFF_cov_2_Aztec}, \eqref{eq_GFF_cov_3_Aztec}.

 For any $m=1,2,\dots$ take a collection of $m$ points $0< x_1,x_2,\dots,x_m< \hat B
$ and $m$ polynomials $f_1,\dots,f_m$. Then the random variables
\begin{equation}
\label{eq_polynomial_observables_3}
 \int_{y:\, (x_i,y)\in \mathrm{Hex}} f_i(y) \bigl( H_L(x_i,y)-\E H_L(x_i,y) )
dy,\quad i=1,\dots,m,
\end{equation}
converge as $L\to\infty$ (in the sense of moments) to a centered Gaussian
$m$--dimensional random vector. The asymptotic covariance of $i$th and $j$th
components is
\begin{equation}
\label{eq_target_covariance_2}
 \frac{1}{\pi} \int_{y_i:\, (x_i,y_i)\in \mathcal L} \int_{y_j:\, (x_j,y_j)\in \mathcal L}
f_i(y_i) f_j(y_j) \mathcal G^K\bigl( \Theta(x_i,y_i), \Theta(x_j,y_j) \bigr)    dy_i
d y_j, \quad 1\le i,j \le m.
\end{equation}
\end{proposition}
\begin{proof}
 The proof repeats that of Proposition \ref{Proposition_holey_hex_detailed}. Let us only list
 substantial differences.
  The formula \eqref{eq_lozenge_through_Schur} is replaced by
 \begin{equation}
\label{eq_domino_through_Schur}
 \frac{s_\lambda(x_1,\dots,x_q,1^{t-q})}{s_\lambda(1^t)} \prod_{i=1}^q \left(\frac{1+x_i}{2}\right)^{t-q}=\sum_{\mu\in\GT_q}
 \mathrm{Prob}(\mu\mid \lambda) \frac{s_\mu(x_1,\dots,x_q)}{s_\mu(1^q)},
\end{equation}
whose connection to domino tilings is explained in details in \cite[Section 9.3]{BG_CLT},
\cite{BouCC}, \cite[Section 2.2]{BK}. Therefore, in the domino analogue of
\eqref{eq_specifications_lozenge}, the functions $g$ are now appropriate powers of $\prod_{i=1}^q
\left(\frac{1+x_i}{2}\right)$. Finally, our reference to \cite[Section 9.1]{BG_CLT} for lozenges is
replaced by the \cite[Section 6]{BK} for the dominos.
\end{proof}

\section{Appendix: CLT for quantum random walk}

\label{Section_quantum_CLT}

Perhaps, the main idea of this paper is to show that Schur generating functions describe asymptotic
behavior of probability measures on the real line similarly to how the classical characteristic
function describes the asymptotic behavior of a random variable. Theorem \ref{Theorem_main} is an
analog of the Fourier transform in our setting, while Theorem \ref{Theorem_main_multi} is an analog
of the Fourier transform for random finite-dimensional vectors. In this section we comment on one
of analogs of another classical statement: The convergence of a random walk to the standard
Brownian motion, see Proposition \ref{prop:app-main}.

\subsection{Preliminaries} We will need the following asymptotics.

\begin{lemma}
\label{lem:app}
Let $\la$ be a fixed Young diagram. We have
\begin{equation*}
\left. N \pa_1^k \log \frac{s_{\la} (x_1, \dots, x_N)}{ s_{\la} (1^N)} \right|_{x_i=1} \xrightarrow[N \to \infty]{} |\la|  \mathbf{1}_{k=1}
\end{equation*}
\begin{equation*}
\left. N^2 \pa_1^{k_1} \pa_2^{k_2} \log \frac{s_{\la} (x_1, \dots, x_N)}{ s_{\la} (1^N)} \right|_{x_i=1} \xrightarrow[N \to \infty]{} - |\la| \mathbf{1}_{k_1=1} \mathbf{1}_{k_2=1},
\end{equation*}
\begin{equation*}
\left. N^2 \pa_1^{k_1} \pa_2^{k_2} \pa_3^{k_3} \dots \pa_r^{k_r} \log \frac{s_{\la} (x_1, \dots, x_N)}{ s_{\la} (1^N)} \right|_{x_i=1} \xrightarrow[N \to \infty]{} 0, \qquad r \ge 3.
\end{equation*}
for any $k_i \in \mathbb{Z}_{\ge 1}$.
\end{lemma}
\begin{proof}
Recall that Newton power sums are defined via $p_k (x_1, \dots, x_m) := x_1^k + \dots + x_m^k$ and
$$
p_{\la_1, \la_2, \dots} (x_1, \dots, x_m) = p_{\la_1} (x_1,\dots, x_m) p_{\la_2} (x_1, \dots, x_m) \dots,
$$
for a partition $\la = \la_1 \ge \la_2 \ge \dots$. Let us denote by $\ell(\la)$ the
length (number of non-zero coordinates) of a partition $\la$, and by $|\la|$ the
number of boxes in $\la$. The Schur function $s_{\la}$ can be decomposed into a
homogeneous linear combination of products of Newton power sums, and this
decomposition is unique, when the number of variables is larger than $|\lambda|$:
$$
s_{\la} = \sum_{\mu:\, |\mu|=|\lambda|} b^{\mu}_\lambda p_{\mu} = b^{1^{|\lambda|}}_\lambda p_1^{|\lambda|} +
\sum_{\mu:\, \ell(\mu) < |\mu|=|\lambda|} b^{\mu}_{\lambda} p_{\mu}.
$$
It is well-known that $b^{1^{|\lambda|}}_\lambda$ is strictly positive, cf.\
\cite[Chapter I]{Mac}.

Note that $p_{\rho} (1^N) = N^{\ell(\rho)}$, and that for fixed $|\rho|$ the largest $\ell(\rho)$
is achieved on partition $1^{|\rho|}$. It will be convenient to scale power sums:
$$
\hat p_{\rho} (x_1, \dots, x_N) := \frac{1}{N^{\ell(\rho)}} p_{\rho} \left( x_1, \dots, x_N \right), \qquad \hat p_{\rho} (1^N)=1.
$$
As $N \to \infty$, we have
\begin{multline*}
\frac{ s_{\la} (x_1, \dots, x_N)}{ s_{\la} (1^N)} = \frac{1}{s_{\la} (1^N)}
\left( b^{1^{|\lambda|}}_\lambda p_1^{|\la|} + \sum_{|\mu|=|\lambda|,\, \ell(\mu)<|\lambda|} b^{\mu}_\lambda p_{\mu}  \right) \\=
\bigl(1+o(1)\bigr) \left(\hat p_1^{|\la|} +  \sum\limits_{|\mu|=|\lambda|,\, \ell(\mu)<|\lambda|} \frac{b^{\mu}_{\lambda}}{b^{1^{|\lambda|}_\lambda}}
 N^{\ell(\mu)-|\lambda|} \hat p_{\mu}\right)
\end{multline*}
where $1+o(1)$ arose from the approximation $s_{\la} (1^N)\approx
b^{1^{|\lambda|}}_\lambda p_1^{|\la|}(1^N)=b^{1^{|\lambda|}}_\lambda$ and therefore
does not depend on $x_i$.  Note that the number of terms in the sum does not depend
on $N>|\lambda|$. Thus,
\begin{equation}
\label{eq:appp-eq}
\log \frac{ s_{\la} (x_1, \dots, x_N)}{ s_{\la} (1^N)} =
\log \left( \hat p_1^{|\la|} \right) +
\log \left( 1 + \sum_{|\mu|=|\lambda|,\, \ell(\mu)<|\lambda|} \frac{b^{\mu}_{\lambda} }
{b^{1^{|\lambda|}}_{\lambda}}  N^{\ell(\mu)-|\lambda|} \frac{\hat p_{\mu}}{\hat p_1^{|\lambda|}} \right)+C,
\end{equation}
where $C$ is $x_i$--independent (and goes to $0$ as $N\to\infty$). For $(x_1,\dots,x_N)$ in a small
neighborhood of $1^N$, the second term of \eqref{eq:appp-eq} can be expanded in the series
$$
 \sum_{p=1}^{\infty} \frac{(-1)^{p-1}}{p} \left(\sum_{|\mu|=|\lambda|,\, \ell(\mu)<|\lambda|} \frac{b^{\mu}_{\lambda} }
{b^{1^{|\lambda|}}_{\lambda}}  N^{\ell(\mu)-|\lambda|} \frac{\hat p_{\mu}}{\hat p_1^{|\lambda|}}\right)^{p}.
$$
The terms with $p>2$ have at least $N^{-3}$ prefactor, which survives through the
computation of derivatives and plugging in $x_i=1$, and therefore they do not
contribute to the desired asymptotic. The term with $p=2$ has at least $N^{-2}$
prefactor, but when we differentiate at least once in one of the variables, another
$N^{-1}$ arises, because $\partial_1 \frac{x_1^k+x_2^k+\dots x_N^k}{N}= \frac{k
x_1^{k-1}}{N}$. Again we conclude that this term does not contribute to the desired
asymptotic. In $p=1$ term, if we differentiate once, then the result has $N^{-2}$
prefactor, and therefore does not contribute to the first asymptotic of Lemma
\ref{lem:app}. If we differentiate at least in two variables, then the result has at
least $N^{-3}$ prefactor ($N^{-1}$ was there originally and each differentiation
produces another $N^{-1}$), and again there is no contribution to the desired
asymptotic.

We conclude that the second term in \eqref{eq:appp-eq} does not contribute to desired asymptotics;
the limit behavior is determined by the first term. For this term, we have
$$
\pa_{1} \hat p_1 (x_1, \dots, x_N) = \frac{1}{N},\quad
\pa_{1}^{k_1} \pa_2^{k_2} \dots \pa_r^{k_r} \hat p_1 (x_1, \dots, x_N)=0, \quad k_1+\dots +k_r>1.
$$
And therefore,
\begin{align*}
\pa_{1} \ln \hat  p_1^{|\lambda|} (x_1, \dots, x_N) \bigr|_{x_i=1} &=
\frac{|\la|}{N}+o\left(N^{-1}\right),\\ \pa_{1}^k \ln \hat  p_1^{|\lambda|} (x_1,
\dots, x_N) \bigr|_{x_i=1} &= o\left(N^{-1}\right), \quad k>1,
\\
 \pa_{1} \pa_2 \ln \hat p_1^{|\lambda|} (x_1, \dots, x_N) \bigr|_{x_i=1} &=
\frac{-|\la|}{N^2} + o \left( N^{-2} \right),\\
 \pa_{1}^{k_1} \pa_2^{k_2} \dots \pa_r^{k_r} \ln \hat p_1^{|\lambda|} (x_1, \dots, x_N) \bigr|_{x_i =1}
 &=
 o \left( N^{-2} \right), \quad r \ge 3. \qedhere
\end{align*}
\end{proof}

\subsection{Asymptotics of quantum random walk}

A classical random walk converges to a Brownian motion when a single step becomes small while one
makes a growing number of such steps. Note that the exact distribution of a small step does not
significantly affect the limit. We consider the following model in order to see a similar
phenomenon in our framework.\footnote{This construction comes from the restriction of a quantum
random walk on $U(N)$ (in the spirit of Biane \cite{Bi3}, \cite{Bi4}) to the center of the
universal enveloping algebra of $U(N)$; we do not discuss details of this connection.}.

For the  step, we fix a partition (or Young diagram) $\nu=(\nu_1\ge \nu_2 \ge \dots
\ge 0)$. Let $n>0$ be such that $\nu_{n+1}=0$. Then for each $N\ge n$ we can
identify $\nu$ with a signature of length $N$ and consider the irreducible
representation $T^N_\nu$ of the unitary group $U(N)$ with highest weight $\nu$.
Formally, $T_\nu^N$ depend on $N$, as they are representations of different groups,
however, the dependence is very mild. This can be formalized in two ways. The
character of $T_\nu^N$ as a function of eigenvalues $(x_1,\dots,x_N)$ is equal to
$s_{\nu} (x_1, \dots, x_N)$, and so for each $N$ this is the same symmetric
polynomial with different number of variables. From a different point of view, we
have a canonical $\nu$--independent way to construct $T_\nu^N$ given $T_\nu^k$. Fix
a vector $v$ in the space of $T_\nu^n$ and consider the corresponding matrix element
of the representation: $u\mapsto (T^n_\nu(u) v,v)$. This a function of $u\in U(n)$
which can be analytically continued from $U(n)$ to the semigroup of $n\times n$
complex matrices $\mathrm {Mat}_{n\times n}$ --- we use here $\nu_i\ge 0$,
$i=1,\dots,k$, which guarantees the polynomiality of the matrix element. We can now
compose this matrix element with the canonical projection $U(N)\mapsto \mathrm
{Mat}_{n\times n}$ cutting out the top--left $n\times n$ corner of the matrix. The
result turns out to be a matrix element of $T^N_\nu$, which uniquely defines this
representation.

Due to the above reasons we do not distinguish $T_\nu^N$ for $N=n,n+1,\dots$ and simply write
$T_\nu$. By $T_\nu^{\otimes k}$ we denote the $k$-fold Kronecker tensor product $T_\nu \otimes
T_\nu \otimes \dots \otimes T_\nu$. Its character is equal to $s_{\nu} (x_1, \dots, x_N)^k$.

Let $I_N$ be a (possibly reducible) finite--dimensional representation of $U(N)$. It will play the
role of an initial distribution of a quantum random walk.

For any $k \ge 0$ let us denote by $\mathcal{L}_{\la}^k$ the isotypical component
corresponding to a signature $\la$ in $\mathbf{V}^k := I_N \otimes T_\nu^{\otimes
k}$. We have a canonical orthogonal decomposition $\mathbf{V}^k = \oplus_{\la}
\mathcal{L}_{\la}^k$. Let $\mathrm{proj}_{\la}^k$ be the orthogonal projection to
$\mathcal{L}_{\la}^k$ in $\mathbf{V}^k$. For a collection of signatures $(\la^{(0)},
\la^{(1)}, \dots, \la^{(s)}, \dots)\in \prod_{i=0}^{\infty} \GT_N$ let us define
inductively $\mathcal S_0 := \mathcal{L}_{\la^{(0)}}^0$,
$$
\mathcal S_i := \mathrm{proj}_{\la^{(i)}}^i \left( \mathrm{span}_{v,u} \left( v \otimes u \mid
 v \in \mathcal S_{i-1}, u \in T_\nu
 \right) \right)\subset \mathbf{V}^i, \qquad i \ge 1.
$$

Define a probability measure $\rho_{T;N}$ on $\prod_{i=0}^{\infty} \GT_N$ via its values on cylinders
$$
\rho_{T_\nu;N} ( \la^{(0)}, \dots, \la^{(k)} ) = \frac{ \dim \mathcal S_k}{ \dim \mathbf{V}^k};
$$
the consistency of this definition for various $k$ readily follows from the construction.


Finally, for $r>0$ let $p_{k;N}^{(r)}$ be the $k$th Newton power sum of a random
signature $\la^{(\lfloor r \rfloor)}$ distributed according to measure
$\rho_{T_\nu;N}$.

\begin{proposition}
\label{prop:app-main}

Assume that the measures of the initial condition $\rho_{T_\nu;0} (\la)$ satisfy the
CLT in the sense of Definition \ref{Def_CLT}. Let $s \in \N$ and let $t_1 < t_2 <
\dots < t_s$, $t_i \in \R$, be positive reals. The vector $\left\{ p_{k;N}^{( t_i
N^2 )} - \mathbf{E} p_{k;N}^{( t_i N^2 )} \right\}_{i=1}^s$ converges to a jointly
Gaussian random vector with limit covariance

\begin{equation}
\begin{split}
\label{eq:app-cov} \lim_{N \to \infty} \frac{ \cov \left( p_{k_i;N}^{( t_i N^2 )},
p_{k_j, N}^{(t_j N^2 )} \right) }{N^{k_1+k_2}} &= [z^{-1} w^{-1}] \Biggl(
\left(\sum_{a=1}^{\infty} \frac{\mathbf{d}_{a,b} z^{a-1} w^{b-1}}{(a-1)! (b-1)!} -
t_i |\nu| + \frac{1}{(z-w)^2} \right)
\\ &\times
\left( \frac{1}{z} + 1+ (1+z) t_i |\nu| + (1+z) \sum_{a=1}^{\infty}\frac{\mathbf{c}_a z^{a-1}}{(a-1)!} \right)^{k_i+1}
 \\ &\times \left( \frac{1}{w} + 1+ (1+w) t_i |\nu| + (1+w) \sum_{a=1}^{\infty} \frac{\mathbf{c}_a w^{a-1}}{(a-1)!} \right)^{k_j+1}
  \Biggr),
\end{split}
\end{equation}
for $1 \le i<j \le s$; the functions $\sum_{a=1}^{\infty} \frac{\mathbf{c}_a z^{a-1}}{(a-1)!}$ and
$\sum_{a=1}^{\infty} \frac{\mathbf{d}_{a,b} z^{a-1} w^{b-1}}{(a-1)! (b-1)!}$ are uniquely
determined by $\rho_{T;0} (\la)$ via formulas given in Lemma \ref{Lemma_as_invert}.
\end{proposition}

\begin{proof}

Let us use the following construction (cf.\ \cite[Section 2.4]{BG_CLT}). Introduce coefficients
$\mathfrak{p}_{\nu;N} (\la \to \mu)$ via a decomposition in a linear basis:
$$
\frac{s_{\la} (x_1, \dots, x_N)}{s_{\la} (1^N)} \frac{s_\nu (x_1, \dots, x_N)}{s_\nu (1^N)} =
\sum_{\mu \in \GT_N} \mathfrak{p}_{\chi;N} (\la \to \mu) \frac{s_{\mu} (x_1, \dots, x_N)}{ s_{\mu} (1^N)}.
$$

It is a direct check that the measure $\rho_{T_\nu;N}$ can be written as
\begin{equation}
\label{eq:random-walk1}
\rho_{T_\nu;N} ( \la^{(0)}, \dots, \la^{(s)} ) := \rho_{T_\nu;0} (\la^{(0)})
\prod_{i=0}^{s-1} \mathfrak{p}_{\nu ;N} (\la^{(i)} \to \la^{(i+1)}).
\end{equation}

After such an identification, the statement of Proposition is an immediate corollary of Theorem \ref{Theorem_main} of this paper, \cite[Theorem 2.10]{BG_CLT}, and Lemma \ref{lem:app}.
\end{proof}

\begin{remark}
Note that the right-hand side of \eqref{eq:app-cov} depends on $T_\nu$ only through the number
$|\nu|$, which can be eliminated by rescaling of time. Thus, our limit object is universal and can
be viewed as an analog of Brownian motion in this setting. In particular, a deterministic initial
condition $\rho_{T;0} (\varnothing)=1$, which corresponds to $\mathbf{c}_a=0$,
$\mathbf{d}_{a,b}=0$, for all $a$, $b$, is analogous to the Brownian motion starting at 0.
\end{remark}

\begin{remark}
This result is closely related to results about \textit{Schur-Weyl} measure, obtained in \cite{Biane2} and \cite{Mel}. Namely, one can recover results of \cite{Biane2} and \cite{Mel} by considering the case when $\la$ is a one box Young diagram, $s=1$ in the proposition (only one random signature is in play), and $\rho_{T;0} (\varnothing)=1$ (deterministic initial condition).
\end{remark}

\end{document}